\documentclass[a4paper,10pt,reqno]{amsart}
          \usepackage{amssymb}
	  \usepackage{amsmath}
          \usepackage{amsfonts}
          \usepackage[english]{babel}
          \usepackage[utf8]{inputenc}

\usepackage[margin=2.5cm]{geometry}

 \usepackage[unicode,colorlinks,plainpages=false,hyperindex=true,bookmarksnumbered=true,bookmarksopen=false,pdfpagelabels]{hyperref}
 \hypersetup{urlcolor=cyan,linkcolor=blue,citecolor=red,colorlinks=true} 
\usepackage{tikz} 
\usetikzlibrary{arrows,shapes}  
\usepackage{mathdots}
  
\makeatletter
\renewcommand*{\p@section}{\,}
\renewcommand*{\p@subsection}{\S\,}
\renewcommand*{\p@subsubsection}{\S\,}
\makeatother

\newcounter{thmcounter}
\numberwithin{thmcounter}{section}
\newtheorem{thm}[thmcounter]{Theorem}
\newtheorem{cor}[thmcounter]{Corollary}
\newtheorem{lem}[thmcounter]{Lemma}
\newtheorem{prop}[thmcounter]{Proposition}
\theoremstyle{definition}
\newtheorem{exmp}[thmcounter]{Example}
\newtheorem{rem}[thmcounter]{Remark}

\numberwithin{equation}{section}

\newcommand{\CC}{\ensuremath{\mathbb{C}}}
\newcommand{\N}{\ensuremath{\mathbb{N}}}

\newcommand{\Z}{\ensuremath{\mathbb{Z}}}
\newcommand{\kk}{\ensuremath{\Bbbk}}
 
\newcommand{\tr}{\operatorname{tr}}
\newcommand{\diag}{\operatorname{diag}}
\newcommand{\Der}{\operatorname{Der}}
\newcommand{\Mat}{\operatorname{Mat}}
\newcommand{\Id}{\operatorname{Id}}
\newcommand{\id}{\operatorname{id}}
\newcommand{\Gl}{\operatorname{GL}}
\newcommand{\Sl}{\operatorname{SL}}
\newcommand{\Rep}{\operatorname{Rep}}

\newcommand{\OO}{\ensuremath{\mathcal{O}}}
\newcommand{\X}{\ensuremath{\mathcal{X}}}
\newcommand{\del}{\ensuremath{\partial}}
\newcommand{\Aut}{\operatorname{Aut}}
\newcommand{\HAut}{\operatorname{HAut}}
\newcommand{\qHAut}{\operatorname{qHAut}}
\newcommand\dgal[1]{  \left\{\!\!\left\{#1\right\}\!\!\right\} }
\newcommand\dgalU[1]{  \left\{\!\!\left\{#1\right\}\!\!\right\}_{1} }
\newcommand\dgalD[1]{  \left\{\!\!\left\{#1\right\}\!\!\right\}_{2} }
\newcommand\dgalfusU[1]{  \left\{\!\!\left\{#1\right\}\!\!\right\}_{fus}^{2\to 1} }
\newcommand\dgalfusD[1]{  \left\{\!\!\left\{#1\right\}\!\!\right\}_{fus}^{1\to 2} }
\newcommand\dgalindU[1]{  \left\{\!\!\left\{#1\right\}\!\!\right\}_{ind}^{2\to 1} }
\newcommand\dgalindD[1]{  \left\{\!\!\left\{#1\right\}\!\!\right\}_{ind}^{1\to 2} }

\newcommand\br[1]{\{ #1 \}} 
 
\newcommand\dhh[1]{\langle #1 \rangle}
\newcommand\dhu[1]{\langle #1 \rangle_1}
\newcommand\dhd[1]{\langle #1 \rangle_2}
\newcommand{\he}{\hat{e}}
\newcommand{\cF}{\ensuremath{\mathcal{F}}}
\newcommand{\Cn}{\ensuremath{\mathcal{C}_n}}
\newcommand{\tCn}{\ensuremath{\tilde{\mathcal{C}}_n}}
\newcommand{\hCn}{\ensuremath{\widehat{\mathcal{C}}_n}}
\newcommand{\cCn}{\ensuremath{\check{\mathcal{C}}_n}}
\newcommand{\CnTwo}{\ensuremath{\mathcal{C}_{n,2}}}

\newcommand{\Fd}{\operatorname{F}_{2}}
\newcommand{\Ft}{\operatorname{F}_{2,S}}
\newcommand{\Lt}{\operatorname{L}_2}
\newcommand{\Cas}{\operatorname{Cas}}
\newcommand{\Aalg}{\ensuremath{\mathtt{A}_1}}
\newcommand{\Balg}{\ensuremath{\mathtt{B}_1}}
\newcommand{\Calg}{\ensuremath{\mathtt{C}_1}}
\newcommand{\ZSn}{\ensuremath{\Z_m\!\wr\! S_n}}
\newcommand{\rc}{\ensuremath{\mathrm{c}}}
\newcommand{\pP}{\ensuremath{(\text{\textbf{P}})}}

\begin{document}

\title[Morphisms of double (quasi-)Poisson algebras and action-angle duality]{Morphisms of double (quasi-)Poisson algebras and\\action-angle duality of integrable systems}   

\author{Maxime Fairon}
 \address[Maxime Fairon]{School of Mathematics and Statistics\\ University of Glasgow, University Place\\ Glasgow G12 8QQ, UK}
 \email{Maxime.Fairon@glasgow.ac.uk}

   \begin{abstract}
Double (quasi-)Poisson algebras were introduced by Van den Bergh as non-commutative analogues of algebras endowed with a (quasi-)Poisson bracket. In this work, we provide a study of morphisms of double (quasi-)Poisson algebras, which we relate to the $H_0$-Poisson structures of Crawley-Boevey. 
We prove in particular that the double (quasi-)Poisson algebra structure defined by Van den Bergh for an arbitrary quiver only depends upon the quiver seen as an undirected graph, up to isomorphism. 
We derive from our results a representation theoretic description of action-angle duality for several classical integrable systems. 
   \end{abstract}

\maketitle

 \setcounter{tocdepth}{2}

\section{Introduction}

In the seminal paper \cite{VdB1}, Van den Bergh laid the foundation of a noncommutative theory of Poisson geometry based on double brackets. Starting with an algebra $A$ over a field of characteristic zero $\kk$, he introduced the concept of a double bracket as an operation $A\times A\to A\otimes A$ satisfying noncommutative rules of derivation and antisymmetry; this is done in such a way that the corresponding affine scheme of representations $\Rep(A,n)$ carries an antisymmetric biderivation for each $n\geq 1$. 
An interesting class of those structures are double Poisson brackets, which satisfy a version of Jacobi identity valued in $A^{\otimes 3}$ so that $\Rep(A,n)$ is now endowed with a Poisson bracket, in agreement with the `non-commutative principle' of Kontsevich and Rosenberg \cite{Ko,KR}. In fact, pushing this principle even further, Van den Bergh pointed out that the existence of a distinguished element
on $A$ permits to understand the process of Hamiltonian reduction of $\Rep(A,n)$ with respect to the natural action of $\Gl_n(\kk)$ directly at the level of the algebra $A$. 
Due to this representation theoretic perspective, it seems interesting to examine the algebras endowed with a double Poisson bracket, which are called \emph{double Poisson algebras}. They have been the object of several studies \cite{B,ORS,ORS2,PV,P16,VdW}, and our aim is to explore morphisms between double Poisson algebras. Noting that Van den Bergh also introduced the analogous notion of double quasi-Poisson brackets \cite{VdB1,VdB2}, it is natural to extend our investigation to morphisms between the corresponding algebras, called \emph{double quasi-Poisson algebras}, which currently attract attention \cite{Arthesis,CF,CF2,Fai,F19,Fthesis,FH,MT14}.

A rich family of double (quasi-)Poisson algebras is associated with quivers \cite{VdB1} and encodes the Poisson geometry of quiver varieties \cite{Nak} as well as their multiplicative analogues \cite{CBShaw}. These varieties possess ramifications in numerous branches of mathematics (see e.g. \cite{Gi12,Ki,Schi} for quiver varieties), in particular in the field of integrable systems. Indeed, in the complex setting it is known since 1998 and the work of Wilson \cite{W} that the phase space of the Calogero-Moser system can be realised as a quiver variety. Moreover, similar constructions of phase spaces as quiver varieties have been obtained for other related systems \cite{BP,CS,T15}, and they have been extended to Ruijsenaars-Schneider systems using multiplicative quiver varieties \cite{CF,CF2,Fai,Fthesis}. 
Some of these papers used the formalism of double brackets to grasp features of integrability directly on the path algebra of the relevant quivers. Thus, it raises the question of determining if additional properties of these integrable systems can be realised at the level of the quivers. This issue constitutes the main motivation behind our work, as we want to present a novel interpretation of action-angle duality, which we explain now.

Let $M$ and $M'$ be two manifolds of dimension $2n$ endowed with non-degenerate Poisson brackets such that the $n$ functions $H,H'$ on $M,M'$ define (Liouville) integrable systems, i.e. $H$ and $H'$ form sets of $n$ functionally independent Poisson-commuting elements. Let us furthermore assume that on a dense open subset of each manifold, there exist canonical Darboux coordinates $(q,p)$ or $(q',p')$, and that there exists a Poisson diffeomorphism $\Psi:M\to M'$ such that (after restriction to dense subspaces) $H'\circ \Psi$ only depends on the coordinates $q$, while $H\circ \Psi^{-1}$ only depends on the coordinates $q'$.  
Due to the assumptions, we get action-angle variables\footnote{By action-angle variables, we mean the existence of a set of $2n$ Darboux coordinates on a dense subspace of the manifold such that the integrable system only depends on the first $n$ coordinates, the \emph{action variables}. This ensures that the Hamiltonian flows are linearised with respect to the last $n$ coordinates, the \emph{angle variables}.}  as follows. The coordinates $q'$ become the action coordinates of $H$ while the coordinates $p'$ become the angle coordinates of $H$, and the same is true for $q,p$ and $H'$. We thus say that the pairs $(M,H)$ and $(M',H')$ are \emph{action-angle duals}. This construction has been largely investigated for systems of Calogero-Moser and Ruijsenaars-Schneider type following the pioneering work of Ruijsenaars \cite{R88}.  
A widespread method used to unearth action-angle duality is the existence of two different slices in a suitable orbit space defined by Hamiltonian reduction, such that each slice provides one of the two sets of Darboux coordinates. This approach is inspired by the work of Kazhdan, Kostant and Sternberg \cite{KKS78}, and it has been considered both in the complex and the real settings, see 
\cite{FM2,FGNR,Pu12,Res2} and references therein. 

One of the aims of this work is to obtain a different point of view on duality where representation theoretic considerations yield the action-angle map. Indeed, in view of the previous paragraphs, many classical integrable systems for which action-angle duality is known are defined on (multiplicative) quiver varieties. This leads us to the following  natural question\footnote{This formulation is close to the original question posed to the author by V. Rubtsov at the conference \textit{Geometric aspects of momentum maps and integrability} in Ascona, April 2018.} :
\begin{center}
 \textit{Is it possible to understand action-angle duality in terms of relevant quivers?}
\end{center}
We end this work by answering this question positively in Section \ref{S:dual}, as we note that action-angle duality can be realised at the level of quivers simply as a map ``reversing arrows''. 
This is, for example, the case for the well-known duality between the hyperbolic Calogero-Moser system and the rational Ruijsenaars-Schneider system (see \ref{ss:DualTCM}), or the self-duality of the hyperbolic Ruijsenaars-Schneider system (see \ref{ss:DualRS}).  
Note that this simple point of view also allows to derive new instances of action-angle duality for a generalisation of the rational Calogero-Moser system due to Chalykh and Silantyev \cite{CS} (see \ref{ss:DualRCM}), or a modification of the hyperbolic Ruijsenaars-Schneider system obtained jointly with Chalykh in \cite{CF} (see \ref{ss:DualmRS}). 

The main tool needed to provide a precise construction of the action-angle duality maps is a study of morphisms of double (quasi-)Poisson algebras, which forms the core of this text. We will deduce from this formalism the following important result. 
\begin{thm}[See Theorems \ref{Thm:HQuivers} and \ref{Thm:qHQuivers}] \label{Thm:MAIN}
Up to isomorphism, the double (quasi-)Poisson algebra associated with a quiver by Van den Bergh \cite{VdB1} only depends on the underlying quiver seen as an undirected graph.
\end{thm}
There exists a different notion of non-commutative Poisson structures due to Crawley-Boevey \cite{CB11}, called $H_0$-Poisson structure, and we will show that morphisms of double (quasi-)Poisson algebras induce morphisms of $H_0$-Poisson structures. 
If we work over an algebraically closed field of characteristic zero, 
we will note as part of Proposition \ref{Pr:H0Rep} that a morphism of $H_0$-Poisson structures induces a Poisson morphism between the rings of invariant functions on the corresponding moduli spaces of representations. 
Using the latter result for the $H_0$-Poisson structures induced by the double Poisson algebra structures considered in Theorem \ref{Thm:MAIN}, we obtain the well-known statement that a quiver variety only depends on the quiver seen as an undirected graph (for \emph{fixed} dimension vector and parameter of the moment map). Using the double quasi-Poisson algebra structures of Van den Bergh instead, we get the same statement for multiplicative quiver varieties, which can be attributed to Yamakawa \cite{Y}.  
Therefore, such Poisson isomorphisms of (multiplicative) quiver varieties have a non-commutative origin given by Theorem \ref{Thm:MAIN}. This will be the key observation that we shall need to construct explicitly the Poisson isomorphisms that yield the action-angle duality maps that were mentioned earlier. 

Let us finish by remarking a side result of the present work, which relates to the study of noncommutative algebras in their own right. We will establish  that all the automorphisms of the first Weyl algebra $\Aalg$ and the quantum torus $\Calg^{\rc}$ (see \ref{s:H0app} for definitions) are induced by isomorphisms of double (quasi-)Hamiltonian algebras. 

\medskip

\textbf{Outline of the paper.} 
In Section \ref{S:Base}, we introduce double brackets \cite{VdB1} and their morphisms. We also review the notion of fusion, which allows to identify idempotents in an algebra and which preserves double brackets. 
In Section \ref{S:Ham}, we study morphisms of double Poisson algebras, and their relation to fusion. This investigation ends up with Theorem \ref{Thm:HQuivers} where we obtain that Van den Bergh's double Poisson algebras associated with quivers are independent of the orientation of the underlying quivers. 
The latter section sets the stage for Section \ref{S:qHam} where we provide a non-trivial adaptation of the previous results to the case of double quasi-Poisson algebras, and to a subclass of these algebras associated with quivers, see Theorem \ref{Thm:qHQuivers}.  
In Section \ref{S:H0}, we introduce $H_0$-Poisson structures following Crawley-Boevey \cite{CB11}, and we explain how morphisms of double (quasi-)Poisson algebras yield morphisms of $H_0$-Poisson structures. We describe how they induce Poisson morphisms on moduli spaces of representations. 
Finally, we deal with the action-angle duality of various integrable systems in Section \ref{S:dual} using several results that are derived throughout the paper. 
There are three appendices containing ancillary statements and proofs.   

\medskip

{\bf Acknowledgement.} I thank V. Rubtsov for raising the question that motivated this work, and L. Feh\'{e}r for useful comments. 
I am grateful to the reviewers for their numerous suggestions and questions which helped to improve the presentation of this paper. 
This research was supported by a Rankin-Sneddon Research Fellowship of the University of Glasgow.

\section{Basic definitions}  \label{S:Base} 

Throughout the paper,  $\kk$ is a field of characteristic $0$, and we write $\otimes=\otimes_\kk$. A  $\kk$-algebra is always assumed to be associative,  unital and finitely generated. 
We use $d=d'\otimes d''\in A \otimes A$  as a shorthand way for Sweedler's notation $d=\sum_i d'_i \otimes d''_i$. 
If $A,B$ are $\kk$-algebras, we say that $A$ is a $B$-algebra if there is a morphism of $\kk$-algebras $B\to A$. 
We only consider the following special case throughout the text: if the unit in $A$ admits a decomposition in terms of a finite set of orthogonal idempotents $(e_s)_{s\in I}$, i.e. 
\begin{equation} \label{Eq:Bidem}
 1=\sum_{s\in I} e_s\,, \quad e_s e_t = \delta_{st} e_s\,, \quad |I|\in \N^\times\,,
\end{equation}
we view $A$ as a $B$-algebra for 
\begin{equation} \label{Eq:B}
B=\oplus_{s\in I} \kk e_s\,.
\end{equation}

\subsection{Double brackets}

We closely follow the exposition \cite{F19} of the work of Van den Bergh \cite{VdB1}.

Let $A$ be a $\kk$-algebra. A \emph{double bracket} on $A$ is a $\kk$-bilinear map  $\dgal{-,-}:A\times A \to A \otimes A$  satisfying 
\begin{equation} \label{Eq:cycanti}
 \dgal{a,b}=-\dgal{b,a}^\circ \quad \text{ for all } a,b\in A\,, \qquad \text{(cyclic antisymmetry)}
\end{equation}
where $(-)^\circ$ denotes the permutation of factors in $A \otimes A$, together with 
\begin{equation}\label{Eq:outder}
 \dgal{a,bc}=\dgal{a,b}c+b\dgal{a,c} \quad \text{ for all } a,b,c\in A\,. \qquad \text{(right derivation rule)}
\end{equation}
Here, the multiplication refers to the outer $A$-bimodule structure on $A\otimes A$, that is if $d\in A^{\otimes 2}$,  then we have $a \,d\, b=(a d') \otimes (d'' b)$.  Assuming that \eqref{Eq:cycanti} holds, one can easily check that \eqref{Eq:outder} is equivalent to 
\begin{equation}\label{Eq:inder}
 \dgal{bc,a}=\dgal{b,a}\ast c+b\ast\dgal{c,a} \quad \text{ for all } a,b,c\in A\,, \qquad \text{(left derivation rule)}, 
\end{equation} 
where we use the inner $A$-bimodule structure on $A\otimes A$, i.e. $a\ast d \ast b = (d'b) \otimes (a d'')$. 
As a consequence of the derivation rules,  it suffices to define double brackets on generators of $A$. 

Given a double bracket, we can define an operation $A^{\times 3}\to A^{\otimes 3}$ by setting
\begin{equation}
\label{Eq:TripBr}
\begin{aligned}
  \dgal{a,b,c}=&\dgal{a,\dgal{b,c}'}\otimes \dgal{b,c}'' \\
&+\tau_{(123)}\dgal{b,\dgal{c,a}'}\otimes \dgal{c,a}'' \\
&+\tau_{(123)}^2\dgal{c,\dgal{a,b}'}\otimes \dgal{a,b}'' \quad \text{ for all } a,b,c\in A \,.
\end{aligned}
\end{equation}
(Here, we define $\tau_{(123)}:\,A^{\otimes 3}\to A^{\otimes 3}$  by 
$\tau_{(123)}(a_1\otimes a_2\otimes a_3)=a_{3}\otimes a_1 \otimes a_{2}$.)
The map \eqref{Eq:TripBr} is an instance of \emph{triple bracket}, that is a  $\kk$-trilinear map,  which satisfies a generalisation of the cyclic antisymmetry \eqref{Eq:cycanti} 
\begin{equation} \label{Eq:TriAnti}
\tau_{(123)}\circ \dgal{-,-,-}\circ \tau_{(123)}^{-1}
=\dgal{-,-,-}\,,
\end{equation}
and which is a derivation in its last argument for the outer $A$-bimodule structure of $A^{\otimes 3}$.

When we see $A$ as a $B$-algebra for $B$ given as in \eqref{Eq:B}, it is convenient to work in the relative setting. In such a case, we assume that the double bracket $\dgal{-,-}$ vanishes whenever one of its entries is taken in $B$. We then say that the double bracket is $B$-linear. 
Equivalently, $B$-linearity can be stated as having a morphism of $B$-bimodules $\dgal{a,-}:A\to A\otimes A$ for any $a\in A$, where we use the outer $B$-bimodule structure on $A\otimes A$. 
Note that if $a=e_{r_a}ae_{s_a}\in A$ and $b=e_{r_b}be_{s_b}\in A$ for $r_a,r_b,s_a,s_b\in I$, the derivation rules and $B$-linearity yield  
\begin{equation}
 \dgal{a,b}\in e_{r_b}A e_{s_a}\otimes e_{r_a} A e_{s_b}\,.
\end{equation}

\subsubsection{Morphisms of double brackets}
Let $A_1,A_2$ be  $B$-algebras endowed with $B$-linear double brackets $\dgalU{-,-},\dgalD{-,-}$. We say that $\phi:A_1 \to A_2$ is a \emph{morphism of double brackets} if it is a $B$-algebra homomorphism such that for any $a,b \in A_1$
\begin{equation} \label{Eq:Morph}
 \dgalD{\phi(a),\phi(b)}=(\phi\otimes \phi) \dgalU{a,b}\,.
\end{equation}
We say that $\phi$ is an \emph{isomorphism (of double brackets)} if it is also an isomorphism of $B$-algebras. In that case, the inverse $\phi^{-1}: A_2\to A_1$ is a morphism of double brackets.
\begin{exmp}
 The algebra $A=\kk\langle x,y\rangle$ can be endowed with two $\kk$-linear double brackets, denoted $\dgalU{-,-},\dgalD{-,-}$, which are defined on generators by 
 \begin{equation*}
  \dgalU{x,x}=0=\dgalD{x,x}\,, \quad \dgalU{y,y}=0=\dgalD{y,y}\,, \quad 
  \dgalU{x,y}=1 \otimes 1,\,\, \dgalD{x,y}=-1 \otimes 1\,.
 \end{equation*}
The automorphism $x \mapsto y, y \mapsto x$ defines an isomorphism of double brackets $(A,\dgalU{-,-})\to(A,\dgalD{-,-})$.  
More generally, assume that $A_1$ is equipped with a $B$-linear double bracket $\dgalU{-,-}$, and fix an isomorphism of $B$-algebras $\phi:A_1\to A_2$. Then we get a $B$-linear double bracket on $A_2$ defined by  
$$\dgalD{-,-}=(\phi\otimes \phi)\circ \dgalU{-,-}\circ (\phi^{-1} \times \phi^{-1})\,,$$
and $\phi$ naturally becomes an isomorphism of double brackets. 
\end{exmp}

\subsection{Morphisms and fusion}

\subsubsection{Fusion} \label{ssFus}

Following Van den Bergh \cite[\S 2.5]{VdB1}, we assume that $A$ is a $B$-algebra such that there exist orthogonal idempotents $e_1,e_2\in B$, and we construct the \emph{fusion algebra} $A^f_{e_2 \to e_1}$ obtained by fusing $e_2$ onto $e_1$ as follows. (When the choice of idempotents is clear from the context, we simply write $A^f_{e_2 \to e_1}$  as $A^f$.)

First, we extend the algebra $A$ along the pair $(e_1,e_2)$ by adjoining generators $e_{12},e_{21}$ satisfying the usual matrix relations. Namely, 
\begin{equation}
  \bar{A}=A \ast_{\kk e_1 \oplus \kk e_2 \oplus \kk \he} (\Mat_2(\kk)\oplus \kk \he)\,,
\end{equation}
 where $\he=1-e_1-e_2$, and $\Mat_2(\kk)$ is seen as the $\kk$-algebra generated by $e_1=e_{11},e_{12},e_{21},e_2=e_{22}$ with $e_{st}e_{uv}=\delta_{tu}e_{sv}$. 
 Applying the same construction to $B$, we can introduce $\bar{B}$ which is such that $\bar{A}=A \ast_B \bar{B}$. 

Second, we consider the algebra $A^f$  constructed from $\bar{A}$ by dismissing elements of $e_2 \bar{A} + \bar{A} e_2$. 
We call $A^f$ the \emph{fusion algebra of $A$} obtained by fusing $e_2$ onto $e_1$. It is explicitly given by  
\begin{equation}
  A^f=\, \epsilon \bar{A} \epsilon\,, \qquad \text{for } \epsilon=1-e_2\,.
\end{equation}
By construction, $A^f$ is a $B^f$-algebra for  $B^f=\epsilon \bar{B} \epsilon$. 
Using the map 
\begin{equation} \label{Eq:prfus}
 A \to A^f\, : \, a \mapsto a^f:=\epsilon a \epsilon + e_{12}ae_{21} + e_{12}a \epsilon + \epsilon a e_{21}\,,
\end{equation}
we can get a convenient set of generators in $A^f$, as observed in \cite[\S 5.3]{VdB1}. 
\begin{lem} \label{AfGenerat} 
  Generators of $A^f$ can be chosen to be of the following four types:  
\begin{subequations}
  \begin{align}
   (\text{first type})&\qquad\qquad t^f=t\,, \qquad t \in \epsilon A \epsilon\,, \label{type1}\\
   (\text{second type})&\qquad\qquad u^f=e_{12}u\,, \qquad u \in e_2 A \epsilon\,,\label{type2} \\
   (\text{third type})&\qquad\qquad v^f=v e_{21}\,, \qquad v \in \epsilon A e_2\,, \label{type3} \\
   (\text{fourth type})&\qquad\qquad w^f=e_{12} w e_{21}\,, \qquad w \in e_2 A e_2\,. \label{type4}
  \end{align}
\end{subequations}
\end{lem}

\subsubsection{Fusion and quivers} \label{sss:Quiver}

We  fix some terminology related to quivers, i.e. directed graphs, which will be used throughout the text. 
A quiver $Q$ consists of a vertex set $I$ and an arrow set $\{a \in Q\}$. We define the head and tail maps $h,t:Q \to I$ which assign to any arrow $a$ its head (ending vertex) $h(a)$ and its tail (starting vertex) $t(a)$. In other words, $a:t(a)\to h(a)$. The double $\bar{Q}$ of $Q$ is obtained by adding to $Q$ the arrows $a^\ast:h(a)\to t(a)$ for each $a \in Q$. We then extend $t,h$ to $\bar{Q}$ and, if we extend the map $a\to a^\ast$ defined on $Q\subset \bar{Q}$ to $\bar{Q}$ by setting $(a^\ast)^\ast=a$ for $a^\ast \in \Bar Q \setminus Q$, we see that $h(a)=t(a^\ast)$ for each $a \in \bar{Q}$. 
The path algebra $\kk Q$ of a quiver $Q$ is generated by symbols $e_s$ for $s \in I$, and  $a\in Q$, subject to the relations 
$e_s e_t=\delta_{st}e_s$, $a=e_{t(a)}ae_{h(a)}$ (this implies that we read paths from left to right), where the multiplication is given by concatenation of paths.

Let us now see what is the meaning of fusion on a quiver. 
Fix a quiver $Q$ with vertex set $I=\{1,\ldots,k\}$ for some $k \geq 2$.  
The path algebra $\kk Q$ of $Q$ is an algebra over $B=\oplus_{s\in I}\kk e_s$, where $e_s$ denotes the idempotent attached to the $s$-th vertex. 
We can form the algebra $(\kk Q)^f$ as in \ref{ssFus} obtained by fusing $e_2$ onto $e_1$. 
We can see that $(\kk Q)^f$ is an algebra over $B^f=\oplus_{s\in I\setminus \{2\}}\kk e_s$ generated by the following elements 
\begin{equation*}
  \begin{aligned}
&a\,\,\,\text{with }a\in Q \text{ such that } t(a),h(a)\neq 2\,, \\
&e_{12}a\,\,\,\text{with }a\in Q \text{ such that } t(a)=2,\,h(a)\neq 2\,,\\
&a e_{21}\,\,\,\text{with }a\in Q \text{ such that } t(a)\neq 2,\,h(a)= 2\,, \\ 
&e_{12}ae_{21}\,\,\,\text{with }a\in Q \text{ such that } t(a),h(a)= 2\,, \quad 
  \end{aligned}
\end{equation*}
which are images of the arrows in $Q$, of the four types given in Lemma \ref{AfGenerat}. 
At the same time, we can form the quiver $Q^f$ with vertex set $I^f=I\setminus \{2\}$ by fusing together the vertices $1$ and $2$ in $Q$. 
The corresponding path algebra $\kk (Q^f)$ is easily identified with $(\kk Q)^f$, so that the fusion operation is the analogue at the level of the path algebras of identifying two vertices in a quiver.

\begin{rem}
By repeatedly using the construction from \ref{ssFus}, we can  obtain an algebra from $A$ by identification of several idempotents.  
Note that there is a direct way to define this algebra using quivers (with relations), see \cite[\S2.1]{Br01}. 
We will not follow this more general approach due to technicalities arising when performing fusion in double quasi-Poisson algebras, see Proposition \ref{Pr:IsoFusqHam}.  
\end{rem}

\subsubsection{Fusion and double brackets}\label{ssFusDB}

We now assume that $A$ is endowed with a $B$-linear double bracket. 
As noted in \cite[\S 2.5]{VdB1}, the double bracket uniquely extends from $A$ to $\bar A$ by requiring it to be $\bar{B}$-linear, and it can then be restricted to $A^f$. 
If $a^f,b^f\in A^f$ are two generators as in Lemma \ref{AfGenerat}, there exists $a,b\in A$ and $e_a,e_b \in \{\epsilon,e_{12}\}$, $f_a,f_b \in \{\epsilon,e_{21}\}$, such that $a^f=e_a a f_a$, $b^f=e_b b f_b$. We can then define the double bracket induced by $A$ onto $A^f$ using the following identity
\begin{equation} \label{Eq:BrFus}
 \dgal{a^f, b^f}=e_b \dgal{a,b}' f_a \otimes e_a \dgal{a,b}'' f_b\,.
\end{equation}
In \eqref{Eq:BrFus}, the double bracket on the left-hand side is the one induced in $A^f$, while the double bracket on the right-hand side is the original one defined in $A$. 

The next result shows that, up to isomorphism, the fusion algebra and its induced double bracket only depend on the unordered choice of idempotents $\{e_1,e_2\}$.  
\begin{lem} \label{L:IsoId}
 Let $A_1=A^f_{e_2 \to e_1}$ be the fusion algebra obtained by fusing $e_2$ onto $e_1$, and let $A_2=A^f_{e_1 \to e_2}$ be the fusion algebra obtained by fusing $e_1$ onto $e_2$. Then, the identity map on $A$ induces an isomorphism of double brackets $A_1 \to A_2$. 
\end{lem}
\begin{proof}
 We let $\he=1-e_1-e_2$ and note that $A_1=(1-e_2)\bar A (1-e_2)$ and $A_2=(1-e_1)\bar A (1-e_1)$. We can define a map $\phi:A_1 \to A_2$  which is given on generators of first type \eqref{type1} as 
 \begin{equation*}
  \phi(t)=t,\,\,\text{if } t \in \he A \he; \quad 
  \phi(t)=e_{21}t,\,\,\text{if } t \in e_1 A \he; \quad 
  \phi(t)=te_{12},\,\,\text{if } t \in \he A e_1; \quad 
  \phi(t)=e_{21}te_{12},\,\,\text{if } t \in e_1 A e_1; 
 \end{equation*}
on generators of second type  \eqref{type2} as 
 \begin{equation*}
  \phi(e_{12}u)=u,\,\,\text{if } u \in e_2 A \he; \quad 
  \phi(e_{12}u)=ue_{12},\,\,\text{if } u \in e_2 A e_1; 
 \end{equation*}
on generators of third type  \eqref{type3} as  
  \begin{equation*}
  \phi(ve_{21})=v,\,\,\text{if } v \in \he A e_2; \quad 
  \phi(ve_{21})=e_{21}v,\,\,\text{if } v \in e_1 A e_2; 
 \end{equation*}
 on generators of fourth type \eqref{type4} as
   \begin{equation*}
  \phi(e_{12}we_{21})=w,\,\,\text{if } w \in e_2 A e_2.
 \end{equation*}
 The map $\phi$ is easily seen to be the image of the identity under the two projections $\pi_1 : A \to A_1$, $\pi_2 : A \to A_2$ given by \eqref{Eq:prfus}, i.e. $\phi \circ \pi_1 = \pi_2 \circ \id_A$. Moreover, swapping the labels $1,2$ provides the  inverse  of $\phi$. Thus, we only need to check that it is a morphism of double brackets. 
 
We can decompose $A$ as $\bigoplus_{i,j=0,1,2} e'_{i}Ae'_j$ if we set  $e'_0=\he$ and $e'_{i}=e_{i}$ for $i=1,2$. 
Without loss of generality, take  $a,b \in A$ belonging to one of these subsets. Then, there exists $e_a^1,e_b^1 \in \{e_{12},e_1,\he\}$, $f_a^1,f_b^1 \in \{e_{21},e_1,\he\}$, such that  under the projection $\pi_1$ the elements $e_a^1 a f_a^1$ and $e_b^1 b f_b^1$ are generators of $A_1$, see Lemma \ref{AfGenerat}. 

In the same way, there exists $e_a^2,e_b^2 \in \{e_{21},e_2,\he\}$, $f_a^2,f_b^2 \in \{e_{12},e_2,\he\}$, such that   the elements $e_a^2 a f_a^2$ and $e_b^2 b f_b^2$ obtained from the projection $\pi_2$ are generators of $A_2$. In particular, $\phi \circ \pi_1 = \pi_2$ implies that 
\begin{equation*}
 \phi(e_a^1)=e_a^2,\,\, \phi(e_b^1)=e_b^2,\,\, \phi(f_a^1)=f_a^2,\,\, \phi(f_b^1)=f_b^2\,.
\end{equation*}

Using the identity \eqref{Eq:BrFus} for the double bracket $\dgal{-,-}_k$ induced by $\dgal{-,-}$ in $A_k$   for $k=1,2$, we get 
\begin{equation*}
 \dgal{e_a^k a f_a^k, e_b^k b f_b^k}_k=e_b^k \dgal{a,b}' f_a^k \otimes e_a^k \dgal{a,b}'' f_b^k\,,
\end{equation*}
where the double bracket on the right-hand side is taken in $A$. 
We can then directly see that \eqref{Eq:Morph} is satisfied.  
\end{proof}

Fusion also preserves morphisms of double brackets.
\begin{lem} \label{L:IsoFus}
 Let $\phi:A_1 \to A_2$ be a morphism of double brackets over $B$. Let $A^f_1=(A_1)^f_{e_2 \to e_1}$, $A^f_2=(A_2)^f_{e_2 \to e_1}$, be the fusion algebras with double brackets  obtained by fusing $e_2$ onto $e_1$. Then $\phi$ induces a morphism of double brackets $\phi^f : A_1^f \to A_2^f$. 
\end{lem}
\begin{proof}
Recall that $A_1^f$ is generated by elements $a^f=e_a a f_a$ for $a\in A_1$ and $e_a \in \{\epsilon, e_{12}\}$, $f_a\in \{\epsilon,e_{21}\}$. We can then set 
\begin{equation} \label{Eq:IsoFus}
 \phi^f: A_1^f \to A_2^f : a^f=e_a a f_a \mapsto \phi^f(a^f)=e_a \phi(a)f_a\,. 
\end{equation}
For $k=1,2$, let $\dgal{-,-}_k$ and $\dgal{-,-}_k^f$ be respectively the double bracket on $A_k$ and its induced double bracket on $A_k^f$. 
We have for any two generators $a^f=e_a a f_a,b^f=e_b b f_b \in A_1^f$ as described above that 
\begin{equation*}
\begin{aligned}
  \dgalD{\phi^f(a^f),\phi^f(b^f)}^f=&e_b\dgalD{\phi(a),\phi(b)}' f_a \otimes e_a \dgalD{\phi(a),\phi(b)}'' f_b \\
  =&e_b\phi(\dgalU{a,b}') f_a \otimes e_a \phi(\dgalU{a,b}'') f_b\\
  =&(\phi^f \otimes \phi^f) e_b \dgalU{a,b}' f_a \otimes e_a \dgalU{a,b}'' f_b \\
 =& (\phi^f \otimes \phi^f) \dgalU{a^f,b^f}\,.
\end{aligned}
\end{equation*}
Here, we used formula \eqref{Eq:BrFus} in $A_2^f$/$A_1^f$ for the first/last equality, the morphism property \eqref{Eq:Morph} in the second equality, and the definition of $\phi^f$ for the remaining equality. 
\end{proof}

In Lemma \ref{L:IsoFus}, it is clear that if $\phi$ is an isomorphism, then so is $\phi^f$. For the remainder of this text, we omit to mention such an extension to the case of an isomorphism in analogous results.

\subsubsection{Identification of idempotents in different algebras} \label{ssOplus}
Let $A_1$ and $A_2$ be algebras endowed with double brackets respectively over $B_1$ and $B_2$, and consider $A_1 \oplus A_2$ as a $(B_1 \oplus B_2)$-algebra. 
It is easy to see that there exists a unique $(B_1 \oplus B_2)$-linear double bracket $\dgal{-,-}^\oplus$ on  $A_1 \oplus A_2$ extending 
$\dgal{-,-}_1$ and $\dgal{-,-}_2$, while it is such that  $\dgal{c_1,c_2}^\oplus=0$ whenever $c_1=(a_1,0)$, $c_2=(0,a_2)$, with $a_1\in A_1$, $a_2\in A_2$. 

For any idempotents $e_1\in B_1$ and $e_2 \in B_2$, we can form the fusion algebra $(A_1\oplus A_2)^f_{e_2 \to e_1}$ by fusing $e_2$ onto $e_1$. This algebra inherits a $(B_1\oplus B_2)^f_{e_2 \to e_1}$-linear double bracket as noted in \ref{ssFusDB}.

\begin{exmp}[Extension by a central element]
Let $A$ be a $\kk$-algebra equipped with a double bracket $\dgal{-,-}$. Let $\kk\langle y\rangle$ be equipped with the trivial double bracket $\dgal{y,y}=0$. 
If we identify the units of these two algebras embedded as $(1,0),(0,1)$ inside $A\oplus \kk\langle y\rangle$, the fusion algebra $A_+:=A \ast_\kk  \kk\langle y\rangle$ hence obtained is endowed with the double bracket $\dgal{-,-}_+$ given for any $a,b \in A$ by 
$\dgal{a,b}_+=\dgal{a,b}$, $\dgal{a,y}_+=0$ and $\dgal{y,y}_+=0$. 
(Note that if $\dgal{-,-}$ is Poisson in the sense of Section \ref{S:Ham}, then  $(A_+,\dgal{-,-}_+)$ is a double Poisson algebra.)
\end{exmp}

\section{Morphisms of double Poisson algebras} \label{S:Ham}

As in Section \ref{S:Base}, all algebras are over $B=\bigoplus_{s\in I}\kk e_s$.  We will frequently identify $I$ with $\{1,\ldots,|I|\}\subset \N$.

\subsection{Basic definitions and properties}

\subsubsection{Definitions} Following Van den Bergh \cite{VdB1}, a  \emph{double Poisson bracket} $\dgal{-,-}$ on $A$ is a double bracket for which  the associated triple bracket $\dgal{-,-,-}$ defined by \eqref{Eq:TripBr} identically vanishes. 
In such a case, we say that $(A,\dgal{-,-})$ (or simply $A$) is a \emph{double Poisson algebra}. 
An element $\mu_s\in e_s A e_s$ is said to satisfy the \emph{additive property for} $e_s$ if, for any $a\in A$, we have  
\begin{equation}  \label{mum}
 \dgal{\mu_s,a}=a e_s \otimes e_s - e_s \otimes e_s a\,.
\end{equation}
Then, we define a \emph{moment map} as an element $\mu=\sum_{s\in I}\mu_s$ with $\mu_s\in e_sAe_s$ such that for all $s\in I$ we have that $\mu_s$ satisfies the additive property for $e_s$ \eqref{mum}. 
Note in particular that $e_s\mu e_t=\delta_{st}\mu_s$ for each $s,t\in I$. 
We call the triple $(A,\dgal{-,-},\mu)$ (or simply $A$ if no confusion can arise) a \emph{Hamiltonian algebra}.

Let $A_1,A_2$ be  $B$-algebras with $B$-linear double brackets $\dgalU{-,-},\dgalD{-,-}$.
If $\phi:A_1 \to A_2$ is a morphism of double brackets and $A_1,A_2$ are double Poisson algebras, we say that $\phi$ is a \emph{morphism of double Poisson algebras}. 
If furthermore $A_1,A_2$ are Hamiltonian algebras with respective moment maps $\mu_1,\mu_2$ satisfying $\phi(\mu_1)=\mu_2$, we say that $\phi$ is a \emph{morphism of Hamiltonian algebras}.

\begin{exmp} \label{Exmp:Q1}
Consider the quiver $Q$ with vertices $\{1,2\}$ and one arrow $a: 1 \to 2$. Let $\bar{Q}$ be its double obtained by adding the arrow $a^\ast:2\to 1$. Put $B=\kk e_1 \oplus \kk e_2$. 
Following Van den Bergh \cite[\S 6.3]{VdB1}, the path algebra $\kk \bar{Q}$ is  a Hamiltonian algebra for the $B$-linear double bracket given on generators by 
\begin{equation} \label{Eq:HamQ1}
 \dgal{a,a}_1=0\,, \,\, \dgal{a^\ast,a^\ast}_1=0\,, \quad \dgal{a,a^\ast}_1=e_2 \otimes e_1\,, \,\, \dgal{a^\ast,a}_1=-e_1 \otimes e_2\,,
\end{equation}
and the moment map $\mu=[a,a^\ast]$. The moment map can be decomposed as $\mu_1=a a^\ast$, $\mu_2=-a^\ast a$. (See \ref{sss:Quiver} for the conventions that we follow with respect to quivers.)

Similarly, consider the quiver $Q^{op}$ with vertices $\{1,2\}$ and one arrow $b:2 \to 1$, and let $\bar{Q}^{op}$ be its double with new arrow $b^\ast:1 \to 2$. We can also consider Van den Bergh's Hamiltonian structure on $\kk \bar{Q}^{op}$ which is given by the $B$-linear double bracket 
\begin{equation} \label{Eq:HamQ1op}
 \dgal{b,b}_2=0\,, \,\, \dgal{b^\ast,b^\ast}_2=0\,, \quad \dgal{b,b^\ast}_2=e_1 \otimes e_2\,, \,\, \dgal{b^\ast,b}_2=-e_2 \otimes e_1\,,
\end{equation}
with moment map $\mu'=[b,b^\ast]$. We can see that $\phi: \kk \bar{Q}\to \kk \bar{Q}^{op}$ given by $\phi(a)=b^\ast$, $\phi(a^\ast)=-b$, 
is an isomorphism of $B$-algebras. One readily checks that 
\begin{equation*}
 \phi^{\otimes 2} \dgal{a,a}_1=0=\dgal{b^\ast,b^\ast}_2\,, \quad 
  \phi^{\otimes 2} \dgal{a^\ast,a^\ast}_1=0=\dgal{b,b}_2\,, \quad 
   \phi^{\otimes 2} \dgal{a,a^\ast}_1=e_2 \otimes e_1=\dgal{b^\ast,-b}_2\,,
\end{equation*}
and $\phi(\mu)=\mu'$. Hence, $\phi$ is an isomorphism of Hamiltonian algebras.
\end{exmp}

\subsubsection{First properties}

The next result shows how injectivity or surjectivity of a morphism of double brackets can guarantee that it is a morphism of double Poisson (resp. Hamiltonian) algebras. 
\begin{lem} \label{L:surjHam}
Let  $\phi:A_1 \to A_2$ be a morphism of double brackets.

\noindent 1. Assume that $\phi$ is surjective as a $B$-algebra homomorphism. 

 1.1. If $A_1$ is a double Poisson algebra, then $A_2$ is a double Poisson algebra. 
 
 1.2. If $A_1$ is a Hamiltonian algebra, then $A_2$ admits a structure of Hamiltonian algebra such that $\phi$ is a morphism of Hamiltonian algebras. 
 
\noindent 2. Assume that $\phi$ is injective as a $B$-algebra homomorphism. 

 2.1. If $A_2$ is a double Poisson algebra, then $A_1$ is a double Poisson algebra. 
 
 2.2. If $A_2$ is a Hamiltonian algebra with moment map $\mu'$ such that $\mu'\in \operatorname{Im}(\phi)$, then $A_1$ admits a structure of Hamiltonian algebra such that $\phi$ is a morphism of Hamiltonian algebras. 
 
\noindent 3. All the structures obtained in 1.--2. are unique.  
\end{lem}
\begin{proof}
First, let us consider for $1 \leq i \leq 3$ the element $a_{1,i}\in A_1$,  with image  $a_{2,i}:=\phi(a_{1,i})\in A_2$. 
If we denote by $\sigma$ the permutation $(123)$, we get by \eqref{Eq:TripBr} and \eqref{Eq:Morph}
\begin{equation*}
\begin{aligned}
   \dgalD{a_{2,1},a_{2,2},a_{2,3}}
 =&\sum_{i \in \Z_3}\tau_{\sigma}^i\dgalD{\phi(a_{1,\sigma^i(1)}),\dgalD{\phi(a_{1,\sigma^i(2)}),\phi(a_{1,\sigma^i(3)})}'} \otimes 
 \dgalD{\phi(a_{1,\sigma^i(2)}),\phi(a_{1,\sigma^i(3)})}'' \\
 =&\sum_{i \in \Z_3}\tau_{\sigma}^i(\phi\otimes \phi \otimes \phi)\dgalU{a_{1,\sigma^i(1)},\dgalU{a_{1,\sigma^i(2)},a_{1,\sigma^i(3)}}'} \otimes 
 \dgalU{a_{1,\sigma^i(2)},a_{1,\sigma^i(3)}}'' \\
 =&(\phi \otimes \phi \otimes \phi) \dgalU{a_{1,1},a_{1,2},a_{1,3}}\,.
\end{aligned}
\end{equation*}
This yields in particular the identity of associated triple brackets 
\begin{equation} \label{Eq:RelTripBr}
 (\phi\otimes \phi\otimes \phi) \circ \dgalU{-,-,-}=\dgalD{-,-,-}\circ (\phi\times \phi\times \phi)\,. 
\end{equation}

1.1. If $\dgalU{-,-}$ is Poisson, $\dgalU{-,-,-}$ vanishes and \eqref{Eq:RelTripBr} yields that $\dgalD{-,-}$ is Poisson.   

1.2. Since $A_1$ has a moment map $\mu=\sum_s\mu_s$, the element $\phi(\mu)=\sum_s  \phi(\mu_s)$ is a moment map. Indeed, for any $s  \in I$
\begin{equation*}
 \dgalD{\phi(\mu_s),a_{2,1}}=(\phi \otimes \phi) \dgalU{\mu_s,a_{1,1}}=
 (\phi \otimes \phi) \left( a_{1,1}e_s \otimes e_s - e_s \otimes e_s a_{1,1} \right)
 =  a_{2,1}e_s \otimes e_s - e_s \otimes e_s a_{2,1} \,,
\end{equation*}
where we have used that $\mu_s$ satisfies the additive property for $e_s$ with respect to $\dgalU{-,-}$. 

2.1.-2.2. This is similar, and we need the fact that $\phi^{\otimes 2},\phi^{\otimes 3}$ are injective. 

3. The uniqueness of the structures is a consequence of the constructions. 
\end{proof}

Given a double Poisson algebra $(A,\dgal{-,-})$, we introduce  the set of \emph{Casimir elements} of $A$ by 
\begin{equation}
\Cas(A)=\{a\in A\mid \dgal{a,b}=0\,\, \text{ for all }b\in A\,\} \,.
\end{equation}
 In the case of a $B$-linear double bracket, we have that $B\subset \Cas(A)$. 

\begin{lem} \label{Lem:UniqAd}
 Let $\phi:A_1 \to A_2$ be an isomorphism of double Poisson algebras. If $A_1,A_2$ are Hamiltonian algebras with moment maps $\mu_1,\mu_2$, 
 then $\mu_2-\phi(\mu_1)\in B$ and $\Cas(A_1)=\Cas(A_2)=B$. 
In particular, a moment map is unique up to the addition of an element of $B$.
\end{lem}
\begin{proof}
Let us decompose $\mu_{1}\in A_1$ with the idempotents as $\sum_{s\in I} \mu_{1,s}$ and do the same for $\mu_2\in A_2$. 
We first note that for any $s\in I$, due to the additive property for $e_s$ \eqref{mum}, we have for any $a \in A_2$ that  
 \begin{equation*}
\dgalD{\mu_{2,s}-\phi(\mu_{1,s}),a}  =\dgalD{\mu_{2,s},a}- \phi^{\otimes 2}\dgalU{\mu_{1,s},\phi^{-1}(a)}=0\,.
 \end{equation*}
Thus, $\mu_{2,s}-\phi(\mu_{1,s})\in \Cas(A_2)$ by definition. 

Let us now show that $\Cas(A_1)\subset B$, since we already know the opposite inclusion (the case of $A_2$ is obtained in the same way). 
If $c\in \Cas(A_1)$, we have for any $s\in I$ that 
\begin{equation*}
 0=\dgal{\mu_s,c}=c e_s\otimes e_s - e_s \otimes c e_s \,,
\end{equation*}
where we used in the first equality that $\dgal{-,c}=0$ because $c$ is a Casimir element, while in the second we use the property \eqref{mum} of the moment map. This equality implies that $c e_s = \lambda_s e_s = e_s c$ for some $\lambda_s\in \kk$, thus $c=\sum_{s\in I} \lambda_s e_s \in B$. 

For the second part of the statement, it suffices to consider $\phi=\id_{A_1}:A_1\to A_1$. 
\end{proof}

Note that, in general $B \subsetneq \Cas(A)$. For example, any $B$-algebra $A$ can be endowed with the zero double bracket, which is Poisson and such that $\Cas(A)=A$. This does not contradict Lemma \ref{Lem:UniqAd} since the zero double Poisson bracket does not admit a moment map (except if $A=B$).

\subsubsection{Sum of double Poisson algebras}
Recall that we considered the sum of two double brackets in \ref{ssOplus}. 
The following result holds in the presence of double Poisson brackets.
\begin{lem}[\cite{VdB1}] \label{Lem:SumHam}
 If $A_1$ and $A_2$ are endowed with double Poisson brackets, then the double bracket $\dgal{-,-}^\oplus$ on $A_1\oplus A_2$ defined in \ref{ssOplus}  is
 Poisson. If $\mu_1$ and $\mu_2$ are moment maps in $A_1$ and $A_2$, then $A_1\oplus A_2$ is Hamiltonian with $(\mu_1,\mu_2)$ as moment map.
\end{lem}

\subsection{Double Poisson brackets and fusion} \label{ss:HamFus} 

We assume that the index set $I$ of $B$ is such that $|I|>1$. 

\subsubsection{Fusion in an algebra} Recall the fusion algebra $A^f=A^f_{e_2 \to e_1}$ defined in \ref{ssFus}. We noted in \ref{ssFusDB} that if $A$ is endowed with a double bracket, then $A^f$ has an induced double bracket. 
\begin{prop} \emph{(\cite[Corollary 2.5.6, Proposition 2.6.6]{VdB1})} \label{Pr:IsoFusHam}
If $A$ is a double Poisson algebra over $B$, then $A^f$ equipped with the induced double bracket is a double Poisson algebra over $B^f=\bigoplus_{s\in I \setminus \{2\}}e_s$. Furthermore, if $\mu$ is a moment map for $A$, then its projection $\mu^f$ under the map \eqref{Eq:prfus} is a moment map for $A^f$.  
\end{prop}

\begin{exmp}
Consider the path algebra $A=\kk \bar{Q}$ analysed in Example \ref{Exmp:Q1} with its Hamiltonian algebra structure. Fusing $e_2$ onto $e_1$, we get a Hamiltonian algebra structure on $A^f$ which can be identified with the free algebra $\kk \langle a,a^\ast\rangle$. The induced double bracket is determined by the first two identities of \eqref{Eq:HamQ1} and  $\dgal{a,a^\ast}=1 \otimes 1$. The moment map is $\mu=[a,a^\ast]$. The algebra is the path algebra of the double of the one-loop quiver, and the Hamiltonian structure obtained by fusion is just the one defined by Van den Bergh \cite[\S 6.3]{VdB1}. 
\end{exmp}

\begin{lem} \label{L:IsoFusHam0}
Let $A$ be a double Poisson algebra over $B$. Consider the algebra $A_1=A^f_{e_2 \to e_1}$ obtained by fusing $e_2$ onto $e_1$, and the algebra  $A_2=A^f_{e_1 \to e_2}$ obtained by fusing $e_1$ onto $e_2$. Then, the identity map on $A$ induces an isomorphism of double Poisson algebras $\phi:A_1 \to A_2$. 
If $A$ is a Hamiltonian algebra, then $\phi$ is an isomorphism of Hamiltonian algebras.
\end{lem}
\begin{proof}
 The first part directly follows from Lemma \ref{L:IsoId} and Proposition \ref{Pr:IsoFusHam}. In the Hamiltonian case, the morphism $\phi$ constructed in the proof of Lemma \ref{L:IsoId} satisfies
 \begin{equation*}
  \phi(\mu^{(1)})=\mu^{(2)}\,, \quad \text{ for }
\mu^{(1)}=\mu_1+e_{12}\mu_2e_{21}+\sum_{s\neq 1,2}\mu_s,\,\, \mu^{(2)}=e_{21}\mu_1e_{12} + \mu_2 +\sum_{s\neq 1,2}\mu_s\,.
 \end{equation*}
But $\mu^{(1)},\mu^{(2)}$ are the moment maps of $A_1$ and $A_2$ by Proposition \ref{Pr:IsoFusHam}.
\end{proof}

In Lemma \ref{L:IsoFusHam0}, the map $\phi:A_1 \to A_2$ hence obtained is a morphism of $B'$-algebras, where $B'=\kk \tilde e \oplus\bigoplus_{s\in I \setminus \{1,2\}}\kk e_s$. Here, the base maps satisfy  $\tilde e \mapsto e_1\in A_1$, $\tilde e \mapsto e_2 \in A_2$ and are given in an obvious way on the other idempotents $e_s$, $s\in I\setminus \{1,2\}$.

\subsubsection{Fusion of several idempotents}

\begin{lem}   \label{L:IsoFusHam1}
Let $A$ be a double Poisson algebra over $B$, and let $e_1,e_2,e_3\in B$ be orthogonal idempotents. 
Let $A_1:= (A^f_{e_3\to e_2})_{e_2 \to e_1}^f$ (resp. $A_2:=(A^f_{e_2 \to e_1})_{e_3 \to e_1}^f$) be the algebra obtained by fusing $e_3$ onto $e_1$, then $e_2$ onto $e_1$ 
(resp. $e_2$ onto $e_1$, then $e_3$ onto $e_1$). Then the identity map on $A$ induces an isomorphism of double Poisson algebras $\psi:A_1\to A_2$ over $B'=\bigoplus_{s\in I \setminus \{2,3\}}\kk e_s$. 
Furthermore, if $A$ is a Hamiltonian algebra, then $\psi$ is an isomorphism of Hamiltonian algebras. 
\end{lem}
\begin{proof}
 Using Lemma \ref{AfGenerat} twice, we can write generators of $A_1$ and $A_2$ as 
\begin{equation*}
\begin{aligned}
  &e_a a f_a\,, \quad \text{ for }  a \in A,\,\, e_a\in \{1-e_3-e_2,e_{12},e_{12}e_{23}\},\,\, f_a\in \br{1-e_3-e_2,e_{21},e_{32}e_{21}}, \quad (\text{in }A_1)\,; \\
  &e_a' a f_a'\,, \quad  \text{ for }  a \in A,\,\, e_a'\in \{1-e_3-e_2,e_{12},e_{13}\},\,\, f_a'\in \br{1-e_3-e_2,e_{21},e_{31}}, \quad (\text{in }A_2)\,.
\end{aligned}
\end{equation*}
We can then define $\phi:A_1\to A_2$ on generators as $\phi(e_a a f_a)=e_a' a f_a'$ with $e_a'=e_a$ if $e_a\neq e_{12}e_{23}$ or $e_a'=e_{13}$ if $e_a=e_{12}e_{23}$, while 
$f_a'=f_a$ if $f_a\neq e_{32}e_{21}$ or $f_a'=e_{31}$ if $f_a=e_{32}e_{21}$. This is easily seen to be an isomorphism and, using \eqref{Eq:BrFus} twice to induce the double bracket from $A$ to $A_1$ or $A_2$, we remark that it is also a morphism of double brackets. Hence, we have an isomorphism of double Poisson algebras. 

In the Hamiltonian case, Proposition \ref{Pr:IsoFusHam} yields that the moment maps are given by 
\begin{equation*}
 \mu^{(1)}=e_{12}\mu_2e_{21}+e_{12}e_{23}\mu_3e_{32}e_{21}+\sum_{s \neq 2,3}\mu_s \, \in A_1\,, \quad 
 \mu^{(2)}=e_{12}\mu_2e_{21}+e_{13}\mu_3e_{31}+\sum_{s \neq 2,3}\mu_s \, \in A_2\,,
\end{equation*}
so that $\phi(\mu^{(1)})=\mu^{(2)}$. 
\end{proof}

Let us note the following interesting result. 
By gathering Lemmas \ref{L:IsoFusHam0} and \ref{L:IsoFusHam1}, the double Poisson algebra structure obtained by fusion of three orthogonal idempotents does not depend on the order with respect to which we perform fusion, up to isomorphism. For example, 
 \begin{equation*}
(A^f_{e_2\to e_1})_{e_3 \to e_1}^f  \overset{\sim}{\rightarrow} (A^f_{e_3\to e_2})_{e_2 \to e_1}^f 
\overset{\sim}{\rightarrow} (A^f_{e_3\to e_2})_{e_1 \to e_2}^f 
\overset{\sim}{\rightarrow} (A^f_{e_3\to e_1})_{e_1 \to e_2}^f 
\overset{\sim}{\rightarrow} (A^f_{e_3\to e_1})_{e_2 \to e_1}^f \,.
 \end{equation*}

Next, we establish that the order in which we perform fusion is not important if the idempotents are not related. 

\begin{lem}   \label{L:IsoFusHam2}
Let $A$ be a double Poisson algebra over $B$, and let $e_1,e_2,e_3,e_4\in B$ be orthogonal idempotents. 
Let $A_1:= (A^f_{e_4\to e_3})_{e_2 \to e_1}^f$ (resp. $A_2:=(A^f_{e_2 \to e_1})_{e_4 \to e_3}^f$) be the algebra obtained by fusing $e_4$ onto $e_3$, then $e_2$ onto $e_1$ 
(resp. $e_2$ onto $e_1$, then $e_4$ onto $e_3$). Then the identity map on $A$ induces an isomorphism of double Poisson algebras $\psi:A_1\to A_2$ over $B'=\bigoplus_{s\in I \setminus \{2,4\}}\kk e_s$. 
Furthermore, if $A$ is a Hamiltonian algebra, then $\psi$ is an isomorphism of Hamiltonian algebras. 
\end{lem}
\begin{proof}
 Using Lemma \ref{AfGenerat} twice, we can write both generators of $A_1$ and $A_2$ as 
\begin{equation*}
e_a a f_a\,, \quad \text{ for }  a \in A,\,\, e_a\in \{1-e_4-e_2,e_{12},e_{34}\},\,\, f_a\in \br{1-e_4-e_2,e_{21},e_{43}}\,.
\end{equation*}
In terms of these generators, the identity map provides the desired morphism. 
\end{proof}
By gathering these results, we obtain that the double Poisson (or Hamiltonian) algebra structure of an algebra obtained by successive fusions is independent of the precise order in which we identify idempotents. To state the result, fix a double Poisson algebra $A$ over $B=\bigoplus_{s\in I}\kk e_s$, and consider a partition $I=\sqcup_{j\in J} I_j$. We say that $A_1$ is a \emph{fusion algebra respecting the partition} $\sqcup_{j\in J} I_j$ if $A_1$ is obtained by a finite number of fusions starting from $A$, so that all the idempotents $(e_s)_{s\in I_j}$ end up being identified together for each $j\in J$. Equivalently, there exists $s_j\in I_j$ for each $j\in J$ and a map   $\pi_1:A \to A_1$ obtained by composing the morphisms \eqref{Eq:prfus} induced by a finite number of fusions  such that
\begin{equation}
 \pi_1(e_s)=e_{s_j}\quad \text{ for each }s\in I_j \text{ and }\,\,j\in J\,.
\end{equation} 
Note that if $A$ is a double Poisson algebra, then so too is $A_1$ by repeated use of Proposition \ref{Pr:IsoFusHam}.

\begin{thm} \label{Thm:IsoFusHam}
 Let $A$ be a double Poisson algebra over $B=\bigoplus_{s\in I}\kk e_s$. 
Assume that $A_1,A_2$ are fusion algebras respecting a partition $I=\sqcup_{j\in J} I_j$.  
Then, there is a double Poisson algebra isomorphism $\phi:A_1 \to A_2$. If $A$ is Hamiltonian, $\phi$ is an isomorphism of Hamiltonian algebras. 
\end{thm}
\begin{proof}
We prove the result by induction on $|J|$. The base case corresponds to the partition 
\begin{equation}
 |J|=1\,, \quad I_1=I \,,  \label{partition1}
\end{equation}
which amounts to define $A_1$ and $A_2$ by fusing all the idempotents of $B$ together. In that case, the result follows from Lemmas \ref{L:IsoFusHam0}, \ref{L:IsoFusHam1} and \ref{L:IsoFusHam2} (see also the observation made right after Lemma \ref{L:IsoFusHam1}).  

Next, we prove that there exists an isomorphism between $A_1$ and a fusion algebra $\hat{A}_1$ respecting the partition,  such that $\hat{A}_1$ is defined by first performing fusion of all the elements in $I_{|J|}$. If $|I_{|J|}|=1$ there is nothing to prove. Assuming $|I_{|J|}|>1$, we have that 
$A_1$ is obtained from $A$ by a chain of fusions 
\begin{equation} \label{Eq:ChF1}
 e_{r_1} \to e_{t_1},\, \text{ then }
  e_{r_2} \to e_{t_2},\, \text{ then } \,\ldots \, ,\,\,
 e_{r_{\kappa}}\to e_{t_{\kappa}}\,, \quad \kappa:=|I|-|J|\,.
\end{equation}
(Here $e_{r_k}\to e_{t_k}$ means that we fuse $e_{r_k}$ onto $e_{t_k}$.)  
Note that $r_1,\ldots,r_{\kappa}$ are distinct, and there exists $k$ such that $r_k\in I_{|J|}$ (or equivalently $t_k\in I_{|J|}$) by assumption. Let $k_1\in \{1,\ldots,\kappa\}$ be the smallest such integer for which $r_{k_1}\in I_{|J|}$. This means that all the fusions that are performed before are done by fusing idempotents corresponding to subsets distinct from $I_{|J|}$. we can thus use Lemma \ref{L:IsoFusHam2} to get that $A_1$ is isomorphic as a double Poisson algebra to the algebra which is obtained by the chain of fusions 
\begin{equation*}
 e_{r_{k_1}} \to e_{t_{k_1}},\,\, \text{ then }
 e_{r_1} \to e_{t_1}, \, \ldots \,,
  e_{r_{{k_1}-1}} \to e_{t_{{k_1}-1}},\,\, \text{ then } 
    e_{r_{{k_1}+1}} \to e_{t_{{k_1}+1}},\,\ldots \, ,
 e_{r_{\kappa}}\to e_{t_{\kappa}}. 
\end{equation*}
We then run this argument again on the last $\kappa-1$ fusions of this chain. 
By a repeated use of this argument involving Lemma \ref{L:IsoFusHam2} only, we have that  $A_1$ is isomorphic as a double Poisson algebra to $\hat{A}_1$ obtained by a chain of fusion as \eqref{Eq:ChF1}, where this time  
$r_1,\ldots,r_{|I_{|J|}|} \in I_{|J|}$.  In particular, $r_k \notin I_{|J|}$ for $k> |I_{|J|}|$. 

We now do the same with $A_2$ to get an isomorphism of double Poisson algebras with some $\hat{A}_2$, where the latter is obtained by a chain of fusion as \eqref{Eq:ChF1}, where $r_1,\ldots,r_{|I_{|J|}|} \in I_{|J|}$. In particular, up to using the argument for $|J|=1$, we can assume that the first $|I_{|J|}|$ fusions performed in $\hat{A}_1$ and $\hat{A}_2$ are precisely the same. Since  $\hat{A}_1$, $\hat{A}_2$ only differ by fusions of the idempotents corresponding to the partition $I\setminus I_{|J|}=I_1 \sqcup \ldots \sqcup I_{|J|-1}$, we have by induction that $\hat{A}_1$ and  $\hat{A}_2$ are isomorphic. Thus $A_1$ and $A_2$ are also isomorphic as double Poisson algebras. 

We can conclude since the isomorphisms involved are isomorphisms of Hamiltonian algebras if $A$ admits a moment map. 
\end{proof}

\subsubsection{Fusion of morphisms}

\begin{lem}   \label{L:IsoPhiHam}
 Let $\phi:A_1 \to A_2$ be a morphism of double Poisson algebras over $B$. Let $A^f_1=(A_1)^f_{e_2 \to e_1}$, $A^f_2=(A_2)^f_{e_2 \to e_1}$, be the fusion algebras with double brackets  obtained by fusion of $e_2$ onto $e_1$. Then $\phi$ induces a morphism of double Poisson algebras $\phi^f : A_1^f \to A_2^f$.  
 Furthermore if $\phi$ is a morphism of Hamiltonian algebras, then so is $\phi^f$. 
\end{lem}
\begin{proof}
 The map $\phi^f$ given by \eqref{Eq:IsoFus}  is a morphism of double Poisson algebras by Lemma \ref{L:IsoFus} and Lemma \ref{L:surjHam}. 
 In the Hamiltonian case, we can check from \eqref{Eq:IsoFus} that $\phi^f$ maps the moment map of $A_1^f$ to that of $A_2^f$. Hence, it is a morphism of Hamiltonian algebras.  
\end{proof}

\subsection{Application: Van den Bergh's Hamiltonian structure for quivers} \label{ssHamQuiver}

Following Van den Bergh \cite[\S 6.3]{VdB1}, we can endow the path algebras of quivers with a Hamiltonian algebra structure.  More precisely, given a quiver $Q$, we define on $\kk \bar{Q}$ the double Poisson bracket given by 
\begin{equation} \label{Eq:VdBHam1}
 \dgal{a,a^\ast}=e_{h(a)}\otimes e_{t(a)}\,\,\,\, \text{for } a \in Q\,, \quad 
  \dgal{a,a^\ast}=-e_{h(a)}\otimes e_{t(a)}\,\,\,\, \text{for } a \in \bar{Q}\setminus Q\,,
\end{equation}
and such that $\dgal{a,b}=0$ if $a\in \bar{Q}$ and $b\in\bar{Q}\setminus\{a,a^\ast\}$. The algebra admits the moment map 
\begin{equation} \label{Eq:VdBHam2}
 \mu=\sum_{a \in Q} (a a^\ast-a^\ast a)\quad \text{or}\quad 
 \mu=\sum_{s\in I}\mu_s,\quad \text{for } \mu_s=\sum_{\substack{a\in Q\\t(a)=s}}a a^\ast - \sum_{\substack{a\in Q\\h(a)=s}} a^\ast a\,.
\end{equation}
This was considered in a simple case in Example \ref{Exmp:Q1}. 

\begin{thm} \label{Thm:HQuivers}
 Up to isomorphism, the Hamiltonian algebra $(\kk \bar{Q},\dgal{-,-},\mu)$ only depends on $Q$ seen as an undirected graph. 
\end{thm}
\begin{proof}
 As in the proof of \cite[Theorem 6.7.1]{VdB1}, we begin with the `separated' quiver $Q^{sep}$ which has vertex and arrow sets given by 
\begin{equation}
  I^{sep}=\{v_b,\,v_{b^\ast}\mid b \in Q\}\,, \quad 
Q^{sep}=\{b:v_b \to v_{b^\ast} \mid b \in Q \}\,.
\end{equation}
We form the double $\bar{Q}^{sep}$ of $Q^{sep}$, which amounts to add the arrows $\{b^\ast: v_{b^\ast} \to v_b \mid b \in Q \}$. We define on it the involution $\ast$ given by $b \mapsto b^\ast$ and $b^\ast \mapsto b$.   By combining Example \ref{Exmp:Q1} (with $a=b$ for each $b\in Q^{sep}$) and Lemma \ref{Lem:SumHam}, $A^{sep}=\kk \bar{Q}^{sep}$ is Hamiltonian for the double Poisson bracket given by 
\begin{equation} \label{Eq:BasicBr}
  \dgal{b,b^\ast}=e_{v_{b^\ast}}\otimes e_{v_b}\,,
\end{equation}
for all $b\in Q^{sep}$ and which is zero on each other pair of generators, while  the moment map is  defined as 
\begin{equation}
  \mu=\sum_{b \in \bar{Q}^{sep}} \mu_{v_b}\,,
\quad \mu_{v_b}=bb^\ast \text{ if }b \in Q^{sep},\,\,\text{or }\mu_{v_b}=-bb^\ast \text{ if }b \in \bar{Q}^{sep} \setminus Q^{sep} \,. 
\end{equation}
To get a Hamiltonian structure on $\kk \bar{Q}$, it remains to fuse all these disjoint quivers of $\bar{Q}^{sep}$ to form $\bar{Q}$. We can easily see that the Hamiltonian algebra structure induced by fusion coincides with \eqref{Eq:VdBHam1}--\eqref{Eq:VdBHam2}. The independence of order in which we identify the vertices is obvious from the notations, and in fact it follows from Theorem \ref{Thm:IsoFusHam}. 

It remains to see that this structure is independent of the directions of the arrows in $Q$, up to isomorphism. To do so, we first use the isomorphism from Example \ref{Exmp:Q1} to  get that we can reverse the direction of any arrow $b\in Q^{sep}$ and obtain a Hamiltonian algebra isomorphic to $A^{sep}$. But Lemma \ref{L:IsoPhiHam}  guarantees that fusion preserves isomorphisms of Hamiltonian algebras, so we are done.  
\end{proof}

\section{Morphisms of double quasi-Poisson algebras}  \label{S:qHam}

As in Section \ref{S:Base}, all algebras are over $B=\bigoplus_{s\in I}\kk e_s$. With the exception of Remark  \ref{Rem:qHamMod}, all the statements from this section have a counterpart in Section \ref{S:Ham}.

\subsection{Basic definitions and properties}

\subsubsection{Definitions} Following Van den Bergh \cite{VdB1}, we say that a  \emph{double quasi-Poisson bracket} $\dgal{-,-}$ on $A$ is a double bracket for which  the associated triple bracket $\dgal{-,-,-}$ defined by \eqref{Eq:TripBr}  satisfies 
\begin{equation}
   \begin{aligned} \label{qPabc}
    \dgal{a,b,c}=&\frac14 \sum_{s\in I} \Big(
c e_s a \otimes e_s b \otimes e_s  - c e_s a \otimes e_s \otimes b e_s - c e_s \otimes a e_s b \otimes e_s 
+ c e_s \otimes a e_s \otimes b e_s \\
&\qquad \quad - e_s a \otimes e_s b \otimes e_s c + e_s a \otimes e_s \otimes b e_s c + e_s \otimes a e_s b \otimes e_s c - e_s \otimes a e_s \otimes b e_s c \Big)\,,
  \end{aligned}
\end{equation}
on any $a,b,c\in A$. In such a case, we say that $(A,\dgal{-,-})$ (or simply $A$) is a \emph{double quasi-Poisson algebra}. 
An element $\Phi_s\in e_s A e_s$ is said to satisfy the \emph{multiplicative property for} $e_s$ if, for any $a\in A$, we have  
\begin{equation} \label{Phim}
 \dgal{\Phi_s,a}=\frac12 (ae_s\otimes \Phi_s-e_s \otimes \Phi_s a +  a \Phi_s \otimes e_s-\Phi_s \otimes e_s a)\,.
\end{equation}
Then, we define a \emph{multiplicative moment map} as an invertible element $\Phi=\sum_{s\in I}\Phi_s$ for $\Phi_s\in e_sAe_s$ which is such that for all $s\in I$, $\Phi_s$ satisfies the multiplicative property for $e_s$ \eqref{Phim}. 
The decomposition of $\Phi$ yields $e_s\Phi e_t=\delta_{st}\Phi_s$ for each $s,t\in I$,  
hence the invertibility condition implies that $\Phi_s^{-1}:=e_s \Phi^{-1}e_s$ is an inverse for $\Phi_s$ in $e_sA e_s$.
We call the triple $(A,\dgal{-,-},\Phi)$ (or simply $A$ if no confusion can arise) a \emph{quasi-Hamiltonian algebra}.

Let $A_1,A_2$ be  $B$-algebras with $B$-linear double brackets $\dgalU{-,-},\dgalD{-,-}$.
If $\psi:A_1 \to A_2$ is a morphism of double brackets and $A_1,A_2$ are double quasi-Poisson algebras, we say that $\psi$ is a \emph{morphism of double quasi-Poisson algebras}. 
If furthermore $A_1,A_2$ are quasi-Hamiltonian algebras with respective multiplicative moment maps $\Phi_1,\Phi_2$ satisfying $\psi(\Phi_1)=\Phi_2$, we say that $\psi$ is a \emph{morphism of quasi-Hamiltonian algebras}.

\begin{exmp} \label{Exmp:Q1qHam}
Consider the quiver $Q$ with vertices $\{1,2\}$ and one arrow $a: 1 \to 2$. Let $\bar{Q}$ be its double obtained by adding the arrow $a^\ast:2\to 1$. Put $B=\kk e_1 \oplus \kk e_2$.
Following Van den Bergh \cite[\S 6.5]{VdB1}, the path algebra $\kk \bar{Q}$ is a double quasi-Poisson algebra for the $B$-linear double bracket  given on generators by 
\begin{equation} \label{Eq:qHamQ1}
\begin{aligned}
 &\dgal{a,a}_1=0\,, \quad \dgal{a^\ast,a^\ast}_1=0\,, \quad  \dgal{a,a^\ast}_1=e_2 \otimes e_1+\frac12 (a^\ast a \otimes e_1 + e_2 \otimes a a^\ast)\,.
\end{aligned}
\end{equation}
(We get $\dgal{a^\ast,a}_1$ by cyclic antisymmetry \eqref{Eq:cycanti}.) 
To have a quasi-Hamiltonian algebra structure, we work in the localised algebra $A_1=(\kk \bar{Q})_{S_1}$ for $S_1=\{1+aa^\ast,1+a^\ast a\}$, where we introduce the multiplicative moment map $\Phi=(1+a a^\ast)(1+a^\ast a)^{-1}$. It can be decomposed as $\Phi=\Phi_1+\Phi_2$ for $\Phi_1=e_1+aa^\ast$ and  $\Phi_2=(e_2+a^\ast a)^{-1}:=e_2(1+a^\ast a)^{-1}e_2$. 

Similarly, consider the quiver $Q^{op}$ with vertices $\{1,2\}$ and one arrow $b:2 \to 1$, and let $\bar{Q}^{op}$ be its double with new arrow $b^\ast:1 \to 2$. We can also use Van den Bergh's quasi-Hamiltonian structure on $A_2=(\kk \bar{Q}^{op})_{S_2}$, $S_2=\{1+bb^\ast,1+b^\ast b\}$, which is given by the $B$-linear double bracket 
\begin{equation} \label{Eq:qHamQ1op}
\begin{aligned}
 &\dgal{b,b}_2=0\,, \quad \dgal{b^\ast,b^\ast}_2=0\,, \quad \dgal{b,b^\ast}_2=e_1 \otimes e_2+\frac12(b^\ast b \otimes e_2 + e_1 \otimes b b^\ast)\,.  
\end{aligned}
\end{equation}
The multiplicative moment map $\Phi'=(1+bb^\ast)(1+b^\ast b)^{-1}$ can be decomposed as $\Phi'=(e_1+b^\ast b)^{-1}+(e_2+bb^\ast)$ as above.

Following Crawley-Boevey and Shaw \cite[Section 2]{CBShaw},  we introduce the isomorphism of $B$-algebras $\psi:A_1 \to A_2$ given by 
\begin{equation} \label{Eq:psiOp}
 \psi(a)=b^\ast\,, \quad \psi(a^\ast)=-(1+bb^\ast)^{-1}b\,.
\end{equation}
It is indeed a morphism since 
\begin{equation} \label{Eq:qHamQ1mom}
 \psi(1+a^\ast a)=1-(1+bb^\ast)^{-1}bb^\ast=(1+bb^\ast)^{-1}\,, \quad 
 \psi(1+a a^\ast)=1-b^\ast(1+bb^\ast)^{-1}b=(1+b^\ast b)^{-1}\,,
\end{equation}
are both invertible.  It is an isomorphism since the map 
$\theta:A_2 \to A_1$ given by 
\begin{equation*}
 \theta(b)=-a^\ast (1+a a^\ast)^{-1}\,, \quad \theta(b^\ast)=a\,,
\end{equation*}
is its inverse. We have, in fact, that $\psi: A_1 \to A_2$ is an isomorphism of quasi-Hamiltonian algebras. Using \eqref{Eq:qHamQ1mom}, we easily see that $\psi(\Phi_1)=\Phi'_1$ and $\psi(\Phi_2)=\Phi'_2$, so that we only need to show that $\psi$ is a morphism of double brackets. By \eqref{Eq:Morph}, this would follow from the identities 
\begin{equation} \label{Eq:qHamQ1dbr}
\begin{aligned}
 &\psi^{\otimes 2} \dgal{a,a}_1=\dgal{b^\ast,b^\ast}_2\,, \quad 
  \psi^{\otimes 2} \dgal{a^\ast,a^\ast}_1=\dgal{(e_2+bb^\ast)^{-1}b,(e_2+bb^\ast)^{-1}b}_2\,, \\
 &\psi^{\otimes 2} \dgal{a,a^\ast}_1=-\dgal{b^\ast,(e_2+bb^\ast)^{-1}b}_2\,.
\end{aligned}
\end{equation}
The first equality in \eqref{Eq:qHamQ1dbr} is trivial as both sides vanish. We leave the proof of the second identity in \eqref{Eq:qHamQ1dbr} to the reader since it is similar to the third one which we check now. 
Combining \eqref{Eq:qHamQ1} and the definition of $\psi$ \eqref{Eq:psiOp}, we have that 
\begin{equation*}
\begin{aligned}
  \psi^{\otimes 2} \dgal{a,a^\ast}_1=& e_2 \otimes e_1 - \frac12 (e_2+bb^\ast)^{-1}bb^\ast \otimes e_1 - \frac12 e_2 \otimes b^\ast (1+bb^\ast)^{-1}b \\
  =&(e_2+bb^\ast)^{-1}\otimes e_1+\frac12 (e_2+bb^\ast)^{-1}bb^\ast \otimes e_1 - \frac12 e_2 \otimes b^\ast (e_2+bb^\ast)^{-1}b\,.
\end{aligned}
\end{equation*}
Since $\Phi'$ is a multiplicative moment map, we have from the multiplicative property \eqref{Phim} for $e_2$  that 
\begin{equation*}
 \begin{aligned}
  \dgal{b^\ast,(e_2+bb^\ast)^{-1}}_2=&-(\Phi'_2)^{-1}\dgal{b^\ast,\Phi'_2} (\Phi'_2)^{-1}=+(\Phi'_2)^{-1}( \dgal{\Phi'_2,b^\ast})^\circ (\Phi'_2)^{-1} \\
  =&\frac12(\Phi'_2)^{-1}(\Phi'_2 \otimes b^\ast e_2-\Phi'_2 b^\ast \otimes e_2 + e_2 \otimes b^\ast \Phi'_2- e_2 b^\ast \otimes \Phi'_2) (\Phi'_2)^{-1}\\
  =&\frac12(e_2 \otimes b^\ast (\Phi'_2)^{-1} +(\Phi'_2)^{-1} \otimes b^\ast)  \,,
 \end{aligned}
\end{equation*}
where $(\Phi_2')^{-1}:=e_2 (\Phi')^{-1}e_2=(e_2+bb^\ast)^{-1}$ and we used that $b^\ast \in e_1 A_2 e_2$. We thus get 
\begin{equation*}
 \begin{aligned}
-\dgal{b^\ast,(e_2+bb^\ast)^{-1}b}_2=&-(\Phi_2')^{-1}\dgal{b^\ast,b}'-\dgal{b^\ast,(\Phi_2')^{-1}}'b \\
=&(\Phi_2')^{-1}\otimes e_1+\frac12 (\Phi_2')^{-1}bb^\ast \otimes e_1 - \frac12 e_2 \otimes b^\ast (\Phi_2')^{-1} b\,,
 \end{aligned}
\end{equation*}
which coincides with $\psi^{\otimes 2} \dgal{a,a^\ast}_1$. 
\end{exmp}

\subsubsection{First properties} 
The following result can be proved in the same way as Lemma \ref{L:surjHam}.
\begin{lem} \label{L:surjqHam}
Let  $\psi:A_1 \to A_2$ be a morphism of double brackets.

\noindent 1. Assume that $\psi$ is surjective as a $B$-algebra homomorphism. 

 1.1. If $A_1$ is a double quasi-Poisson algebra, then $A_2$ is a double quasi-Poisson algebra. 
 
 1.2. If $A_1$ is a quasi-Hamiltonian algebra, then $A_2$ admits a structure of quasi-Hamiltonian algebra such that $\psi$ is a morphism of quasi-Hamiltonian algebras. 
 
\noindent 2. Assume that $\psi$ is injective as a $B$-algebra homomorphism. 

 2.1. If $A_2$ is a double quasi-Poisson algebra, then $A_1$ is a double quasi-Poisson algebra. 
 
 2.2. If $A_2$ is a quasi-Hamiltonian algebra with multiplicative moment map $\Phi'$ such that $\Phi',(\Phi')^{-1}\in \operatorname{Im}(\psi)$, then $A_1$ admits a structure of quasi-Hamiltonian algebra such that $\psi$ is a morphism of quasi-Hamiltonian algebras. 
 
\noindent 3. All the structures obtained in 1.--2. are unique.  
\end{lem}

In analogy with the Hamiltonian case, a morphism of double brackets between two quasi-Hamiltonian algebras may fail to be a morphism of quasi-Hamiltonian algebras: we can rescale multiplicative moment maps. In fact, such a rescaling always exists if the morphism of double brackets is an isomorphism.
\begin{lem}
 Let $\psi:A_1 \to A_2$ be an isomorphism of double quasi-Poisson algebras. If $A_1,A_2$ are quasi-Hamiltonian algebras with multiplicative moment maps $\Phi_1,\Phi_2$, then $\Phi_2=\nu\, \psi(\Phi_1)$ for some $\nu \in B^\times$. 
In particular, a multiplicative moment map is unique up to multiplication by an element of $B^\times$.
\end{lem}
\begin{proof}
Let us decompose $\Phi_{1}\in A_1$ with the idempotents as $\sum_{s\in I} \Phi_{1,s}$ and do the same for $\Phi_2\in A_2$. 
 We first note that, for any $\lambda_s\in \kk$, the element $\Phi_{2,s}-\lambda_s\psi(\Phi_{1,s})$ satisfies the multiplicative property for $e_s$ \eqref{Phim} because for any $a \in A_2$, 
 \begin{equation*}
\dgalD{\Phi_{2,s}-\lambda_s\psi(\Phi_{1,s}),a}  =\dgalD{\Phi_{2,s},a}-\lambda_s \psi^{\otimes 2}\dgalU{\Phi_{1,s},\psi^{-1}(a)}\,,
 \end{equation*}
and we can use \eqref{Phim} for the two double brackets on the right-hand side since $\Phi_1,\Phi_2$ are multiplicative moment maps. 
For any $\lambda=\sum_s \lambda_s e_s\in B$, introduce the element $\Phi^{(\lambda)}:=\Phi_2-\lambda \Phi_1$. 
Thus, for any $\lambda,\kappa \in B$, we get
\begin{equation}
\begin{aligned}  \label{Eq:Philk}
  \dgalD{\Phi^{(\lambda)},\Phi^{(\kappa)}}
 =&\frac12 \sum_{s\in I}(\Phi^{(\kappa)}_s\otimes \Phi^{(\lambda)}_s-e_s \otimes \Phi^{(\lambda)}_s \Phi^{(\kappa)}_s +  \Phi^{(\kappa)}_s \Phi^{(\lambda)}_s \otimes e_s-\Phi^{(\lambda)}_s \otimes \Phi^{(\kappa)}_s)  \\
 =&\frac12 \sum_{s\in I}(-\Phi^{(\kappa)}_s\otimes \Phi^{(\lambda)}_s+\Phi^{(\kappa)}_s \Phi^{(\lambda)}_s \otimes e_s -e_s \otimes  \Phi^{(\lambda)}_s \Phi^{(\kappa)}_s +\Phi^{(\lambda)}_s \otimes \Phi^{(\kappa)}_s)  \,,
\end{aligned}
\end{equation}
where we used for the first (resp. second) equality that $\Phi^{(\lambda)}_s$ (resp. $\Phi^{(\kappa)}_s$) satisfies the multiplicative property for $e_s$ \eqref{Phim}.  By decomposing \eqref{Eq:Philk} in terms of idempotents, we get that 
$\Phi^{(\kappa)}_s\otimes \Phi^{(\lambda)}_s=\Phi^{(\lambda)}_s \otimes \Phi^{(\kappa)}_s$ for all $s\in I$, which by definition of these elements is equivalent to 
\begin{equation*}
 (\lambda_s-\kappa_s)\left( \psi(\Phi_{1,s}) \otimes \Phi_{2,s} - \Phi_{2,s}\otimes \psi(\Phi_{1,s})\right)=0\,.
\end{equation*}
Multiplying on both side with $\Phi_{2,s}^{-1}:=e_s \Phi_2^{-1} e_s$, this is equivalent when $\lambda-\kappa \in B^\times$ to 
\begin{equation*}
\Phi_{2,s}^{-1}\psi(\Phi_{1,s}) \otimes e_s =e_s\otimes \psi(\Phi_{1,s})\Phi_{2,s}^{-1}\in e_s A_2e_s \otimes e_s A_2 e_s\,. 
\end{equation*}
Thus, $\Phi_{2,s}=\nu_s\psi(\Phi_{1,s})$ for some $\nu_s \in \kk^\times$. 

For the second part of the statement, it suffices to consider $\psi=\id_{A_1}:A_1\to A_1$. 
\end{proof}

\subsubsection{Sum of double quasi-Poisson algebras}

\begin{lem}[\cite{VdB1}] \label{Lem:SumqHam}
 If $A_1$ and $A_2$ are endowed with double quasi-Poisson brackets, then the double bracket $\dgal{-,-}^\oplus$ on $A_1\oplus A_2$ defined in \ref{ssOplus}  is quasi-Poisson. If  $\Phi_1,\Phi_2$ are multiplicative moment maps in $A_1$ and $A_2$, then $A_1\oplus A_2$ is quasi-Hamiltonian with $(\Phi_1,\Phi_2)$ as multiplicative moment map.
\end{lem}

\subsection{Double quasi-Poisson brackets and fusion} 

We assume that the index set $I$ of $B$ is such that $|I|>1$. 
Our aim is to adapt the results from \ref{ss:HamFus} to the quasi-Hamiltonian setting. The main difficulty that we will encounter is that the double bracket given by \eqref{Eq:BrFus} obtained by fusion from a double quasi-Poisson algebra \emph{is not} a double quasi-Poisson bracket in general, because the associated triple bracket does not satisfy \eqref{qPabc}. To overcome this issue, one needs to add a corrective term to the double bracket after fusion, see Proposition \ref{Pr:IsoFusqHam}.

\subsubsection{Fusion in an algebra}
Recall the fusion algebra $A^f=A^f_{e_2 \to e_1}$ defined in \ref{ssFus}. The following Proposition  was first observed by Van den Bergh under mild assumptions \cite[Theorems 5.3.1, 5.3.2]{VdB1} and is proved in full generalities in \cite[Theorems 2.14, 2.15]{F19}.

\begin{prop} \label{Pr:IsoFusqHam}
If $A$ is a double quasi-Poisson algebra over $B$, then $A^f$ is a double quasi-Poisson algebra over $B^f=\bigoplus_{s\in I \setminus \{2\}}e_s$. The double quasi-Poisson bracket in $A^f$ is  given by 
\begin{equation} \label{dgalf}
  \dgal{-,-}^f:= \dgal{-,-} _{ind} + \dgal{-,-}_{fus}\,, 
\end{equation}
where  the first double bracket on the right-hand side is induced in $A^f$ by the one of $A$ (see \eqref{Eq:BrFus}), and the second double bracket $\dgal{-,-}_{fus}$ is defined in Appendix \ref{App:Dbr}. 
Furthermore, if $\Phi$ is a multiplicative moment map for $A$ and  $\Phi^f_s$ denotes the projection of the element $\Phi_s=e_s \Phi e_s$ under the map \eqref{Eq:prfus} for each $s\in I$, then $\Phi^{f\!f}=\Phi^f_1 \Phi_2^f+\sum_{s\neq 1,2}\Phi_s^f$ is a multiplicative moment map for $A^f$.  
\end{prop}

For the remainder of this section, if $A$ is a double quasi-Poisson algebra and $\tilde A$ is an algebra obtained from $A$ by performing a finite number of fusions, we will see $\tilde A$ as a double quasi-Poisson algebra  using this last result. More precisely, if we define $\tilde A$ using the following chain of algebras obtained by fusion 
\begin{equation} \label{Eq:qChain}
  \begin{aligned}
 &A^{(0)} \longrightarrow A^{(1)} \longrightarrow \ldots \longrightarrow
 A^{(k)} \longrightarrow \ldots\longrightarrow A^{(n)}\,, \quad \text{ where } \\
&A^{(0)}=A,\quad A^{(n)}=\tilde A, \quad  A^{(k)}= (A^{(k-1)})^f_{e_{j_k}\to e_{i_k}} \,\, \text{ for }k=1,\ldots,n\,,
  \end{aligned}
\end{equation}
each algebra $A^{(k)}$ in the chain \eqref{Eq:qChain} is a double quasi-Poisson algebra whose double quasi-Poisson bracket is obtained from $A^{(k-1)}$ by using Proposition \ref{Pr:IsoFusqHam}. 
If $A$ is a quasi-Hamiltonian algebra, we see $\tilde A$ as a quasi-Hamiltonian algebra using the same argument.  

\begin{exmp}
Consider the localised path algebra $A=\kk \bar{Q}_S$ considered in Example \ref{Exmp:Q1qHam} with its quasi-Hamiltonian algebra structure. Fusing $e_2$ onto $e_1$, we get a quasi-Hamiltonian algebra structure on $A^f$ which can be identified with the localised free algebra $\kk \langle a,a^\ast\rangle_S$, $S=\{1+aa^\ast,1+a^\ast a\}$. The double bracket is determined by 
\begin{equation} 
\begin{aligned} \label{Eq:qHamQ1fus}
 &\dgal{a,a}^f=\frac12 (a^2 \otimes 1 - 1 \otimes a^2)\,, \quad \dgal{a^\ast,a^\ast}^f=-\frac12((a^\ast)^2 \otimes 1 - 1 \otimes (a^\ast)^2)\,, \\
 &\dgal{a,a^\ast}^f=e_2 \otimes e_1+\frac12 (a^\ast a \otimes e_1 + e_2 \otimes a a^\ast +a^\ast \otimes a - a \otimes a^\ast)\,, 
\end{aligned}
\end{equation}
while the multiplicative moment map is $\Phi=(1+aa^\ast)(1+a^\ast a)^{-1}$. 
Indeed, to get the double brackets  \eqref{Eq:qHamQ1fus} we add to \eqref{Eq:qHamQ1} the terms given by $\dgal{-,-}_{fus}$, so we add \eqref{vv} to $\dgal{a,a}$,  \eqref{uu} to $\dgal{a^\ast,a^\ast}$, and  \eqref{vu} to $\dgal{a,a^\ast}$. 
The algebra $A^f$ is the localised path algebra of the double of the one-loop quiver, and the quasi-Hamiltonian structure obtained by fusion is the one of Van den Bergh \cite[\S 6.5]{VdB1}. 
\end{exmp}

\begin{lem} \label{L:IsoFusqHam0}
Let $(A,\dgal{-,-})$ be a double quasi-Poisson algebra over $B$. Consider the algebra $A_1=A^f_{e_2\to e_1}$ obtained by fusing $e_2$ onto $e_1$ and the algebra $A_2=A^f_{e_1 \to e_2}$ obtained by fusing $e_1$ onto $e_2$, which are endowed with a double quasi-Poisson bracket by Proposition \ref{Pr:IsoFusqHam}. 
  
Assume that there exists an element $\Phi_2\in e_2 A e_2$ invertible in $e_2 A e_2$ which satisfies the multiplicative property \eqref{Phim} for $e_2$. 
Then there exists an isomorphism of double quasi-Poisson algebras $A_1\to A_2$. If furthermore $A$ is quasi-Hamiltonian, then the isomorphism $A_1\to A_2$ is an isomorphism of quasi-Hamiltonian algebras. 
\end{lem}

The last part of this statement was announced in \cite[Proposition 3.8]{F19} as a non-commutative version of \cite[Proposition 5.7]{QuasiP}.

\begin{proof}[Proof (of Lemma \ref{L:IsoFusqHam0}).]
We first define an isomorphism $\psi:A_1 \to A_2$. (This is a morphism over the common base $B'=\kk \tilde e \oplus\bigoplus_{s\in I \setminus \{1,2\}}\kk e_s$, see the remark after Lemma \ref{L:IsoFusHam0}.) 
As in Lemma \ref{L:IsoId}, we define $\psi$ on a specialisation of the set of  generators of $A_1$, i.e. on a suitably chosen subset of the generators given in Lemma \ref{AfGenerat}. For $\he=1-e_1-e_2$, $\psi$ is given on generators of first type \eqref{type1} as 
 \begin{equation*}
 \begin{aligned}
  &\psi(t)=t,\,\,\text{if } t \in \he A \he; \qquad 
  \psi(t)=\Phi_2 e_{21}t,\,\,\text{if } t \in e_1 A \he; \\
  &\psi(t)=te_{12}\Phi_2^{-1},\,\,\text{if } t \in \he A e_1; \quad 
  \psi(t)=\Phi_2e_{21}te_{12} \Phi_2^{-1},\,\,\text{if } t \in e_1 A e_1; 
 \end{aligned}
 \end{equation*}
on generators of second type  \eqref{type2} as 
 \begin{equation*}
  \psi(e_{12}u)=u,\,\,\text{if } u \in e_2 A \he; \quad 
  \psi(e_{12}u)=ue_{12}\Phi_2^{-1},\,\,\text{if } u \in e_2 A e_1; 
 \end{equation*}
on generators of third type  \eqref{type3} as  
  \begin{equation*}
  \psi(ve_{21})=v,\,\,\text{if } v \in \he A e_2; \quad 
  \psi(ve_{21})=\Phi_2 e_{21}v,\,\,\text{if } v \in e_1 A e_2; 
 \end{equation*}
 on generators of fourth type \eqref{type4} as
   \begin{equation*}
  \psi(e_{12}we_{21})=w,\,\,\text{if } w \in e_2 A e_2.
 \end{equation*}
Similarly, we define a morphism $\theta:A_2\to A_1$ on a specialisation of the set of  generators of $A_2$, which is given on generators of first type \eqref{type1} as 
 \begin{equation*}
 \begin{aligned}
  &\theta(t')=t',\,\,\text{if } t' \in \he A \he; \qquad 
  \theta(t')= e_{12}t',\,\,\text{if } t' \in e_2 A \he; \\
  &\theta(t')=t'e_{21},\,\,\text{if } t' \in \he A e_2; \quad 
  \theta(t')=e_{12}t' e_{21},\,\,\text{if } t' \in e_2 A e_2; 
 \end{aligned}
 \end{equation*}
on generators of second type  \eqref{type2} as 
 \begin{equation*}
  \theta(e_{21}u')=e_{12}\Phi^{-1}_2 e_{21}u',\,\,\text{if } u' \in e_1 A \he; \quad 
  \theta(e_{21}u')=e_{12}\Phi^{-1}_2 e_{21} u'e_{21},\,\,\text{if } u' \in e_1 A e_2; 
 \end{equation*}
on generators of third type  \eqref{type3} as  
  \begin{equation*}
  \theta(v'e_{12})=v' e_{12}\Phi_2 e_{21},\,\,\text{if } v' \in \he A e_1; \quad 
  \theta(v'e_{12})=e_{12}v'e_{12}\Phi_2 e_{21},\,\,\text{if } v' \in e_2 A e_1; 
 \end{equation*}
 on generators of fourth type \eqref{type4} as
   \begin{equation*}
  \theta(e_{21}w'e_{12})=e_{12}\Phi^{-1}_2 e_{21} w'e_{12}\Phi_2 e_{21},\,\,\text{if } w' \in e_1 A e_1.
 \end{equation*}
It is then a straightforward exercise to show that $\psi \circ \theta = \Id_{A_2}$ and $\theta \circ \psi = \Id_{A_1}$. For example, 
\begin{equation*}
 \psi \circ \theta (e_{21}w'e_{12})=\psi(e_{12}\Phi^{-1}_2 e_{21}) \psi( w') \psi(e_{12}\Phi_2 e_{21})
 =\Phi^{-1}_2(\Phi_2e_{21}w'e_{12} \Phi_2^{-1}) \Phi_2=e_{21}w'e_{12}\,, \quad w' \in e_1 A e_1\,.
\end{equation*}
We easily see that in the quasi-Hamiltonian case we get $\psi(\Phi_1e_{12}\Phi_2e_{21})=\Phi_2e_{21}\Phi_1 e_{12}$, so the only property that remains to be shown is that $\psi:A_1 \to A_2$ is a morphism of double quasi-Poisson algebras. 

We  note from Proposition \ref{Pr:IsoFusqHam} that the double quasi-Poisson bracket on $A_1$ is given by 
\begin{equation} \label{dgalU}
  \dgalU{-,-}:= \dgalindU{-,-} + \dgalfusU{-,-}\,, 
\end{equation}
where  $\dgalindU{-,-}$ is induced by the double bracket $\dgal{-,-}$ in $A_1$ using \eqref{Eq:BrFus}, while $\dgalfusU{-,-}$ is the double bracket defined in Appendix \ref{App:Dbr} for $i=1, j=2$; 
 the double quasi-Poisson bracket on $A_2$ is given by 
\begin{equation} \label{dgalD}
  \dgalD{-,-}:= \dgalindD{-,-} + \dgalfusD{-,-}\,, 
\end{equation}
 where  $\dgalindD{-,-}$ is induced by the double bracket $\dgal{-,-}$ in $A_2$ using \eqref{Eq:BrFus}, while $\dgalfusD{-,-}$ is the double bracket defined in Appendix \ref{App:Dbr} for $i=2, j=1$. 
Hence, we need to check that 
\begin{equation} \label{Eq:FusqHam}
(\psi\otimes \psi)\dgalU{c,d}  
= \dgalD{\psi(c),\psi(d)}\,,       
\end{equation}
on each pair $(c,d)$ of specialisations of generators of $A_1$ described above. There are $9$ such specialisations, so making use of the cyclic antisymmetry we need to check \eqref{Eq:FusqHam} in $45$ cases. Let us explain how to carry out the computations in one case; the remaining cases are treated in a similar way, and the corresponding double brackets that must be computed are gathered in Appendix \ref{App:LIsoFusqHam0} for the reader's convenience.
 
Consider $c=a$, $d=b$, where $a,b\in e_1A\he$. We have that $c,d$ are generators of first type \eqref{type1} in $A_1$, so that $\dgalfusU{c,d}=0$ by \eqref{tt} (for $i=1, j=2$).  Also, we have by \eqref{Eq:BrFus} that 
\begin{equation*}
 \dgalindU{c,d}=e_1 \ast e_1\dgal{a,b} \he \ast \he = \dgal{a,b}\,,
\end{equation*}
since $\dgal{a,b}'=e_1 \dgal{a,b}' \he$ and the same holds for $\dgal{a,b}''$. Hence, by definition of \eqref{dgalU} we can simply write the left-hand side of \eqref{Eq:FusqHam} as  $\psi^{\otimes 2}\dgal{a,b}$. Meanwhile, we have that $\psi(c)=\Phi_2 e_{21}a$ and $\psi(d)=\Phi_2 e_{21}b$. Therefore, we get 
\begin{equation}
\begin{aligned} \label{Eq:FusComp1}
  \dgalfusD{\psi(c),\psi(d)} =&
 + \dgalfusD{\Phi_2, \Phi_2 }e_{21}b \ast e_{21}a  +\Phi_2 \ast \dgalfusD{e_{21}a, \Phi_2 }e_{21}b \\
& +\Phi_2 \dgalfusD{\Phi_2, e_{21}b} \ast e_{21}a  +\Phi_2 \ast \Phi_2 \dgalfusD{e_{21}a, e_{21}b}\,,
\end{aligned} 
\end{equation}
using the derivation rules of the double bracket $\dgalfusD{-,-}$. We now use the explicit form of this double bracket given in Appendix \ref{App:Dbr} (for $i=2, j=1$). To do so, note that in $A_2$, $\Phi_2$ is a generator of first type \eqref{type1} while $e_{21}a,e_{21}b$ are generators of second type \eqref{type2}. Hence the double bracket appearing in the first term of \eqref{Eq:FusComp1} is given by \eqref{tt}, the one in the second term by \eqref{ut}, the one in the third term by \eqref{tu}, and the one in the fourth term by \eqref{uu}. Therefore, 
\begin{equation}
\begin{aligned} \label{Eq:FusComp2}
  \dgalfusD{\psi(c),\psi(d)} =&
 +\frac12 \Phi_2 \ast (e_{21}a \otimes \Phi_2 - \Phi_2 e_{21}a \otimes e_2)e_{21}b \\
& +\frac12\Phi_2 (e_2 \otimes \Phi_2 e_{21}b - \Phi_2 \otimes e_{21}b) \ast e_{21}a  \\
&+\frac12 \Phi_2 \ast \Phi_2 (e_2 \otimes e_{21}a e_{21}b - e_{21}b e_{21}a \otimes e_2)\,.
\end{aligned} 
\end{equation}
Noting that $a e_2=0=b e_2$, we get that the last two terms of \eqref{Eq:FusComp2} vanish, and since the second and third terms cancel out this yields 
\begin{equation}\label{Eq:FusComp3}
  \dgalfusD{\psi(c),\psi(d)} = \frac12 (e_{21}a \otimes \Phi_2^2 e_{21}b  - \Phi_2^2  e_{21}a  \otimes e_{21}b)\,.
\end{equation} 
Finally, we compute that 
\begin{equation}
\begin{aligned} \label{Eq:FusComp4}
\dgalindD{\psi(c),\psi(d)} =&
 + \dgalindD{\Phi_2, \Phi_2 }e_{21}b \ast e_{21}a  +\Phi_2 \ast \dgalindD{e_{21}a, \Phi_2 }e_{21}b \\
& +\Phi_2 \dgalindD{\Phi_2, e_{21}b} \ast e_{21}a  +\Phi_2 \ast \Phi_2 \dgalindD{e_{21}a, e_{21}b} \\
=& + \dgal{\Phi_2, \Phi_2 }e_{21}b \ast e_{21}a  +\Phi_2 e_{21}\ast \dgal{a, \Phi_2}e_{21}b \\
& +\Phi_2 e_{21}\dgal{\Phi_2, b} \ast e_{21}a  +\Phi_2 e_{21} \ast \Phi_2e_{21} \dgal{a,b} \,,
\end{aligned} 
\end{equation}
where we have used the derivation rules of the double bracket $\dgalindD{-,-}$ for the first equality, and its definition \eqref{Eq:BrFus} using $\dgal{-,-}$ for the second equality. By assumption, $\Phi_2$ satisfies the multiplicative property \eqref{Phim} for $e_2$ in $A$, hence we can compute the double brackets  in the first three terms of \eqref{Eq:FusComp4}. We get that (since $e_2a=0=ae_2$ and the same holds for $b$, $\dgal{a, \Phi_2}=0=\dgal{\Phi_2,b}$ in this case) 
\begin{equation} \label{Eq:FusComp5}
\dgalindD{\psi(c),\psi(d)} =
 \frac12(\Phi_2^2 e_{21}a \otimes e_{21}b - e_{21}a \otimes \Phi_2^2e_{21}b)+\Phi_2 e_{21} \dgal{a,b}' \otimes  \Phi_2e_{21} \dgal{a,b}'' \,.
\end{equation}
Combining \eqref{Eq:FusComp3} and \eqref{Eq:FusComp5}, we get that 
\begin{equation}
\dgalD{\psi(c),\psi(d)} =\Phi_2 e_{21} \dgal{a,b}' \otimes  \Phi_2e_{21} \dgal{a,b}''= \psi^{\otimes 2} \dgal{a,b} \,,
\end{equation}
by definition of $\psi$, so this is precisely $\psi^{\otimes 2}\dgalU{c,d}$.
\end{proof}

\subsubsection{Fusion of several idempotents}

\begin{lem}   \label{L:IsoFusqHam1}
Let $A$ be a double quasi-Poisson algebra over $B$, and let $e_1,e_2,e_3\in B$ be orthogonal idempotents. 
Let $A_1:= (A^f_{e_3\to e_2})_{e_2 \to e_1}^f$ (resp. $A_2:=(A^f_{e_2 \to e_1})_{e_3 \to e_1}^f$) be the algebra obtained by fusing $e_3$ onto $e_1$, then $e_2$ onto $e_1$ 
(resp. $e_2$ onto $e_1$, then $e_3$ onto $e_1$). 
Then the identity map on $A$ induces an isomorphism of double quasi-Poisson algebras $\psi:A_1\to A_2$ over $B'=\bigoplus_{s\in I \setminus \{2,3\}}\kk e_s$. 
If furthermore $A$ is quasi-Hamiltonian, then $\psi$ is an isomorphism of quasi-Hamiltonian algebras. 
\end{lem}
\begin{proof}
 The map $\psi:A_1 \to A_2$ induced by the identity is explicitly spelled out in the proof of Lemma \ref{L:IsoFusHam1}. 
 In particular,  if  $\dgal{-,-}_{k,ind}$ denotes the double bracket from $A$ induced in $A_k$, $k=1,2$, the above-mentioned proof also implies that 
 \begin{equation} \label{Eq:IsoqHam1}
  \dgal{\psi(c),\psi(d)}_{2,ind}=(\psi\otimes \psi)\dgal{c,d}_{1,ind}\,,\quad \text{for all } c,d\in A_1\,.
 \end{equation}
By construction $A_1$ and $A_2$ are obtained by the following composition of fusions 
\begin{equation} \label{Eq:IsoqHam2}
 A  \,\stackrel{e_3\to e_2}{\longrightarrow}\,  A^f_{e_3\to e_2}  \,\stackrel{e_2\to e_1}{\longrightarrow}\,  A_1\,, \qquad 
  A  \,\stackrel{e_2\to e_1}{\longrightarrow}\,  A^f_{e_2\to e_1}  \,\stackrel{e_3\to e_1}{\longrightarrow}\,  A_2\,. 
\end{equation}
Thus, using Proposition \ref{Pr:IsoFusqHam} twice, the double quasi-Poisson bracket on $A_1$ is given by 
\begin{equation}
 \dgal{-,-}_1:= \dgal{-,-}_{1,ind} + \dgal{-,-}_{1,fus}^{3\to 2} + \dgal{-,-}_{1,fus}^{2\to 1}\,. 
\end{equation}
Here, $\dgal{-,-}_{1,fus}^{3\to 2}$ denotes the double bracket from Appendix \ref{App:Dbr} which is added after the fusion $e_3\to e_2$ to get a double quasi-Poisson bracket on $A^f_{e_3\to e_2}$ and is then induced in $A_1$; $\dgal{-,-}_{1,fus}^{2\to 1}$ denotes the double bracket from Appendix \ref{App:Dbr} which is added after the fusion $e_2\to e_1$ to get a double quasi-Poisson bracket on $A_1$. In the same way, the double quasi-Poisson bracket on $A_2$ is given by 
\begin{equation}
 \dgal{-,-}_2:= \dgal{-,-}_{2,ind} + \dgal{-,-}_{2,fus}^{2\to 1} + \dgal{-,-}_{2,fus}^{3\to 1}\,,
\end{equation}
where the double brackets are defined in analogy to the case of $A_1$. 
In particular, we get from \eqref{Eq:IsoqHam1} that the map $\psi:A_1 \to A_2$ is  an isomorphism of double quasi-Poisson algebras provided that  for all $c,d\in A_1$, 
 \begin{equation} \label{Eq:IsoqHam3}
  \dgal{\psi(c),\psi(d)}_{2,fus}^{2\to 1} + \dgal{\psi(c),\psi(d)}_{2,fus}^{3\to 1}
  =(\psi\otimes \psi)\dgal{c,d}_{1,fus}^{3\to 2} + (\psi\otimes \psi)\dgal{c,d}_{1,fus}^{2\to 1}\,.
 \end{equation}
This is proved in Appendix \ref{App:LIsoFusqHam1}. If $A$ is quasi-Hamiltonian with multiplicative moment map $\Phi=\sum_s \Phi_s$ where $\Phi_s=e_s \Phi_s e_s$, then using Proposition \ref{Pr:IsoFusqHam}  the multiplicative moment map transforms under the morphisms \eqref{Eq:IsoqHam2} as 
\begin{equation*}
\begin{aligned}
  &\sum_s \Phi_s \mapsto e_2 \Phi_2 e_{23}\Phi_3 e_{32}+ \sum_{s\neq 2,3}\Phi_s 
 \mapsto e_1 \Phi_1 e_{12}\Phi_2 e_{23}\Phi_3 e_{32}e_{21}+ \sum_{s\neq 1,2,3}\Phi_s \in A_1 \,, \\
 &\sum_s \Phi_s \mapsto e_1 \Phi_1 e_{12}\Phi_2 e_{21}+ \sum_{s\neq 1,2}\Phi_s 
 \mapsto e_1 \Phi_1 e_{12}\Phi_2 e_{21} e_{13}\Phi_3 e_{31}+ \sum_{s\neq 1,2,3}\Phi_s \in A_2\,.
\end{aligned}
\end{equation*}
We can see that $\psi$ sends the multiplicative moment map of $A_1$ to the one of $A_2$. 
\end{proof}

\begin{lem}   \label{L:IsoFusqHam2}
Let $A$ be a double Poisson algebra over $B$, and let $e_1,e_2,e_3,e_4\in B$ be orthogonal idempotents. 
Let $A_1:= (A^f_{e_4\to e_3})_{e_2 \to e_1}^f$ (resp. $A_2:=(A^f_{e_2 \to e_1})_{e_4 \to e_3}^f$) be the algebra obtained by fusing $e_4$ onto $e_3$, then $e_2$ onto $e_1$ 
(resp. $e_2$ onto $e_1$, then $e_4$ onto $e_3$). 
Then the identity map on $A$ induces an isomorphism of double quasi-Poisson algebras $\psi:A_1\to A_2$ over $B'=\bigoplus_{s\in I \setminus \{2,4\}}\kk e_s$. 
If furthermore $A$ is quasi-Hamiltonian, then $\psi$ is an isomorphism of quasi-Hamiltonian algebras. 
\end{lem}
\begin{proof}
As in Lemma  \ref{L:IsoFusHam2}, we note that we can write generators of $A_1$ and $A_2$ in the form 
\begin{equation} \label{Eq:IsoqHam21}
e_a a f_a\,, \quad \text{ for }  a \in A,\,\, e_a\in \{1-e_4-e_2,e_{12},e_{34}\},\,\, f_a\in \br{1-e_4-e_2,e_{21},e_{43}}\,.
\end{equation}
Let us prove that the morphism $\psi:A_1\to A_2$ that we seek is simply the identity when written in terms of these elements. 
Using Proposition \ref{Pr:IsoFusqHam} twice, the double quasi-Poisson bracket on $A_1$ is given by 
\begin{equation}
 \dgal{-,-}_1:= \dgal{-,-}_{1,ind} + \dgal{-,-}_{1,fus}^{4\to 3} + \dgal{-,-}_{1,fus}^{2\to 1} \,. 
\end{equation}
Here, $\dgal{-,-}_{1,ind}$ denotes the double bracket from $A$ induced in $A_1$; $\dgal{-,-}_{1,fus}^{4\to 3}$ denotes the double bracket from Appendix \ref{App:Dbr} which is added after the fusion $e_4\to e_3$ to get a double quasi-Poisson bracket on $A^f_{e_4\to e_3}$ and is then induced in $A_1$; $\dgal{-,-}_{1,fus}^{2\to 1}$ denotes the double bracket from Appendix \ref{App:Dbr} which is added after the fusion $e_2\to e_1$ to get a double quasi-Poisson bracket on $A_1$. In the same way, we can write 
\begin{equation}
 \dgal{-,-}_2:= \dgal{-,-}_{2,ind} + \dgal{-,-}_{2,fus}^{2\to 1} + \dgal{-,-}_{2,fus}^{4\to 3}  \,. 
\end{equation}
By  Lemma  \ref{L:IsoFusHam2}, we directly have that 
 \begin{equation}
  \dgal{\psi(c),\psi(d)}_{2,ind}=(\psi\otimes \psi)\dgal{c,d}_{1,ind}\,,\quad \text{for all } c,d\in A_1\,,
 \end{equation}
and we note that if we can show for all $c,d\in A_1$ 
 \begin{equation} \label{Eq:IsoqHam22}
  \dgal{\psi(c),\psi(d)}_{2,fus}^{4\to 3}=(\psi\otimes \psi)\dgal{c,d}_{1,fus}^{4\to 3}\,,\quad \text{and }
 \dgal{\psi(c),\psi(d)}_{2,fus}^{2\to 1}=(\psi\otimes \psi)\dgal{c,d}_{1,fus}^{2\to 1}\,,
 \end{equation}
then $\psi:A_1\to A_2$ will be a morphism of double quasi-Poisson algebras. It suffices to prove one of the two equalities in \eqref{Eq:IsoqHam22} since the other follows by symmetry. Checking that such an equality holds is easy on generators of the form \eqref{Eq:IsoqHam21}. 
For example, if we take $c=e_{34}a \he$ and $d=e_{12}be_{43}$ where $\he=1-e_2-e_4$, we have in $A_1$ that 
\begin{equation*}
 \dgal{c,d}_{1,fus}^{2\to 1}= \dgal{(e_{34}a),e_{12}(b e_{43})}_{1,fus}^{2\to 1}= \frac12 \left(e_1 \otimes (e_{34}a) e_{12}(be_{43}) - e_1(e_{34}a)\otimes e_{12}(be_{43}) \right)  = \frac12 \, e_1 \otimes cd\,,
\end{equation*}
after using \eqref{tu} since $e_{34}a \he \in A_1$ is a generator of first type, while  $e_{12}(b e_{43})\in A_1$ is a generator of second type. Meanwhile, we have in $A_2$ that 
\begin{equation*}
 \dgal{\psi(c),\psi(d)}_{2,fus}^{2\to 1}
= e_{34}\ast\dgal{ a,e_{12}b}_{2,fus}^{2\to 1} e_{43}
= \frac12 e_{34}\ast\left(e_1 \otimes a e_{12}b - e_1 a\otimes e_{12}b  \right)e_{43}  
= \frac12 \, e_1 \otimes \psi(c)\psi(d)\,,
\end{equation*}
where we used that the double bracket $\dgal{-,-}_{2,fus}^{2\to 1}$ is induced in $A_2$ from the double bracket in $(A_{e_2\to e_1})^f$ given in Appendix \ref{App:Dbr} with $j=2$, $i=1$. 

If $A$ is quasi-Hamiltonian, the multiplicative moment map in $A_1$ and $A_2$ can be written as 
\begin{equation*} 
\Phi = \Phi_1e_{12}\Phi_2e_{21}+\Phi_3e_{34}\Phi_4e_{43}+\sum_{s\neq 1,2,3,4}\Phi_s\,,
\end{equation*}
so $\psi$ is obviously a morphism of quasi-Hamiltonian algebras.
\end{proof}

We can now derive the quasi-Poisson version of Theorem \ref{Thm:IsoFusHam} by reproducing its proof with Lemmas \ref{L:IsoFusHam0}, \ref{L:IsoFusHam1} and \ref{L:IsoFusHam2} replaced by Lemmas \ref{L:IsoFusqHam0}, \ref{L:IsoFusqHam1} and \ref{L:IsoFusqHam2} respectively. We use the notion of fusion algebra respecting a partition defined before Theorem \ref{Thm:IsoFusHam}. Such an algebra inherits a double quasi-Poisson bracket by repeated use of Proposition \ref{Pr:IsoFusqHam}.

\begin{thm} \label{Thm:IsoFusqHam}
 Let $A$ be a double quasi-Poisson algebra over $B=\bigoplus_{s\in I}\kk e_s$. 
Assume that $A_1,A_2$ are fusion algebras respecting a partition $I=\sqcup_{j\in J} I_j$. 

Let $J'=\{j\in J \mid |I_j|>1\}$, and 
assume that for each $j'\in J'$ and $k\in I_{j'}$ there exists  $\Phi_k\in e_k A e_k$ invertible in $e_k A e_k$ which satisfies the multiplicative property \eqref{Phim} for $e_k$. 
Then, there is a double quasi-Poisson algebra isomorphism $\psi:A_1 \to A_2$. 
If furthermore $A$ is quasi-Hamiltonian, then the morphism  $\psi:A_1 \to A_2$ is an isomorphism of quasi-Hamiltonian algebras. 
\end{thm}
In the statement, the set $J'$ corresponds to the idempotents in $B\subset A$ which are involved in the fusions defining $A_1$ and $A_2$. The extra assumption requiring the existence of $\Phi_k$, $k\in J'$,  is then sufficient to use Lemma \ref{L:IsoFusqHam0}. In particular, given specific $A_1$ and $A_2$, it is possible to construct the isomorphism with the existence of such elements $\Phi_k$ in a proper subset of $J'$, which can be empty as is easily seen from the case of Lemma \ref{L:IsoFusqHam1}.  

\subsubsection{Fusion of morphisms}

\begin{lem}   \label{L:IsoPhiqHam}
 Let $\psi:A_1 \to A_2$ be a morphism of double quasi-Poisson algebras over $B$. Let $A^f_1=(A_1)^f_{e_2 \to e_1}$, $A^f_2=(A_2)^f_{e_2 \to e_1}$, be the fusion algebras with double brackets  obtained by fusion of $e_2$ onto $e_1$. 
 Then $\psi$ induces a morphism of double quasi-Poisson algebras $\psi^f : A_1^f \to A_2^f$.  
 If furthermore $\psi$ is a morphism of quasi-Hamiltonian algebras, then so is $\psi^f$. 
\end{lem}
\begin{proof}
By Proposition \ref{Pr:IsoFusqHam}, the algebra $A^f_k$ with $k=1,2$ has a double quasi-Poisson bracket given by 
\begin{equation} 
  \dgal{-,-}^f_k:= \dgal{-,-}_{k,ind} + \dgal{-,-}_{k,fus}\,. 
\end{equation}
We claim that the map $\psi^f:A_1^f \to A_2^f$ obtained from $\psi$ by \eqref{Eq:IsoFus} 
is the morphism that we seek. 
By Lemma \ref{L:IsoFus}, we have that 
\begin{equation*}
 (\psi^f\otimes \psi^f)\dgal{c,d}_{1,ind}=\dgal{\psi^f(c),\psi^f(d)}_{2,ind}\,, \quad \text{ for all } c,d\in A_1^f\,.
\end{equation*}
Moreover, $\psi^f$ preserves the type of generators in $A_1^f,A_2^f$ as defined in Lemma \ref{AfGenerat}, and the double brackets $\dgal{-,-}_{k,fus}$ are defined in terms of these types of generators. Therefore 
\begin{equation*}
 (\psi^f\otimes \psi^f)\dgal{c,d}_{1,fus}=\dgal{\psi^f(c),\psi^f(d)}_{2,fus}\,, \quad \text{ for all } c,d\in A_1^f\,.
\end{equation*}
Thus $\psi^f:A_1^f \to A_2^f$ is a morphism of double quasi-Poisson algebras.

In the quasi-Hamiltonian case, we can check that $\psi^f$ maps the multiplicative moment map of $A_1^f$ to that of $A_2^f$. Hence, it is a morphism of quasi-Hamiltonian algebras.  
\end{proof}

\subsection{Application: Van den Bergh's quasi-Hamiltonian structure for quivers} \label{ssqHamQuiver} 

We can endow specific localisations of the path algebras of quivers with a quasi-Hamiltonian algebra structure defined by Van den Bergh \cite[\S6.7]{VdB1}. The double bracket was explicitly spelled out first in \cite[Proposition 2.6]{CF}, then in full generalities in \cite[Theorem 3.3]{F19}. To define this structure, we fix a quiver $Q$, and consider the path algebra of its double $\kk \bar{Q}$ following the conventions given in \ref{sss:Quiver}. We consider the algebra $A_Q$ obtained by universal localisation of $\kk \bar{Q}$ from the set $S=\{1+a a^\ast \mid a\in \bar Q\}$. This localisation can be understood as adding local inverses $e_{t(a)}+aa^\ast \in e_{t(a)}A_Q e_{t(a)}$. 

 For each vertex $s\in I$, consider a total ordering $<_s$ on the set $T_s=\{a\in \bar{Q}\mid t(a)=s\}$. This induces an ordering function at the  vertex $s$
\begin{equation*}
o_s(-,-): \bar Q \times \bar Q \to \{-1,0,1\}\,,  
\end{equation*}
which is defined on arrows $a,b\in \bar{Q}$ by $o_s(a,b)=+1$ if $a<_s b$, $o_s(a,b)=-1$ if $b<_s a$, while it is zero otherwise, i.e. if $a=b\in T_s$, if $a \notin T_s$ or if $b \notin T_s$. The algebra $A_Q$ has a double quasi-Poisson bracket defined by 
  \begin{subequations}
        \begin{align}
 \dgal{a,a}\,=\,&\frac{1}{2}o_{t(a)}(a,a^*)\left( a^2\otimes e_{t(a)}- e_{h(a)}\otimes a^2 \right)\,, \qquad \text{for }a\in\bar{Q}\,, \label{loopG}\\
 \dgal{a,a^*}\,=\,&\, e_{h(a)}\otimes e_{t(a)}
 +\frac{1}{2} a^*a\otimes e_{t(a)} +\frac{1}{2} e_{h(a)}\otimes aa^* \nonumber\\
 & +\frac{1}{2}o_{t(a)}(a,a^*)\, (a^*\otimes a-a\otimes a^*)\,,\qquad \qquad \text{for } a\in Q\,, \label{aastG}
        \end{align}
  \end{subequations}
 and for $b,c\in\bar{Q}$ such that $ c\ne b,b^*$ 
 \begin{equation}
  \begin{aligned}
 \dgal{b,c}\,=\,&-\frac{1}{2}o_{t(b)}(b,c)\,(b\otimes c)-\frac{1}{2}o_{h(b)}(b^*,c^*)\,(c\otimes b)
\\ \label{a<bG}
 &+\frac{1}{2}o_{t(b)}(b,c^*)\, cb\otimes e_{t(b)} + \frac{1}{2}o_{h(b)}(b^*,c)\,e_{h(b)}\otimes bc   \,.
  \end{aligned}
 \end{equation}
Moreover, $A_Q$ is quasi-Hamiltonian for the multiplicative moment map 
\begin{equation} \label{EqPhiVdB}
  \Phi=\sum_{s\in I} \Phi_s\,, \quad \Phi_s=\prod_{a\in T_s}^{\longrightarrow}(e_s+ a a^\ast)^{\epsilon(a)}\,.
\end{equation}

\begin{thm} \label{Thm:qHQuivers}
 Up to isomorphism, the quasi-Hamiltonian algebra $(A_Q,\dgal{-,-},\Phi)$ only depends on $Q$ seen as an undirected graph. 
\end{thm}
\begin{proof}
We first prove that the structure is independent of the direction of the arrow in a one-arrow quiver. If $Q$ is the quiver $1\to 2$, the quasi-Hamiltonian algebra structure on $A_Q$ defined above is nothing else than the structure defined on $A_1$ in Example \ref{Exmp:Q1qHam}. The same is true for the opposite quiver: the quasi-Hamiltonian structure on $A_{Q^{op}}$ for $Q^{op}$ given by $2\to 1$ is the structure on $A_2$ given in Example \ref{Exmp:Q1qHam}. In that same example, it was proved that there exists an isomorphism $\psi:A_1\to A_2$ of quasi-Hamiltonian algebras, thus the structures on $A_Q$ and $A_{Q^{op}}$ are isomorphic, as claimed.  

Consider now an arbitrary quiver $Q$. 
 As in the proof of Theorem \ref{Thm:HQuivers}, we note that $\bar Q$ can be obtained by identifying vertices in the double $\bar Q^{sep}$ of the  separated quiver $Q^{sep}$. 
The algebra $A_{Q^{sep}}$ is quasi-Hamiltonian by combining Lemma \ref{Lem:SumqHam} and the case of a one-arrow quiver described at the beginning of this proof.
In particular, the double quasi-Poisson bracket is given by 
\begin{equation} \label{Eq:BasicqBr}
  \dgal{b,b^\ast}=e_{v_{b^\ast}}\otimes e_{v_b}+\frac12 (b b^\ast \otimes e_{v_b} +  e_{v_{b^\ast}}\otimes b^\ast b)\,,
\end{equation}
for all $b\in Q^{sep}$ and it is zero on every other pair of generators. The multiplicative moment map is  defined as 
\begin{equation}
  \Phi=\sum_{b \in \bar{Q}^{sep}} \Phi_{v_b}\,,
\quad \Phi_{v_b}=e_{v_b}+bb^\ast \text{ if }b \in Q^{sep},\,\,\text{or }\Phi_{v_b}=(e_{v_b}+bb^\ast)^{-1} \text{ if }b \in \bar{Q}^{sep} \setminus Q^{sep} \,. 
\end{equation}

Fusing idempotents in $A_{Q^{sep}}$ according to the identification of vertices from $Q^{sep}$ to $Q$ gives us $A_Q$. It is proved in \cite[Theorem 3.3]{F19} that the quasi-Hamiltonian algebra structure obtained on $A_Q$ by fusion using Proposition \ref{Pr:IsoFusqHam} is precisely the one given above. In particular, it does not depend on the order in which we fuse the idempotents by Theorem \ref{Thm:IsoFusqHam}.  

It remains to see that this structure is independent of the directions of the arrows in $Q$, up to isomorphism. This follows from Lemma \ref{L:IsoPhiqHam} and the isomorphism that exists for a one-arrow quiver which was given at the beginning of the proof.  
\end{proof}
\begin{rem} \label{Rem:qHamMod}
 In \cite[\S3.2]{F19}, modifications of the algebra $A_Q$ and its quasi-Hamiltonian algebra structure were considered. 
Namely, fixing a choice of coefficients $\gamma_a\in \kk$, $a\in \bar Q$, satisfying $\gamma_a=\gamma_{a^\ast}$, introduce the algebra $A_{Q,\gamma}$ obtained by universal localisation of $\kk \bar{Q}$ from the set $S_\gamma=\{1+(\gamma_a-1)e_{t(a)}+a a^\ast \mid a\in \bar Q\}$. 
The double quasi-Poisson bracket is given by \eqref{loopG}, \eqref{a<bG}, 
while we consider \eqref{aastG} with its first term multiplied by $\gamma_a$. The multiplicative moment map $\Phi=\sum_{s\in I} \Phi_s$ is such that  
\begin{equation*}
 \Phi_s= \prod_{a\in T_s}^{\longrightarrow}(\gamma_a e_s+ a a^\ast)^{\epsilon(a)}\,.
\end{equation*}
The proof of Theorem \ref{Thm:qHQuivers} can then be adapted to $A_{Q,\gamma}$. 
If $\gamma_a\neq 0$ for all $a\in \bar{Q}$, we have furthermore  an isomorphism of double quasi-Poisson algebras $ \psi:A_Q \to A_{Q,\gamma}$ given on generators by 
\begin{equation}
\psi(a)=\gamma_a^{-1}a,\,\, \psi(a^\ast)=a^\ast,\,\, \text{ for all }a\in Q\,.
\end{equation}
 Since $\psi((e_{t(a)}+aa^\ast)^{\epsilon(a)})=\gamma_a^{-\epsilon(a)}\,(\gamma_a e_{t(a)}+aa^\ast)^{\epsilon(a)}$ for all $a\in \bar{Q}$ due to the condition $\gamma_a=\gamma_{a^\ast}$, $\psi$ is in fact an isomorphism of quasi-Hamiltonian algebras if $a$ is a loop whenever $\gamma_a \neq 1$. 
\end{rem}


\section{\texorpdfstring{$H_0$}{H0}-Poisson structures} \label{S:H0}

\subsection{Definition and general results} \label{ss:H0-gen}

Let $A$ be a $\kk$-algebra. Let $[A,A]$ be the vector space spanned by commutators in $A$, from which we can define $H_0(A):=A/[A,A]$, the zeroth Hochschild homology of $A$. Denote by $a \mapsto \bar{a}$ the map which sends an element of $A$ to its image in $A/[A,A]$. Remark that any derivation $\del \in \Der(A)$ induces a linear map on $H_0(A)$ since $\del([a,b])\in [A,A]$ for any $a,b \in A$. 
Following Crawley-Boevey \cite{CB11}, we say that a $\kk$-bilinear map $\langle -,- \rangle: H_0(A)\times H_0(A)\to H_0(A)$ is a \emph{$H_0$-Poisson structure} on $A$  if it is a Lie bracket, i.e. 
\begin{equation} \label{Eq:H0Lie}
\begin{aligned}
\dhh{\bar a, \bar b}=- \dhh{\bar b,\bar a}\,, \quad  
 \dhh{\bar a, \dhh{\bar b, \bar c}}+\dhh{\bar b, \dhh{\bar c, \bar a}}+\dhh{\bar c, \dhh{\bar a, \bar b}}=0 \,, 
\end{aligned}
\end{equation}
and each linear map  $ \dhh{\bar a,-} : H_0(A)\to H_0(A)$ is induced by a derivation $\del_a \in \Der(A)$. We will write $(A,\dhh{-,-})$ when we want to emphasise the $H_0$-Poisson structure on $A$.

In the relative setting where $A$ is a $B$-algebra (following the convention of Section \ref{S:Base}), we require that the map $\dhh{\bar e,-}$ is induced by the trivial derivation $\del_e=0_A$ for each $e \in B$. In that case, we say that the $H_0$-Poisson structure $\dhh{-,-}$ is $B$-linear.  

Let $(A_1,\dhu{-,-})$ and $(A_2,\dhd{-,-})$ be two $B$-algebras with $H_0$-Poisson structures. 
Note that if $\phi:A_1\to A_2$ is a morphism of $B$-algebras, then $\phi([a,b])=[\phi(a),\phi(b)]$ for any $a,b \in A$ so $\phi$ induces a morphism $\bar \phi: H_0(A_1)\to H_0(A_2)$. 
We say that $\phi:A_1\to A_2$ is a \emph{$H_0$-Poisson morphism} if it is a morphism of $B$-algebras such that the induced map $\bar \phi: H_0(A_1)\to H_0(A_2)$ is a morphism of Lie algebras, i.e. for any $\bar a,\bar b\in H_0(A_1)$, 
\begin{equation}\label{Eq:H0Morph}
 \dhd{\bar\phi(\bar a),\bar \phi (\bar b)}=\bar\phi( \dhu{\bar a,\bar b})\,.
\end{equation}
We say that it is an \emph{$H_0$-Poisson isomorphism} if it is an isomorphism of $B$-algebras (hence $\bar \phi$ is an isomorphism of Lie algebras).

 Let $A$ be a $B$-algebra with double bracket $\dgal{-,-}$. Consider the bilinear map 
\begin{equation} \label{Eq:brH0}
 \br{-,-}=m \circ \dgal{-,-} \, : \, A \times A \to A\,,
\end{equation}
obtained by composing the double bracket with the multiplication $m$ of $A$. 
It inherits the $B$-linearity of $\dgal{-,-}$ in the following form: for any $e\in B$, $\br{e,-}$ and $\br{-,e}$ are identically zero.   
Then, as noticed in \cite[Lemma 2.4.1]{VdB1}, we have that the operation \eqref{Eq:brH0}  induces well-defined maps 
 \begin{equation} 
 \br{-,-}: H_0(A)\times A \to A\,, \quad \br{-,-}: H_0(A)\times H_0(A)\to H_0(A)\,. 
 \end{equation}
 The first operation is such that for any $\bar a \in H_0(A)$ we have $\br{\bar a,-}\in \Der(A)$, while the second operation is antisymmetric. 
By construction, this second linear map is given by 
 \begin{equation} \label{Eq:H0ind1}
  \br{\bar a, \bar b}=\overline{\br{\bar a,b}}=\overline{\br{a,b}}\,,
 \end{equation}
for any $\bar a, \bar b\in H_0(A)$ and lifts $a,b \in A$.  
We have a $B$-linear $H_0$-Poisson structure on $A$ provided that the operation $\br{-,-}$ on $H_0(A)$ satisfies Jacobi identity. 
\begin{lem} \emph{(\cite[Lemma 2.6.2]{VdB1})} \label{Lem:IndH0}
 Let $A$ be a $B$-algebra with double bracket $\dgal{-,-}$, and write $\br{-,-,-}_3=m \circ \dgal{-,-,-}$ for the map obtained by composing the triple bracket $\dgal{-,-,-}$ given by \eqref{Eq:TripBr} with the multiplication $m : A^{\times 3}\to A$. 
Then the induced map $\br{-,-}: H_0(A)\times H_0(A)\to H_0(A)$ is a $H_0$-Poisson structure if $\br{-,-,-}_3$ vanishes identically. 

In particular, we get a $H_0$-Poisson structure if $\dgal{-,-}$ is a double (quasi-)Poisson bracket.
\end{lem}
\begin{proof}
From \cite[Corollary 2.4.4]{VdB1}, we have the following identity in $A$ 
\begin{equation}
 \br{a,\br{b,c}} - \br{b,\br{a,c}}- \br{\br{a,b},c}=\br{a,b,c}_3-\br{b,a,c}_3\,.
\end{equation}
Hence if the right-hand side vanishes, this equality induced in $H_0(A)$ is just Jacobi identity. 

The second part of the statement is obvious for a double Poisson bracket, since $\dgal{-,-,-}=0$ identically. 
For a double quasi-Poisson bracket, $\br{-,-,-}_3=0$ because applying the multiplication map to \eqref{qPabc} gives zero.  
\end{proof}

\subsubsection{Structures induced by Hamiltonian algebras} \label{sss:H0ham}
We  assume that  $(A,\dgal{-,-},\mu)$ is a Hamiltonian algebra over $B=\bigoplus_{s\in I}\kk e_s$. Fix $(\lambda_s)\in \kk^I$ or equivalently $\lambda:=\sum_{s\in I}\lambda_s e_s\in B$, 
and note that by definition of the moment map, \eqref{mum} implies that for any $a \in A$  
\begin{equation*}
 \dgal{a, \mu-\lambda}=\sum_s (e_s a \otimes e_s-e_s \otimes a e_s)\,.
\end{equation*}
In particular, $\br{a, \mu-\lambda}=0$ after multiplication. This yields that, if $\overline{(\mu-\lambda)}$ denotes the vector subspace of $H_0(A)$ spanned by the image of 
the ideal $(\mu-\lambda)$ under the map $A\to A/[A,A]$, we have $\br{\bar a, \overline{(\mu-\lambda)}}\in \overline{(\mu-\lambda)}$ for any $\bar{a}\in H_0(A)$. 
In particular, the $H_0$-Poisson structure descends from $H_0(A)$ to a Lie bracket $\br{-,-}^\lambda$ on $H_0(A)^\lambda:=H_0(A)/\overline{(\mu-\lambda)}$ given by 
\begin{equation} \label{Eq:H0ind2}
 \br{\bar a +\overline{(\mu-\lambda)}, \bar b + \overline{(\mu-\lambda)}}^\lambda = (\br{a,b}+[A,A])+\overline{(\mu-\lambda)}\,, 
\quad \text{for any } a,b\in A\,.
\end{equation}

Set  $A^\lambda=A/(\mu-\lambda)$, and remark that we can identify $H_0(A)^\lambda$ with $H_0(A^\lambda):=A^\lambda/[A^\lambda,A^\lambda]$. Under this identification, the Lie bracket $\br{-,-}^\lambda$ is given by 
\begin{equation*}
 \br{\bar a +\overline{(\mu-\lambda)}, \bar b + \overline{(\mu-\lambda)}}^\lambda = \overline{\br{a,b}+(\mu-\lambda)}\,.
\end{equation*}
(Here, the bar in the right-hand side denotes the map $A^\lambda \to H_0(A^\lambda)$.) We get that $A^\lambda$ is endowed with a $H_0$-Poisson structure since 
the linear map $\br{\bar a +\overline{(\mu-\lambda)},-}$ on $H_0(A^\lambda)$ is induced by $\del_a\in \Der(A^\lambda)$ given by 
\begin{equation*}
 \del_a(b+(\mu-\lambda))=\br{a,b}+(\mu-\lambda)\,,
\end{equation*}
for any lifts $a,b\in A$. Note that the induced linear map on $H_0(A^\lambda)$ is independent of the lift, though $\del_a,\del_{a+\mu-\lambda}\in \Der(A^\lambda)$ are not the same in general. Indeed  
$\br{\mu-\lambda,b}=\sum_s (b e_s-e_s b)$ may be nonzero, but it vanishes modulo commutators. Combining this discussion with Lemma \ref{Lem:IndH0}, we have obtained the following result. 

\begin{prop} \emph{(\cite[Proposition 2.6.5]{VdB1})} \label{Pr:HamH0}
 Let $(A,\dgal{-,-},\mu)$ be a Hamiltonian algebra. Then for any $\lambda \in B$, the $H_0$-Poisson structure $\br{-,-}$ on $A$ descends to a $H_0$-Poisson structure $\br{-,-}^\lambda$ on $A^\lambda$.
\end{prop}

\begin{prop}  \label{Pr:HamH0b}
Let $\phi:(A_1,\dgal{-,-}_1)\to (A_2,\dgal{-,-}_2)$ be a morphism of double Poisson algebras. Then $\phi:(A_1,\br{-,-}_1)\to (A_2,\br{-,-}_2)$ is a $H_0$-Poisson morphism.  

If $\phi$ is a morphism of Hamiltonian algebras, then for any $\lambda \in B$, $\phi$ induces a $H_0$-Poisson morphism $\phi^\lambda:(A_1/(\mu_1-\lambda),\br{-,-}_1^\lambda)\to (A_2/(\mu_2-\lambda),\br{-,-}_2^\lambda)$.
\end{prop}
\begin{proof}
 In the first case,  we just have to show that $\bar \phi$ is a  morphism of Lie algebras. 
For any $a,b \in A_1$, \eqref{Eq:Morph} yields $\phi (\br{a,b}_1)=\br{\phi(a),\phi(b)}_2$. Hence, we get from \eqref{Eq:H0ind1} that for any 
 $\bar a, \bar b\in H_0(A_1)$ with arbitrary lifts $a,b \in A_1$, 
 \begin{equation*}
  \br{\bar \phi (\bar a), \bar \phi(\bar b)}_2=\br{\overline{ \phi (a)}, \overline{ \phi( b)}}_2=\overline{\br{\phi(a),\phi(b)}_2}
  =\overline{\phi(\br{a,b}_1)}=\bar \phi(\overline{\br{a,b}_1})=\bar \phi(\br{\bar a,\bar b}_1)\,.
 \end{equation*}
In the second case, note that $\phi((\mu_1-\lambda))\subset (\mu_2-\lambda)$ so $\phi$ induces a morphism $\phi^\lambda:A^\lambda_1\to  A^\lambda_2$ hence a morphism $\overline{\phi^\lambda}:H_0(A^\lambda_1)\to  H_0(A^\lambda_2)$ 
where $A^\lambda_i=A/(\mu_i-\lambda)$. 
In the same way, since 
$\bar \phi(\overline{(\mu_1-\lambda)})\subset \overline{(\mu_2-\lambda)}$, we have that $\bar \phi$ induces a map 
\begin{equation*}
\bar \phi^\lambda: H_0(A^\lambda_1)\simeq H_0(A_1)/\overline{(\mu_1-\lambda)} \longrightarrow H_0(A_2)/\overline{(\mu_2-\lambda)} \simeq  H_0(A^\lambda_2)\,,
\end{equation*}
which coincides with the map $\overline{\phi^\lambda}$ induced by $\phi^\lambda$. We can thus use \eqref{Eq:H0ind2} to conclude. 
\end{proof}
\begin{exmp}
Consider the Hamiltonian algebra structure on the path algebra of a double quiver $\bar{Q}$ given in \ref{ssHamQuiver}. 
By Lemma \ref{Lem:IndH0}, we get a $H_0$-Poisson structure on $\kk \bar{Q}$, and its associated Lie bracket on $H_0(\kk\bar{Q})$ is the necklace Lie bracket \cite{BLB,Gi01}.  
By Proposition \ref{Pr:HamH0}, the double bracket on $\kk \bar{Q}$ descends to a $H_0$-Poisson structure on $\Pi^\lambda(Q):=\kk \bar{Q}/(\mu-\lambda)$ \cite{CB11,VdB1}.  
The algebra $\Pi^\lambda(Q)$ is called a \emph{deformed preprojective algebra} \cite[Section 2]{CBH}. 
It was proved by Crawley-Boevey and Holland in \cite[Lemma 2.2]{CBH} that deformed preprojective algebras are independent of the orientation chosen on $Q$. Using Theorem \ref{Thm:HQuivers}, we obtain that the $H_0$-Poisson structure hence defined is independent of the orientation chosen on $Q$ up to isomorphism, and we can easily check that these isomorphisms are realised by the maps considered by Crawley-Boevey and Holland.
\end{exmp}

 \subsubsection{Structures induced by quasi-Hamiltonian algebras} 
We  assume that  $(A,\dgal{-,-},\Phi)$ is a quasi-Hamiltonian algebra over $B=\bigoplus_{s\in I}\kk e_s$. Fix $(\rc_s)\in (\kk^\times)^I$ or equivalently $\rc:=\sum_{s\in I} \rc_s e_s\in B^\times$. 
By definition of the multiplicative moment map, for any $a \in A$ we have  
\begin{equation*}
 \dgal{a, \Phi-\rc}=-\frac12\sum_{s\in I} (\Phi_s \otimes a e_s -\Phi_sa \otimes e_s + e_s\otimes a\Phi_s -e_s a\otimes \Phi_s)\,,
\end{equation*}
so that $\br{a, \Phi-\rc}=0$ after multiplication. We can thus adapt the discussion from \ref{sss:H0ham} to the quasi-Hamiltonian setting with $A^{\rc}=A/(\Phi-\rc)$ and get the following results.  
\begin{prop} \emph{(\cite[Proposition 5.1.5]{VdB1})} \label{Pr:HamqH0}
 Let $(A,\dgal{-,-},\Phi)$ be a quasi-Hamiltonian algebra. Then for any $\rc \in B^\times$, the $H_0$-Poisson structure $\br{-,-}$ on $A$ descends to a $H_0$-Poisson structure $\br{-,-}^{\rc}$ on $A^{\rc}$.
\end{prop}
\begin{prop}  \label{Pr:HamqH0b}
Let $\phi:(A_1,\dgal{-,-}_1)\to (A_2,\dgal{-,-}_2)$ be a morphism of double quasi-Poisson algebras. Then $\phi:(A_1,\br{-,-}_1)\to (A_2,\br{-,-}_2)$ is a $H_0$-Poisson morphism.  

If $\phi$ is a morphism of quasi-Hamiltonian algebras, then for any $\rc \in B^\times$, 
$\phi$ induces a $H_0$-Poisson morphism $\phi^\rc:(A_1/(\Phi_1-\rc),\br{-,-}_1^{\rc})\to (A_2/(\Phi_2-\rc),\br{-,-}_2^{\rc})$.
\end{prop}
 
\begin{exmp}
Consider the quasi-Hamiltonian algebra structure on the localisation $A_Q$ of the path algebra of a double quiver $\bar{Q}$ given in \ref{ssqHamQuiver}, which depends on an ordering of the arrows. 
The algebra $\Lambda^{\rc}(Q):=A_Q/(\Phi-\rc)$ is called a \emph{multiplicative preprojective algebra} \cite{CBShaw}. 
As noticed by Van den Bergh \cite[Proposition 6.8.1]{VdB1}, the double bracket on $A_Q$ descends to a $H_0$-Poisson structure on $\Lambda^{\rc}(Q)$, see Proposition \ref{Pr:HamqH0}.  
It was proved by Crawley-Boevey and Shaw \cite[Theorem 1.4]{CBH} that multiplicative preprojective algebras are independent of the orientation chosen on $Q$, and of the ordering of the arrows. Using Theorem \ref{Thm:qHQuivers}, we obtain that the $H_0$-Poisson structure on a multiplicative preprojective algebra is independent of the orientation and the ordering of the arrows up to isomorphism, and we can check that such isomorphisms are precisely realised by the maps considered by Crawley-Boevey and Shaw.
\end{exmp}
 
  \subsubsection{Relation to the affine moduli space of representations} \label{ss:RepSp}

We assume that the base field $\kk$ is an algebraically closed field of characteristic zero. Consider an algebra $A$ over $B=\bigoplus_{s\in I}\kk e_s$. We denote the affine representation space (relative to $B$) of $A$ with dimension vector  $\alpha\in \N^I$ by $\Rep(A,\alpha)$. Explicitly, $\Rep(A,\alpha)$ parametrises representations $\rho$ of $A$ on $\kk^N$, $N=\sum_s \alpha_s$, such that
the idempotent matrix  $\rho(e_s)$ is given by $\diag(0_{\alpha_1},\ldots,0_{\alpha_{s-1}},\Id_{\alpha_s},0_{\alpha_{s+1}},\ldots, 0_{\alpha_{|I|}})$ after identification of $I$ with $\{1,\ldots,|I|\}$. 
In other words,  under the decomposition 
\begin{equation} \label{Eq:kdec}
 \kk^N = \kk^{\alpha_1} \oplus \kk^{\alpha_2} \oplus \ldots \oplus \kk^{\alpha_{|I|}} \,,
\end{equation}
$\rho(e_s)$ projects an element $(v_1,\ldots,v_{|I|})$ onto $(0,\ldots,0,v_s,0,\ldots,0)$, where for all $r\in I$ we have $v_r\in \kk^{\alpha_{r}}$. 

There is a natural action of ${\Gl_\alpha}:=\prod_s \Gl_{\alpha_s}(\kk)$ by change of basis with respect to the decomposition \eqref{Eq:kdec}. 
For any $a\in A$, we denote by $\X(a)$ the function on $\Rep(A,\alpha)$ which returns the matrix representing $a$ at each point, i.e. 
$\X(a)(\rho)=\rho(a)$. 
Following \cite{CB11,VdB1}, we note that the map 
\begin{equation}
 \tr : A\to \kk[\Rep(A,\alpha)]\,, \quad \tr(a)=\sum_{1\leq i\leq N} \X(a)_{ii}\,,
\end{equation}
has its image which generates $\kk[\Rep(A,\alpha)]^{\Gl_\alpha}$. Furthermore, given $\bar{a}\in H_0(A)$ and two lifts $a_1,a_2\in A$, we note that $\tr(a_1)=\tr(a_2)$ since the trace vanishes on commutators of matrices.

\begin{thm} \emph{(\cite[Theorem 4.5]{CB11})} \label{Thm:H0Rep}
Let $\dhh{-,-}$ be a $H_0$-Poisson structure on $A$. Then, for any dimension vector $\alpha$, there is a unique Poisson bracket $\br{-,-}$ on $\OO[\Rep(A,\alpha)]^{\Gl_\alpha}$ such that for any $a,b\in A$ 
\begin{equation} \label{Eq:H0Rep}
 \br{\tr (a), \tr (b)}=\tr (\dhh{\bar a, \bar b}^l)\,,
\end{equation}
where $\dhh{\bar a, \bar b}^l\in A$ is an arbitrary lift of $\dhh{\bar a, \bar b}\in H_0(A)$.
\end{thm}

\begin{prop} \label{Pr:H0Rep}
 If $\phi: (A_1,\dhu{-,-})\to (A_2,\dhd{-,-})$ is a $H_0$-Poisson morphism, 
then the morphism $\bar \phi_\alpha:\OO[\Rep(A_1,\alpha)]^{\Gl_\alpha} \to \OO[\Rep(A_2,\alpha)]^{\Gl_\alpha}$, uniquely defined by  
 $\bar \phi_\alpha(\tr (a))=\tr (\phi(a))$ for any $a \in A_1$, is a Poisson morphism.
\end{prop}
\begin{proof}
Denote by $\br{-,-}_k$ the Poisson bracket on $\OO[\Rep(A_k,\alpha)]^{\Gl_\alpha}$ induced by Theorem \ref{Thm:H0Rep} for $k=1,2$. 
Then, on generators $\tr(a),\tr(b)$ of $\OO[\Rep(A_1,\alpha)]^{\Gl_\alpha}$, we have that 
\begin{subequations}
\begin{align}
 \br{\bar \phi_\alpha(\tr (a)), \bar \phi_\alpha(\tr (b))}_2 =& \br{\tr (\phi(a)),\tr (\phi(b))}_2 
 =\tr (\dhd{\overline{\phi(a)}, \overline{\phi(b)}}^l) 
 =\tr (\dhd{\bar \phi(\bar a), \bar\phi(\bar b)}^l)\,, \label{eq:phiH0a} \\
\bar \phi_\alpha(\br{\tr (a), \tr (b)}_1)=& \bar \phi_\alpha\big(\tr (\dhu{\bar a, \bar b}^l)\big)
=\tr \big(\phi(\dhu{\bar a, \bar b}^l)\big)
  =\tr \big((\bar \phi \dhu{\bar a, \bar b})^l\big) \,. \label{eq:phiH0b}
\end{align}
\end{subequations}
For \eqref{eq:phiH0a}, we used the definition of $\bar \phi_\alpha$ in the first equality, the definition \eqref{Eq:H0Rep} of the Poisson bracket in the second, the fact that $\bar \phi:H_0(A_1)\to H_0(A_2)$ is induced by $\phi$ in the third; 
we used these results similarly in \eqref{eq:phiH0b}. As the final terms in \eqref{eq:phiH0a} and \eqref{eq:phiH0b} are equal due to  \eqref{Eq:H0Morph} since $\phi$ is a $H_0$-Poisson morphism, we get that $\bar \phi_\alpha$ is a Poisson morphism. 
If $\phi$ is an isomorphism with inverse $\phi^{-1}:A_2 \to A_1$, 
then the inverse of $\bar \phi_\alpha$ is given by  $\tr (a)\mapsto\tr (\phi^{-1}(a))$ for any $a\in A_2$.
\end{proof}

 \subsection{Some applications} \label{s:H0app}
 
We have obtained in \ref{ss:H0-gen} that morphisms of double (quasi-)Poisson algebras induce morphisms of quotients of these algebras obtained by fixing the (multiplicative) moment maps to a linear combination of idempotents. In particular, the induced morphisms define morphisms of coordinate rings of moduli spaces of representations. We investigate these constructions in the rest of this section. We will study several basic examples and their relation to the first Weyl algebra and quantum tori, as well as the morphisms that they induce on ($q$-)Calogero-Moser spaces which will appear in Section \ref{S:dual}.     

In this subsection, all algebras are $B$-algebras for the convention of Section \ref{S:Base}. 
In that case, a $B$-linear morphism is a morphism satisfying $\phi(e)=e$ for all $e\in B$. 
 
\subsubsection{Using Hamiltonian algebras} \label{ss:H0Ham}

Given a Hamiltonian algebra $(A,\dgal{-,-},\mu)$ over $B$, let us denote by $\Aut(A):=\Aut_B(A)$ its group of $B$-linear automorphisms, and $\HAut(A)$ the subgroup of automorphisms which are morphisms of Hamiltonian algebras. Following \cite{BP}, let us also introduce $\Aut(A;\mu)$ as the subgroup of automorphisms of $A$ preserving $\mu$. It is clear that we have the inclusions 
\begin{equation} \label{Eq:Aut}
 \HAut(A) \subset  \Aut(A;\mu) \subset \Aut(A)\,.
\end{equation}
\begin{rem}
 The inclusions in \eqref{Eq:Aut} are not necessarily equalities. 
Consider 
$A=\kk\langle x,y,z\rangle$ with double Poisson bracket given on generators by 
\begin{equation*}
 \dgal{x,y}=1\otimes 1,\quad \dgal{z,z}=z\otimes 1 - 1 \otimes z,\quad \dgal{x,z}=0=\dgal{y,z}\,,
\end{equation*}
and moment map  $\mu=[x,y]+z$. Consider the automorphisms $\phi,\psi:A\to A$ defined on generators by  
\begin{equation*}
 \phi(x)= \frac12 x,\,\, \phi(y)= y,\,\, \phi(z)= z+\frac12[x,y], \qquad 
\psi(x)=-x,\,\, \psi(y)=y,\,\, \psi(z)=z\,.
\end{equation*}
Then $\phi$ preserves $\mu$ but it is not a morphism of double brackets, while $\psi$ does not preserve $\mu$. 
\end{rem}

The first inclusion from \eqref{Eq:Aut} descends to an inclusion $\overline{\HAut(A)} \subset  \overline{\Aut(A;\mu)}$ in the quotient  $A^\lambda=A/(\mu-\lambda)$ for any $\lambda \in B$. Moreover, the automorphisms in $\overline{\HAut(A)}$ preserve the induced $H_0$-Poisson structure by Proposition \ref{Pr:HamH0b}. 

\begin{exmp} \label{Exmp:A1}
Consider the Jordan quiver $Q_\circ$ consisting of the vertex set $I=\{0\}$ and a single arrow which is a loop $a:0\to 0$. Its double $\bar Q_\circ$ contains an additional arrow $a^\ast:0\to 0$. The path algebra of $\bar Q_\circ$ has a Hamiltonian algebra structure constructed in  \ref{ssHamQuiver}. Under $a\mapsto x,a^\ast \mapsto y$ we can induce the Hamiltonian algebra structure on the free algebra $\Fd=\kk\langle x,y\rangle$ by taking the  double Poisson bracket 
\begin{equation} \label{Eq:dbrF2}
 \dgal{x,y}=1\otimes 1\,, \quad \dgal{x,x}=0=\dgal{y,y}\,,
\end{equation}
with the moment map $\mu=[x,y]$. Meanwhile, we can start with the opposite quiver $Q^{op}_\circ$ given by $b:0\to 0$, and the isomorphism 
$b\mapsto x, b^\ast \mapsto y$  induces the same Hamiltonian algebra structure on $\Fd \simeq \kk \bar{Q}^{op}_\circ$. 
By Theorem \ref{Thm:HQuivers}, we have isomorphic Hamiltonian algebras if we start with $Q_\circ$ or $Q^{op}_\circ$, and the isomorphism can be computed from Example \ref{Exmp:Q1} to be 
\begin{equation}
 \psi:\kk\bar{Q}_\circ\to \kk\bar{Q}_\circ^{op}\,, \quad 
\psi(a)=b^\ast,\,\, \psi(a^\ast)=-b\,.
\end{equation}
 Under the identifications with $\Fd$, $\psi$ induces an isomorphism of Hamiltonian algebras given by  
\begin{equation} \label{Eq:A1FT}
 \cF:\Fd  \to \Fd \,, \quad 
\cF(x)=y,\,\, \cF(y)=-x\,.
\end{equation}
The automorphism $\cF$ satisfies $\cF^4=\id$ and is sometimes called the formal Fourier transform. We can also easily check that for all $\gamma\in \kk$ and $k\geq 0$ the automorphism 
\begin{equation} \label{Eq:A1phi}
 \phi_{k,\gamma}:\Fd  \to \Fd \,, \quad 
\phi_{k,\gamma}(x)=x+\gamma y^k,\,\, \phi_{k,\gamma}(y)=y\,,
\end{equation}
defines an isomorphism of Hamiltonian algebras as it preserves \eqref{Eq:dbrF2} and the moment map. Hence, we also get the isomorphism of Hamiltonian algebras 
\begin{equation}  \label{Eq:A1phib}
 \phi_{k,\gamma}'=\cF^{-1}\circ \phi_{k,-\gamma} \circ \cF :\Fd  \to \Fd \,, \quad 
\phi_{k,\gamma}'(x)=x,\,\, \phi_{k,\gamma}'(y)=y+\gamma x^k\,. 
\end{equation}
Imposing the relation $\mu=1$, the automorphisms \eqref{Eq:A1phi} and \eqref{Eq:A1phib} descend to automorphisms of the first Weyl algebra 
$\Aalg=\kk\langle x,y \rangle/(xy-yx-1)$. 
\end{exmp}

It is a result of Dixmier \cite[Th\'{e}or\`{e}me 8.10]{D68} that the images in $\Aalg$ of the automorphisms \eqref{Eq:A1phi} and \eqref{Eq:A1phib} generate the whole group of automorphisms. 
\begin{cor} \label{Cor:A1}
Let $\Fd$ be the free algebra on two generators and $\Aalg$ be the first Weyl algebra. 
Then $\HAut(\Fd)$ surjects onto $\Aut(\Aalg)$. 
\end{cor}
\begin{rem} \label{Rem:A1}
Using Proposition \ref{Pr:HamH0b}, any automorphism of the first Weyl algebra preserves the $H_0$-Poisson structure induced by \eqref{Eq:dbrF2}. However, this result is not interesting as $H_0(\Aalg)$ is trivial. If we set instead the moment map $\mu$ to $0$ in Example \ref{Exmp:A1}, we get that the $H_0$-Poisson structure induced by \eqref{Eq:dbrF2} on $\kk[x,y]$ is the canonical Poisson bracket defined by $\br{x,y}=1$ and $\br{x,x}=0=\br{y,y}$. 
\end{rem}

Let us now assume that $\kk$ is of characteristic zero and algebraically closed. The elements of $H_0(A^\lambda)$ induce generators of representation spaces using the trace map as in \ref{ss:RepSp}, and we can obtain the following diagram for any dimension vector $\alpha \in \N^I$
\begin{center}
   \begin{tikzpicture}
  \node (HA) at (-3,1.5) {$\HAut(A)$}; 
  \node (Amu) at (0,1.5) {$\Aut(A;\mu)$};
  \node (A) at (3,1.5) {$ \Aut(A)$};
  \node (HAl) at (-3,0) {$\overline{\HAut(A)}$};
  \node (Amul) at (0,0) {$\overline{\Aut(A;\mu)}$};
  \node (Al) at (3,0) {$\Aut(A^\lambda)$};
  \node (HArep) at (-3,-1.5) {$\overline{\HAut(A)}_\alpha$};
  \node (Amurep) at (0,-1.5) {$\overline{\Aut(A;\mu)}_\alpha$};
  \node (Arep) at (3,-1.5) {$\Aut(\mathcal{A}^\lambda_\alpha)$};
   \path[->,>=stealth,font=\scriptsize]  (HA) edge   (Amu) ;
\path[->,>=stealth,font=\scriptsize]  (Amu) edge   (A) ;
   \path[->,>=stealth,font=\scriptsize]  (HAl) edge   (Amul) ;
\path[->,>=stealth,font=\scriptsize]  (Amul) edge   (Al) ;
   \path[->,>=stealth,font=\scriptsize]  (HArep) edge   (Amurep) ;
\path[->,>=stealth,font=\scriptsize]  (Amurep) edge   (Arep) ;
\path[->,>=stealth,font=\scriptsize]  (HA) edge   (HAl) ;
\path[->,>=stealth,font=\scriptsize]  (HAl) edge   (HArep) ;
\path[->,>=stealth,font=\scriptsize]  (Amu) edge   (Amul) ;
\path[->,>=stealth,font=\scriptsize]  (Amul) edge   (Amurep) ;
   \end{tikzpicture}
\end{center}

\noindent where $\mathcal{A}^\lambda_\alpha=\OO[\Rep(A^\lambda,\alpha)]^{\Gl_\alpha}$ denotes the coordinate ring of $\X(\mu)^{-1}(\lambda \Id_\alpha)/\!/\Gl_\alpha$ for $\lambda \Id_\alpha=(\lambda_s \Id_{\alpha_s})_s$.  The automorphisms in $\overline{\HAut(A)}_\alpha$ are Poisson by Proposition \ref{Pr:H0Rep}. 

For a quiver $Q$, consider the Hamiltonian algebra structure on $\kk \bar{Q}$ given in \ref{ssHamQuiver}. We get that $\HAut(\kk \bar{Q})$ induces Poisson automorphisms on the corresponding quiver varieties by the above argument. It is an interesting question to understand what are the properties of these morphisms. 

\begin{exmp} \label{Ex:HamQ1}
We work over $\kk=\CC$. 
 Consider the quiver $Q_1$ formed by the vertices $I=\{0,\infty\}$ and arrows $x:0\to 0$, $v=0\to \infty$.  The Hamiltonian automorphisms on $\CC\langle x,y\rangle$ from 
Example \ref{Exmp:A1} can be extended to $\CC \bar{Q}_1$ using $x\mapsto x$, $y\mapsto x^\ast$ and acting as the identity on $v,v^\ast$. Fix $n\in \N^\times$. If we take $(\lambda_0,\lambda_\infty)=(1,-n)$ and $(\alpha_0,\alpha_\infty)=(n,1)$ and denote the corresponding quiver variety as $\Cn$, the elements of $\HAut(\CC \bar{Q}_1)$ descend to Poisson automorphisms on $\Cn$. As the group of automorphisms generated by the images of $\phi_{k,\gamma},\phi_{k,\gamma}'$ acts transitively on $\Cn$ by \cite[Theorems 1.2,1.3]{BW}, the same result holds for the image of $\HAut(\CC \bar{Q}_1)$. The variety $\Cn$ is the $n$-th Calogero-Moser space \cite{W}, see \ref{ss:DualRCM}. 

If $Q_2$ is the quiver obtained from $Q_1$ by adding an arrow $\infty \to 0$, we can reproduce the same construction with the same parameters to get a quiver variety $\CnTwo$. It is proved in \cite{BP} that a subgroup of $\Aut(\CC\bar{Q}_2;\mu)$ acts transitively on $\CnTwo$ when induced onto this space, see also \cite{MeTac}. 
It is not known if the subgroup of Poisson automorphisms induced by $\HAut(\CC\bar{Q}_2)$ also acts transitively. 
\end{exmp}

\subsubsection{Using quasi-Hamiltonian algebras} \label{ss:H0qHam}

Given a quasi-Hamiltonian algebra $(A,\dgal{-,-},\Phi)$ over $B$, we can reproduce the construction from \ref{ss:H0Ham}.  
Namely, we can define the group of $B$-linear automorphisms $\Aut(A)$, the subgroup of quasi-Hamiltonian automorphisms $\qHAut(A)$, and the subgroup $\Aut(A;\Phi)$ of automorphisms of $A$ preserving $\Phi$. We get the inclusions 
\begin{equation} \label{Eq:qAut}
 \qHAut(A) \subset  \Aut(A;\Phi) \subset \Aut(A)\,.
\end{equation}
The first inclusion descends to an inclusion of automorphisms of the quotient  $A^\rc=A/(\Phi-\rc)$ for any $\rc \in B^\times$ such that the image of $\qHAut(A)$ preserves the $H_0$-Poisson structure obtained from  Proposition \ref{Pr:HamqH0b}. If $\kk$ is of characteristic zero and algebraically closed, they furthermore descend to automorphisms of affine moduli spaces of representations, for which the elements in the image of $\qHAut(A)$ are Poisson isomorphisms by Proposition \ref{Pr:H0Rep}.

\begin{exmp} \label{Exmp:B1}
Consider the Jordan quiver $Q_\circ$ as in Example \ref{Exmp:A1}. The localisation $A_{Q_\circ}$ of the path algebra of $\bar Q_\circ$ has a quasi-Hamiltonian algebra structure constructed in  \ref{ssqHamQuiver}. 
Let $\Ft=\kk\langle x,y\rangle_S$ be the universal localisation of $\kk\langle x,y\rangle$ with respect to the set $S=\{1+xy,1+yx\}$. Under $a\mapsto y,a^\ast \mapsto x$ we can induce the quasi-Hamiltonian algebra structure on $\Ft \simeq A_{Q_\circ}$ with double quasi-Poisson bracket 
\begin{equation} \label{Eq:dbrqFa}
\begin{aligned}
 &\dgal{x,x}_1=-\frac12(x^2\otimes 1-1\otimes x^2)\,, \quad \dgal{y,y}_1=+\frac12 (y^2 \otimes 1 - 1 \otimes y^2)\,, \\
 &\dgal{y,x}_1=1\otimes 1+\frac12 (xy\otimes 1 +1 \otimes yx + x \otimes y - y\otimes x)\,,
\end{aligned}
\end{equation}
and the multiplicative moment map $\Phi=(1+yx)(1+xy)^{-1}$. This corresponds to taking the ordering $a<a^\ast$.  If we consider the other ordering $a^\ast<a$, we get the double quasi-Poisson bracket 
\begin{equation} \label{Eq:dbrqFb}
\begin{aligned}
 &\dgal{x,x}_2=+\frac12(x^2\otimes 1-1\otimes x^2)\,, \quad \dgal{y,y}_2=-\frac12 (y^2 \otimes 1 - 1 \otimes y^2)\,, \\
 &\dgal{y,x}_2=1\otimes 1+\frac12 (xy\otimes 1 +1 \otimes yx - x \otimes y + y\otimes x)\,,
\end{aligned}
\end{equation}
and the multiplicative moment map $\Phi'=(1+xy)^{-1}(1+yx)$. We can read from the proof of Lemma \ref{L:IsoFusqHam0} that the isomorphism of quasi-Hamiltonian algebras from $(\Ft,\dgal{-,-}_1,\Phi)$ to $(\Ft,\dgal{-,-}_2,\Phi')$ is given by 
\begin{equation} \label{Eq:B1psi}
 \psi:\Ft \to \Ft\,, \quad 
\psi(x)=x(1+xy),\,\, \psi(y)=(1+xy)^{-1}y\,.
\end{equation}
As in Example \ref{Exmp:A1}, we can start with the opposite quiver $Q_\circ^{op}$ given by $b:0\to 0$, and considering the ordering $b^\ast<b$ we get a quasi-Hamiltonian algebra structure on $A_{Q_\circ^{op}}$. Under the identification  $\Ft \simeq A_{Q_\circ^{op}}$ given by  
$b\mapsto y, b^\ast \mapsto x$, we again get a quasi-Hamiltonian algebra structure on  $\Ft$, which is given by $\dgal{-,-}_2$ and $\Phi_2$ defined above.  
By Theorem \ref{Thm:qHQuivers}, we have isomorphic quasi-Hamiltonian algebras if we start with $Q_\circ$ or $Q_\circ^{op}$, and the isomorphism\footnote{We take $a<a^\ast$ and $b^\ast<b$ to define the quasi-Hamiltonian structures.} can be computed from Example \ref{Exmp:Q1qHam} to be 
\begin{equation*}
a\mapsto b^\ast\,, \quad a^\ast \mapsto -(1+bb^\ast)^{-1}b\,.
\end{equation*}
 Hence it gives  an isomorphism of quasi-Hamiltonian algebras from $(\Ft,\dgal{-,-}_1,\Phi)$ to $(\Ft,\dgal{-,-}_2,\Phi')$ as 
\begin{equation} \label{Eq:B1xi}
 \xi:\Ft \to \Ft\,, \quad 
\xi(x)=-(1+yx)^{-1}y,\,\, \xi(y)=x\,.
\end{equation}
Note that $\xi\circ \psi^{-1}\in \qHAut(\Ft)$ when $\Ft$ is endowed with $\dgal{-,-}_2$ and $\Phi'$. Finally, we note that the automorphism 
\begin{equation} \label{Eq:B1phi}
 \phi_\beta:\Ft \to \Ft\,, \quad 
\phi_\beta(x)=\beta^{-1}x,\,\, \phi_\beta(y)=\beta y\,,\quad \beta\in \kk^\times\,,
\end{equation}
is such that $\phi_\beta\in \qHAut(\Ft)$ for both quasi-Hamiltonian algebra structures. 
After imposing the relations $\Phi=\rc^{-1}$, $\Phi'=\rc^{-1}$ for $\rc\in \kk^\times$, the automorphisms $\psi,\xi,\phi_\beta$ descend to automorphisms $\bar\psi,\bar\xi,\bar\phi_\beta$ of the localised first quantised Weyl algebra 
$\Balg^\rc$ defined as 
\begin{equation} \label{Eq:algB1}
\Balg^\rc=  (\Aalg^\rc)_{1+xy}\,, \quad \Aalg^\rc:=\kk\langle x,y\rangle/\big(1+xy-\rc(1+ yx)\big) \,.
\end{equation} 
They can be used to define the following elements of $\Aut(\Balg^\rc)$,
\begin{equation} \label{Eq:AutB1}
 (x,y)\mapsto ((1+yx)x,y (1+yx)^{-1}),\,\, 
(x,y)\mapsto (-(1+yx)^{-1}y,x),\,\, 
(x,y)\mapsto (\beta^{-1}x,\beta y),\,\,\beta \in \kk^\times\,.
\end{equation}
\end{exmp}

Alev and Dumas classified the automorphisms of $\Balg^\rc$ in \cite[Th\'{e}or\`{e}me 1.7]{AD}, and they obtained that for $\rc\neq \pm 1$, $\Aut(\Balg^\rc)$ is generated by the three automorphisms \eqref{Eq:AutB1}. 
\begin{cor} \label{Cor:B1}
Fix $\rc\in \kk\setminus \{0, \pm 1\}$. 
Let $\Ft=\kk\langle x,y\rangle_S$ for $S=\{1+xy,1+yx\}$  and $\Balg^{\rc}$ be the localisation of the first quantised Weyl algebra defined by \eqref{Eq:algB1}. 
Then, the automorphisms of $\Balg^\rc$ are all induced by  isomorphisms of quasi-Hamiltonian algebras on $\Ft$ (possibly for different structures). 
\end{cor}
Contrary to the case of the first Weyl algebra obtained in Corollary \ref{Cor:A1}, we do not have a surjection $\qHAut(\Ft)\to \Aut(\Balg^\rc)$ since $\psi:\Ft\to \Ft$ given by \eqref{Eq:B1psi} does not preserve the double bracket \eqref{Eq:dbrqFa} or its multiplicative moment map $\Phi$. This illustrates the key difference between Lemma \ref{L:IsoFusHam0} and  Lemma \ref{L:IsoFusqHam0} : in the quasi-Hamiltonian setting, performing fusion in the opposite order induces a non-trivial isomorphism.  

\begin{exmp} \label{Exmp:C1}
 Let $\Lt=\kk\langle x^{\pm 1},y^{\pm 1}\rangle$ denote the (non-commutative) algebra of Laurent polynomials in two variables. 
There is a quasi-Hamiltonian algebra structure on $\Lt$ with double quasi-Poisson bracket 
\begin{equation} \label{Eq:dbrT}
\begin{aligned}
 &\dgal{x,x}=+\frac12(x^2\otimes 1-1\otimes x^2)\,, \quad \dgal{y,y}=-\frac12 (y^2 \otimes 1 - 1 \otimes y^2)\,, \\
 &\dgal{x,y}=\frac12 (yx\otimes 1 +1 \otimes xy - x \otimes y + y\otimes x)\,,
\end{aligned}
\end{equation}
and multiplicative moment map $\Phi=xyx^{-1}y^{-1}$. This quasi-Hamiltonian algebra was considered in \cite{CF} after localisation of $A_{Q_\circ}$. It can also be obtained from an intersection pairing on the fundamental group $\pi$ of a punctured torus as we have 
\begin{equation*}
 \Lt \simeq \kk\langle x^{\pm 1},y^{\pm 1},\Phi^{\pm1}\rangle/(\Phi-xyx^{-1}y^{-1})=: \pi\,,
\end{equation*}
 see the work of Massuyeau-Turaev \cite{MT14}. 
Elements of $\Aut(\pi)$ which are not acting by conjugation can be found in \cite[Appendix]{Go03}, and they include the Dehn twists
\begin{equation}
\begin{aligned}  \label{Eq:C1tau}
  &\tau:\Lt\to \Lt\,, \quad \tau(x)=xy,\,\, \tau(y)=y\,, \\
 &\tilde \tau:\Lt\to \Lt\,, \quad \tilde\tau(x)=x,\,\, \tilde\tau(y)=yx\,.
\end{aligned}
\end{equation}
We have $\tau,\tilde \tau \in \Aut(\Lt,\Phi)$, and we can prove that both automorphisms preserve the double bracket \eqref{Eq:dbrT}, hence these elements belong to $\qHAut(\Lt)$. We also get that 
\begin{equation}
\begin{aligned}
  &\sigma:\Lt\to \Lt\,, \quad \sigma(x)=y^{-1},\,\, \sigma(y)=yxy^{-1}\,, 
\end{aligned}
\end{equation}
belongs to $\qHAut(\Lt)$ in view of $\sigma=\tau^{-1}\circ \tilde \tau \circ \tau^{-1}$. Finally, we note that the automorphism 
\begin{equation} \label{Eq:C1phi}
 \phi_{\alpha,\beta}:\Lt \to \Lt\,, \quad 
\phi_{\alpha,\beta}(x)=\alpha x,\,\, \phi_{\alpha,\beta}(y)=\beta y\,,\quad \alpha,\beta\in \kk^\times\,,
\end{equation}
is such that $\phi_{\alpha,\beta}\in \qHAut(\Lt)$. 
Imposing the relation $\Phi=\rc$ for $\rc\in \kk^\times$, all these automorphisms  descend to automorphisms of the quantum torus  
\begin{equation}\label{Eq:algC1}
\Calg^\rc=\kk\langle x^{\pm 1},y^{\pm 1}\rangle/(xy-\rc yx)\,. 
\end{equation}
They can be used to define the following elements of $\Aut(\Calg^\rc)$,
\begin{equation} \label{Eq:AutC1}
 (x,y)\mapsto (xy,y),\,\, 
(x,y)\mapsto (y^{-1},x),\,\, 
(x,y)\mapsto (\alpha x,\beta y),\,\,\alpha,\beta \in \kk^\times\,.
\end{equation}
\end{exmp}

We get from \cite[Theorem 1.5]{KPS} or \cite[Proposition 1.6]{AD} that for $\rc\neq \pm 1$, the automorphisms \eqref{Eq:AutC1} generate $\Aut(\Calg^\rc)$. 
The first two automorphisms in \eqref{Eq:AutC1} generate a  subgroup isomorphic to $\Sl_2(\Z)$, and we have in particular $\Aut(\Calg^\rc)\simeq \Sl_2(\Z) \ltimes (\kk^\times)^2$. 
\begin{cor} \label{Cor:C1}
Fix $\rc\in \kk\setminus \{0, \pm 1\}$. 
Let $\Lt=\kk\langle x^{\pm1},y^{\pm1}\rangle$ and $\Calg^{\rc}$ be the quantum torus defined by \eqref{Eq:algC1}. 
Then $\qHAut(\Lt)$ surjects onto $\Aut(\Calg^\rc)$. 
\end{cor}
\begin{rem}
Similarly to the case of the first Weyl algebra discussed in Remark \ref{Rem:A1}, the  $H_0$-Poisson structure on $\Calg^\rc$ is trivial for $\rc \neq +1$. For $\rc=+1$, we get that the $H_0$-Poisson structure induced by \eqref{Eq:dbrT} on $\kk[x^{\pm1},y^{\pm1}]$ is the Poisson bracket defined by $\br{x,y}=xy$ and $\br{x,x}=0=\br{y,y}$. 
\end{rem}

\begin{exmp} \label{Ex:C1bis}
We work over $\kk=\CC$. 
 Consider the quasi-Hamiltonian algebras $\Lt$ from Example \ref{Exmp:C1} and $A_Q$ from \ref{ssqHamQuiver}, where $Q$ consists of a unique arrow $v:0\to \infty$. After fusion of the idempotent $e_0\in A_Q$ onto the unit of $\Lt$, we get an algebra $A'$ which is quasi-Hamiltonian  by Proposition \ref{Pr:IsoFusqHam}. Fix $n\in \N^\times$ and $q\in \CC^\times$ not a root of unity. Taking $(\rc_0,\rc_\infty)=(q,q^{-n})$ and $(\alpha_0,\alpha_\infty)=(n,1)$, the corresponding affine moduli space of representations  $\Cn^q$ is the $q$-Calogero-Moser space \cite{CF,Oblomkov}. Since $\qHAut(\Lt)\subset \qHAut(A')$ descends to an algebra acting by Poisson automorphisms on representation spaces, we get an action of $\Sl_2(\Z)$ by Poisson automorphisms on $\Cn^q$. 
\end{exmp}


\section{Dual integrable systems from quivers} \label{S:dual}

We have seen as part of Theorems \ref{Thm:HQuivers} and \ref{Thm:qHQuivers} that the (quasi-)Hamiltonian algebra associated with a quiver only depends on the underlying undirected graph, up to isomorphism. In particular,  we obtain from Proposition \ref{Pr:H0Rep} an isomorphism of Poisson varieties after considering the orbit spaces obtained by (quasi-)Hamiltonian reduction of the representation spaces of these algebras. We will investigate this observation using a family of quivers denoted $Q_m$, and their opposites $Q^{op}_m$ obtained by reversing all the arrows in each quiver $Q_m$. The motivation underlying this investigation is that the quiver should also dictate a particular choice of local coordinates on a subset of the associated Poisson variety, such that the isomorphism described above will lead us to dual integrable systems. When the quiver is a cyclic quiver extended by one arrow, the choice of local coordinates that we consider on the associated space will satisfy the principle $\pP$ stated as follows  
\begin{itemize}
 \item the matrices representing the arrows in the cyclic quiver are diagonal matrices, which are either equal or related through the moment map equations;
\item the matrix representing the additional arrow is a (co)vector with all entries equal to $+1$. 
\end{itemize}
These two rules will serve us to fix the choice of representatives, up to a residual finite action. 
Furthermore, we will see that this choice naturally ensures that we obtain Lax matrices for integrable systems in the Calogero-Moser (CM)  and Ruijsenaars-Schneider (RS) families \cite{Cal,Mo,RS86,Su}. By following this general idea, we will see that the reversing of arrows $Q_m\rightsquigarrow Q_m^{op}$ yields  
\begin{enumerate}
 \item self-duality of the generalised rational CM systems from \cite{CS} in \ref{ss:DualRCM};
 \item the standard duality of the hyperbolic CM and rational RS systems in \ref{ss:DualTCM}; 
 \item the new dual integrable system of a modified hyperbolic RS system in \ref{ss:DualmRS};
 \item the standard self-duality of the hyperbolic RS system in \ref{ss:DualRS}.
\end{enumerate}
For the interested reader which has only a limited background on integrable systems, we recommend to consult \cite[Lectures 1-2]{EtingCM} for an introduction to the subject of CM systems using the perspective of Hamiltonian reduction.  

\begin{rem}
 Below, we work over $\kk=\CC$, we fix $m\geq 1$ and we let $I=\Z/m\Z$.  All the computations provided are given under the assumption that $m\geq 2$ to simplify the presentation. The statements regarding the local coordinates and their Poisson brackets also hold in the case $m=1$, and the computations in that case can be easily adapted by the reader.
\end{rem}

\subsection{Self-duality of rational CM systems} \label{ss:DualRCM}

We consider the quiver $Q_m$ with vertex set $\tilde I = I \cup \{\infty\}$, and  $m+1$ arrows given by $x_s:s\to s+1$ for each $s\in I$, and $v:\infty  \to 0$. In the double $\bar{Q}_m$, we denote the opposite arrows by $x_s^\ast : s+1\to s$, $v^\ast:0\to \infty$. We also consider the quiver $Q_m^{op}$ with the same vertex set but with arrows  $y_s : s+1\to s$, $w:0\to \infty$. The double $\bar{Q}_m^{op}$ has additional arrows $y_s^\ast :s\to s+1$, $w^\ast:\infty  \to 0$. These quivers are depicted in Figure \ref{fig:M1}.

Using \ref{ssHamQuiver}, the path algebra of each quiver has a Hamiltonian algebra structure, and by Theorem \ref{Thm:HQuivers} and Example \ref{Exmp:Q1} we have an isomorphism of Hamiltonian algebras given by 
\begin{equation} \label{Eq:psiH1}
\psi:\CC \bar{Q}_m\to \CC \bar{Q}_m^{op}\,,\quad \psi(x_s)=y_s^\ast,\,\, \psi(x_s^\ast)=-y_s, \quad \psi(v)=w^\ast,\,\, \psi(v^\ast)=-w\,.
\end{equation}
Denote the moment maps by $\mu$ and $\mu^{op}$. 
In view of Proposition \ref{Pr:H0Rep}, the map $\psi$ induces a Poisson isomorphism 
\begin{equation}\label{Eq:psiH2}
\Psi : \Rep(\CC \bar{Q}_m/(\mu-\tilde\lambda),\alpha)/\!/\Gl_\alpha \to \Rep(\CC \bar{Q}_m^{op}/(\mu^{op}-\tilde\lambda),\alpha)/\!/\Gl_\alpha\,,
\end{equation}
 for any $\alpha\in \N^{m+1}$ and $\tilde\lambda=\sum_{s\in I}\lambda_s e_s+\lambda_\infty e_\infty$ with $\lambda_s,\lambda_\infty \in \CC$. 
We will consider the cases where the dimension vector $\alpha$ is such that for some $n\geq 1$ we put $\alpha_s=n$ for each $s\in I$ and $\alpha_\infty=1$, while the parameters defining $\tilde\lambda$ are subject to $\lambda_\infty=-n|\lambda|$ for $|\lambda|:=\sum_s \lambda_s$. 
Moreover the $(\lambda_s)$ are assumed to satisfy the following regularity conditions 
\begin{equation}
 \sum_{s\in I} \lambda_s\neq 0\,, \quad k \sum_{s\in I} \lambda_s \neq \lambda_r+\ldots + \lambda_{r'}\,,
\end{equation}
for all $k\in \Z$ and $1\leq r \leq r'< m-1$. Under these conditions, the spaces appearing in \eqref{Eq:psiH2} are smooth and irreducible of dimension $2n$ \cite{CS}.

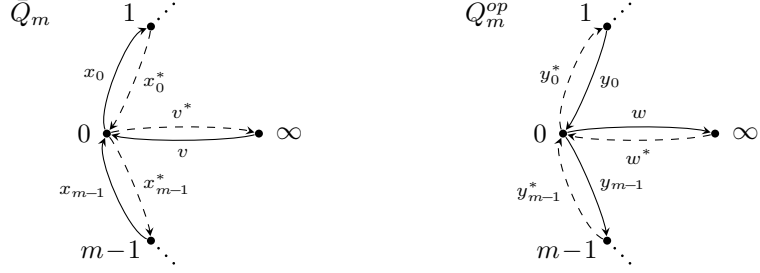
\begin{figure}
\centering
  \begin{tikzpicture}
 \node[fill,circle,inner sep=0pt,minimum size=3pt]  (IN) at (-2,0) {};
 \node[fill,circle,inner sep=0pt,minimum size=3pt]  (INb) at (4,0) {};
 \node[fill,circle,inner sep=0pt,minimum size=3pt] (C) at (-4,0) {} ; 
 \node[fill,circle,inner sep=0pt,minimum size=3pt] (C2) at (-3.42,-1.42) {} ; 
 \node[fill,circle,inner sep=0pt,minimum size=3pt] (D) at (-3.42,1.42) {} ; 
 \node[fill,circle,inner sep=0pt,minimum size=3pt] (Cb) at (2,0) {} ; 
 \node[fill,circle,inner sep=0pt,minimum size=3pt] (C2b) at (2.58,-1.42) {} ; 
 \node[fill,circle,inner sep=0pt,minimum size=3pt] (Db) at (2.58,1.42) {} ; 
  \node (INname) at (-1.6,0) {$\mathbf{\infty}$}; 
  \node (IN2name) at (4.4,0) {$\mathbf{\infty}$};
  \node (Cname) at (-4.3,0) {$0$};
  \node (C2name) at (-3.95,-1.55) {$m\!-\!1$};
  \node (Dname) at (-3.7,1.6) {$1$};
  \node (Cnameb) at (1.7,0) {$0$};
  \node (C2nameb) at (2.05,-1.55) {$m\!-\!1$};
  \node (Dnameb) at (2.3,1.6) {$1$};
  \node (Qm) at (-5,1.6) {$\bar{Q}_m$};
  \node (Qmop) at (1,1.6) {$\bar{Q}_m^{op}$};
   \path[->,>=stealth,font=\scriptsize]  
   (IN) edge [bend left=15,looseness=0.5] node[below] {$v$}  (C) ;
   \path[->,>=stealth,dashed,font=\scriptsize]  
   (C) edge [bend left=15,looseness=0.5] node[above] {$v^\ast$}  (IN) ;
   \path[->,>=stealth,dashed,font=\scriptsize]  
   (INb) edge [bend left=15,looseness=0.5] node[below] {$w^\ast$}  (Cb) ;
   \path[->,>=stealth,font=\scriptsize]  
   (Cb) edge [bend left=15,looseness=0.5] node[above] {$w$}  (INb) ;
 \node[fill,circle,inner sep=0pt,minimum size=1pt]  (up1) at (-3.32,1.52) {};
 \node[fill,circle,inner sep=0pt,minimum size=1pt]  (up2) at (-3.22,1.62) {};
 \node[fill,circle,inner sep=0pt,minimum size=1pt]  (up3) at (-3.12,1.72) {};
 \node[fill,circle,inner sep=0pt,minimum size=1pt]  (do1) at (-3.32,-1.52) {};
 \node[fill,circle,inner sep=0pt,minimum size=1pt]  (do2) at (-3.22,-1.62) {};
 \node[fill,circle,inner sep=0pt,minimum size=1pt]  (do3) at (-3.12,-1.72) {};
 \node[fill,circle,inner sep=0pt,minimum size=1pt]  (up1b) at (2.68,1.52) {};
 \node[fill,circle,inner sep=0pt,minimum size=1pt]  (up2b) at (2.78,1.62) {};
 \node[fill,circle,inner sep=0pt,minimum size=1pt]  (up3b) at (2.88,1.72) {};
 \node[fill,circle,inner sep=0pt,minimum size=1pt]  (do1b) at (2.68,-1.52) {};
 \node[fill,circle,inner sep=0pt,minimum size=1pt]  (do2b) at (2.78,-1.62) {};
 \node[fill,circle,inner sep=0pt,minimum size=1pt]  (do3b) at (2.88,-1.72) {};
   \path[->,>=stealth,font=\scriptsize]  
   (C) edge [bend left=45,looseness=0.5] node[left] {$x_{0}$} (D) ;
   \path[->,>=stealth,dashed,font=\scriptsize]  
   (D) edge [bend left=15,looseness=0.5] node[right] {$x^\ast_{0}$} (C) ;
   \path[->,>=stealth,font=\scriptsize]  
   (C2) edge [bend left=45,looseness=0.5] node[left] {$x_{m\!-\!1}$} (C) ;
   \path[->,>=stealth,dashed,font=\scriptsize]  
   (C) edge [bend left=15,looseness=0.5] node[right] {$x^\ast_{m\!-\!1}$} (C2) ;
   \path[->,>=stealth,dashed,font=\scriptsize]  
   (Cb) edge [bend left=45,looseness=0.5] node[left] {$y^\ast_{0}$} (Db) ;
   \path[->,>=stealth,font=\scriptsize]  
   (Db) edge [bend left=15,looseness=0.5] node[right] {$y_{0}$} (Cb) ;
   \path[->,>=stealth,dashed,font=\scriptsize]  
   (C2b) edge [bend left=45,looseness=0.5] node[left] {$y^\ast_{m\!-\!1}$} (Cb) ;
   \path[->,>=stealth,font=\scriptsize]  
   (Cb) edge [bend left=15,looseness=0.5] node[right] {$y_{m\!-\!1}$} (C2b) ;
   \end{tikzpicture}
\caption{On the left : the double quiver $\bar{Q}_m$ whose continuous arrows belong to the original quiver $Q_m$. On the right : the double quiver $\bar{Q}_m^{op}$ whose continuous arrows belong to the original quiver $Q_m^{op}$.} \label{fig:M1}
\end{figure}

\subsubsection{Space associated with $Q_m$}

Denote the reduced space $\Rep(\CC \bar{Q}_m/(\mu-\tilde \lambda),\alpha)/\!/\Gl_\alpha$ by $\Cn$. By construction it can be described as the set of matrices 
\begin{equation}
 X_s,X_s^\ast \in \Mat_{n\times n}(\CC)\,, \quad V\in \Mat_{1\times n}(\CC),\,\, V^\ast \in \Mat_{n\times 1}(\CC) \,,
\end{equation}
satisfying the $m$ relations 
\begin{equation} \label{Eq:momH1}
 X_sX_s^\ast - X_{s-1}^\ast X_{s-1}-\delta_{s0} V^\ast V = \lambda_s \Id_n\,,
\end{equation}
where we identify the elements in the same orbit of the action 
\begin{equation}
 g\cdot (X_s,X_s^\ast,V,V^\ast) = (g_sX_s g_{s+1}^{-1},g_{s+1}X_s^\ast g_s^{-1},V g_0^{-1},g_0V^\ast)\,, \quad 
g=(g_s)\in \Gl_n(\CC)^m\,.
\end{equation}
In view of the principle  $\pP$, we consider the subspace $\Cn'\subset \Cn$ where there exists a representative with  
\begin{equation}
 X_s=\diag(q_1,\ldots,q_n),\,\, s\in I, \quad V=(1,\ldots,1)\,,
\end{equation}
where $q_i\neq 0$ and $q_i^m\neq q_j^m$ for all $i\neq j$. Solving the constraint \eqref{Eq:momH1}, we see that we can take  
\begin{equation}
 (X_s^\ast)_{ij}=\delta_{ij}p_j +\delta_{ij}\frac{1}{q_j}(\lambda_1+\ldots+\lambda_s)
-\delta_{(i\neq j)}|\lambda| \frac{q_i^{m-s-1}q_j^s}{q_i^m-q_j^m},\,\, s\in I, \quad V^\ast=-|\lambda|(1,\ldots,1)^T\,,
\end{equation}
where the $(p_i)\in \CC^n$ are free, and we recall $|\lambda|=\sum_s \lambda_s$. This choice is unique up to the action by the generalised symmetric group $\ZSn$, where $S_n$ acts by simultaneous permutation of the $(q_i,p_i)$, while $(k_i)\in\Z_m^n$ acts by $(q_i,p_i)\mapsto (q_i \mu_m^{k_i},p_i)$ where $\mu_m$ is a fixed primitive $m$-th root of unity. 
\begin{lem} \label{L:DarbH1}
 The reduced Poisson bracket is such that $\br{q_i,q_j}=0=\br{p_i,p_j}$ and $\br{q_i,p_j}=\frac1m \delta_{ij}$. 
\end{lem}
\begin{proof}
 By construction, the double bracket on $\CC\bar{Q}_m$ is such that $\dgal{x_r,x^\ast_s}=\delta_{rs} e_{r+1}\otimes e_r$, $\dgal{x_r,x_s}=0$ and $\dgal{x_r^\ast,x_s^\ast}=0$. Denote by $\br{-,-}^\lambda$ the $H_0$-Poisson structure on $A^\lambda:=\CC\bar{Q}_m/(\mu-\tilde \lambda)$ and $\bar a \in H_0(A^\lambda)$ the image of an element $a\in A^\lambda$. We can compute from these results that for $x:=x_0\ldots x_{m-1}$ and any $k,l\in \N$, 
\begin{equation*}
 \br{\overline{x^k},\overline{x^l}}^\lambda=0,\quad \br{\overline{x^k},\overline{x^\ast_{m-1} x_{m-1} x^l}}^\lambda=k\, \overline{x^{k+l}}\,, \quad
 \br{\overline{x^\ast_{m-1} x_{m-1} x^k},\overline{x^\ast_{m-1} x_{m-1} x^l}}^\lambda=(k-l)\, \overline{x^\ast_{m-1} x_{m-1} x^{k+l}}\,.
\end{equation*}
Hence by Theorem \ref{Thm:H0Rep}, if we let $X_{cyc}=X_0\ldots X_{m-1}$ we get on $\Cn$
\begin{equation}
\begin{aligned}
 & \br{\tr X_{cyc}^k, \tr X_{cyc}^l}=0,  \quad \br{\tr X_{cyc}^k, \tr X^\ast_{m-1}X_{m-1}X_{cyc}^l} =k \tr X_{cyc}^{k+l}\,, \\
&\br{\tr X^\ast_{m-1}X_{m-1} X_{cyc}^k, \tr X^\ast_{m-1}X_{m-1}X_{cyc}^l} =(k-l)\tr X^\ast_{m-1}X_{m-1}X_{cyc}^{k+l}\,.
\end{aligned}
\end{equation}
It remains to use these identities and the local coordinates as in the proof of \cite[Proposition 2.7]{EtingCM}.
\end{proof}

The above description of $\Cn$ was essentially given by Chalykh and Silantyev \cite[Section V]{CS}. 

\subsubsection{Space associated with $Q_m^{op}$}

Denote the reduced space $\Rep(\CC \bar{Q}_m^{op}/(\mu^{op}-\tilde \lambda),\alpha)/\!/\Gl_\alpha$ by $\Cn^{op}$. As in the previous case, it can be described as the set of matrices 
\begin{equation}
 Y_s,Y_s^\ast \in \Mat_{n\times n}(\CC)\,, \quad W^\ast\in \Mat_{1\times n}(\CC),\,\, W \in \Mat_{n\times 1}(\CC) \,,
\end{equation}
satisfying 
\begin{equation} \label{Eq:momH1op}
 Y_{s-1}Y_{s-1}^\ast - Y_{s}^\ast Y_{s}+\delta_{s0} W W^\ast = \lambda_s \Id_n\,,
\end{equation}
under identifications of the elements of each orbit for the action 
\begin{equation}
 g\cdot (Y_s,Y_s^\ast,W,W^\ast) = (g_{s+1}Y_s g_{s}^{-1},g_{s}Y_s^\ast g_{s+1}^{-1},g_0W,,W^\ast g_0^{-1})\,, \quad 
g=(g_s)\in \Gl_n(\CC)^m\,.
\end{equation}
Following the principle  $\pP$, we consider the subspace $(\Cn^{op})'\subset \Cn^{op}$ where there exists a representative with  
\begin{equation}
 Y_s=\diag(\mathring{q}_1,\ldots,\mathring{q}_n),\,\, s\in I, \quad W=(1,\ldots,1)^T\,,
\end{equation}
where $\mathring{q}_i\neq 0$ and $\mathring{q}_i^m\neq \mathring{q}_j^m$ for all $i\neq j$. We can then  take  
\begin{equation}
 (Y_s^\ast)_{ij}=\delta_{ij}\mathring{p}_j -\delta_{ij}\frac{1}{\mathring{q}_j}(\lambda_1+\ldots+\lambda_s)
-\delta_{(i\neq j)}|\lambda| \frac{\mathring{q}_i^{s} \mathring{q}_j^{m-s-1}}{\mathring{q}_i^m-\mathring{q}_j^m},\,\, s\in I, \quad 
W^\ast=|\lambda|(1,\ldots,1)\,.
\end{equation}
This choice is unique up to the action of $\ZSn$. 
\begin{lem} \label{L:DarbH2}
 The reduced Poisson bracket is such that $\br{\mathring{q}_i,\mathring{q}_j}=0=\br{\mathring{p}_i,\mathring{p}_j}$ 
and $\br{\mathring{q}_i,\mathring{p}_j}=\frac1m \delta_{ij}$. 
\end{lem}
\begin{proof}
 By construction, the double bracket on $\CC\bar{Q}_m^{op}$ is such that $\dgal{y_r,y^\ast_s}=\delta_{rs} e_{r}\otimes e_{r+1}$, $\dgal{y_r,y_s}=0$ and $\dgal{y_r^\ast,y_s^\ast}=0$. It then suffices to adapt the proof of Lemma \ref{L:DarbH1}. 
\end{proof}

\subsubsection{Duality}

After rescaling the Poisson brackets on $\Cn,\Cn^{op}$ by a factor $m$, we have a set of Darboux coordinates on a dense subset of each space. 
Let $X_{cyc}=X_0\ldots X_{m-1}$, $X_{cyc}^\ast=X^\ast_{m-1}\ldots X^\ast_0$, and $Y_{cyc}=Y_{m-1}\ldots Y_0$,  $Y^\ast_{cyc}=Y^\ast_0\ldots Y^\ast_{m-1}$.
It is clear that the Poisson isomorphism \eqref{Eq:psiH2} is such that 
\begin{equation*}
 \tr Y_{cyc}^{k}\circ \Psi =(-1)^{km} \tr (X_{cyc}^\ast)^{k}\,, \quad 
\tr X_{cyc}^{k}\circ \Psi^{-1} = \tr (Y_{cyc}^\ast)^{k}\,.
\end{equation*}
In particular, these identities expressed in terms of Darboux coordinates yield that $(\tr (Y_{cyc}^\ast)^{k})_{k=1}^n$ form an integrable system on $\Cn^{op}$, and the same holds for 
 $(\tr (X_{cyc}^\ast)^{k})_{k=1}^n$ on $\Cn$. Moreover, they are action-angle dual by definition. 
In the coordinates $(q_i,p_i)$, the functions $(\tr (X_{cyc}^\ast)^{k})_{k=1}^n$ define a generalisation of the CM system which was introduced by Chalykh and Silantyev \cite{CS}. The same holds for the functions $(\tr (Y_{cyc}^\ast)^{k})_{k=1}^n$ in the coordinates $(\mathring{q}_i,\mathring{p}_i)$, so that we get self-duality for CM systems having $\ZSn$ symmetry. 

\begin{rem}
 In the case $m=1$, we recover the well-known duality of the CM system of type $A_{n-1}$ \cite{KKS78,R88,W}. 
For $m=2$ we get the duality of type $B_n$ in view of \cite[Example 5.6]{CS}. 
\end{rem}

\subsection{Duality of hyperbolic CM - rational RS systems} \label{ss:DualTCM}

We modify \ref{ss:DualRCM} as follows. 
Consider the algebra $\tilde{R}_m$ obtained by adding to $\CC\bar{Q}_m$ local inverses $x_s^{-1}=e_{s+1}x_s^{-1} e_s$ satisfying 
\begin{equation}\label{Eq:Qrel}
 x_sx_s^{-1}=e_s,\,\, x_s^{-1}x_s=e_{s+1},\quad  \text{for all } s\in I\,.
\end{equation}
 Introducing $\tilde x_s=x_s x_s^\ast$ and $\tilde v = v^\ast$, $\tilde{R}_m$ is the path algebra of the quiver $\tilde{Q}_m$  with relations \eqref{Eq:Qrel} depicted on the left of Figure \ref{fig:M2}. The double bracket on $\CC\bar{Q}_m$ uniquely extend to $\tilde{R}_m$ by the derivation rule \cite[Proposition 2.5.3]{VdB1}, hence $\tilde{R}_m$ is a Hamiltonian algebra. 

Reproducing this construction, we can form the algebra $\tilde{R}_m^{op}$ by adding local inverses $(y_s^\ast)^{-1}=e_{s+1}(y_s^\ast)^{-1} e_s$ to the path algebra $\CC\bar{Q}_m^{op}$ of the double of the quiver $Q_m^{op}$. (Note that the superscript $op$ indicates that we work with the opposite quiver, not that we take the opposite multiplication in $\tilde{R}_m$.) Taking $z_s=y_s^\ast y_s$, $\tilde{z}_s=y_s^\ast$, and $\tilde w = w^\ast$, $\tilde{R}_m^{op}$ is the path algebra of the quiver $\tilde{Q}_m^{op}$ with relation depicted on the right of Figure \ref{fig:M2}. We also obtain a Hamiltonian algebra structure on $\tilde{R}_m^{op}$ such that the map $\psi$ \eqref{Eq:psiH1} extends to an isomorphism of Hamiltonian algebras  
\begin{equation} \label{Eq:psiHt1}
\tilde{\psi}: \tilde{R}_m\to \tilde{R}_m^{op}\,,\quad \tilde{\psi}(x_s)=\tilde{z}_s,\,\, \tilde{\psi}(\tilde{x}_s)=-z_s, \quad 
\tilde{\psi}(v)=\tilde{w},\,\, \tilde{\psi}(\tilde v)=-w\,.
\end{equation}
The map $\tilde{\psi}$ induces a Poisson isomorphism $\tilde{\Psi}$ on associated Poisson varieties as in \ref{ss:DualRCM}. We consider the same dimension vector $\alpha$ and regular parameter $\tilde{\lambda}$ as in \ref{ss:DualRCM} for the rest of this subsection.

\begin{figure}
\centering
  \begin{tikzpicture}
 \node[fill,circle,inner sep=0pt,minimum size=3pt]  (IN) at (-2,0) {};
 \node[fill,circle,inner sep=0pt,minimum size=3pt]  (INb) at (4,0) {};
 \node[fill,circle,inner sep=0pt,minimum size=3pt] (C) at (-4,0) {} ; 
 \node[fill,circle,inner sep=0pt,minimum size=3pt] (C2) at (-3.42,-1.42) {} ; 
 \node[fill,circle,inner sep=0pt,minimum size=3pt] (D) at (-3.42,1.42) {} ; 
 \node[fill,circle,inner sep=0pt,minimum size=3pt] (Cb) at (2,0) {} ; 
 \node[fill,circle,inner sep=0pt,minimum size=3pt] (C2b) at (2.58,-1.42) {} ; 
 \node[fill,circle,inner sep=0pt,minimum size=3pt] (Db) at (2.58,1.42) {} ; 
  \node (INname) at (-1.6,0) {$\mathbf{\infty}$}; 
  \node (IN2name) at (4.4,0) {$\mathbf{\infty}$};
  \node (Cname) at (-4.2,-0.05) {$^0$};
  \node (C2name) at (-3.7,-1.6) {$^{m\text{-}1}$};
  \node (Dname) at (-3.6,1.5) {$^1$};
  \node (Cnameb) at (1.8,-0.05) {$^0$};
  \node (C2nameb) at (2.3,-1.6) {$^{m\text{-}1}$};
  \node (Dnameb) at (2.4,1.5) {$^1$};
  \node (Qm) at (-5.3,1.6) {$\tilde{Q}_m$};
  \node (Qmop) at (0.7,1.6) {$\tilde{Q}_m^{op}$};
   \path[->,>=stealth,font=\scriptsize]  
   (IN) edge [bend left=15,looseness=0.5] node[below] {$v$}  (C) ;
   \path[->,>=stealth,dashed,font=\scriptsize]  
   (C) edge [bend left=15,looseness=0.5] node[above] {$\tilde v$}  (IN) ;
   \path[->,>=stealth,dashed,font=\scriptsize]  
   (INb) edge [bend left=15,looseness=0.5] node[below] {$\tilde w$}  (Cb) ;
   \path[->,>=stealth,font=\scriptsize]  
   (Cb) edge [bend left=15,looseness=0.5] node[above] {$w$}  (INb) ;
 \node[fill,circle,inner sep=0pt,minimum size=1pt]  (up1) at (-3.28,1.52) {};
 \node[fill,circle,inner sep=0pt,minimum size=1pt]  (up2) at (-3.18,1.62) {};
 \node[fill,circle,inner sep=0pt,minimum size=1pt]  (up3) at (-3.08,1.72) {};
 \node[fill,circle,inner sep=0pt,minimum size=1pt]  (do1) at (-3.28,-1.52) {};
 \node[fill,circle,inner sep=0pt,minimum size=1pt]  (do2) at (-3.18,-1.62) {};
 \node[fill,circle,inner sep=0pt,minimum size=1pt]  (do3) at (-3.08,-1.72) {};
 \node[fill,circle,inner sep=0pt,minimum size=1pt]  (up1b) at (2.72,1.52) {};
 \node[fill,circle,inner sep=0pt,minimum size=1pt]  (up2b) at (2.82,1.62) {};
 \node[fill,circle,inner sep=0pt,minimum size=1pt]  (up3b) at (2.92,1.72) {};
 \node[fill,circle,inner sep=0pt,minimum size=1pt]  (do1b) at (2.72,-1.52) {};
 \node[fill,circle,inner sep=0pt,minimum size=1pt]  (do2b) at (2.82,-1.62) {};
 \node[fill,circle,inner sep=0pt,minimum size=1pt]  (do3b) at (2.92,-1.72) {};
   \path[->,>=stealth,font=\scriptsize]  
   (C) edge [bend left=15,looseness=0.5] node[left] {$x_{0}$} (D) ;
   \path[->,>=stealth,font=\scriptsize]  
   (D) edge [bend left=15,looseness=0.5] node[right] {$x^{-1}_{0}$} (C) ;
   \path[->,>=stealth,font=\scriptsize]  
   (C2) edge [bend left=15,looseness=0.5] node[left] {$x_{m\!-\!1}$} (C) ;
   \path[->,>=stealth,font=\scriptsize]  
   (C) edge [bend left=15,looseness=0.5] node[right] {$x^{-1}_{m\!-\!1}$} (C2) ;
\path[->,>=stealth,dashed,font=\scriptsize] (C) edge [out=245,in=115,looseness=40] node[left] {$\tilde{x}_0$} (C);
\path[->,>=stealth,dashed,font=\scriptsize] (C2) edge [out=290,in=160,looseness=40] node[left] {$\tilde{x}_{m-1}$} (C2);
\path[->,>=stealth,dashed,font=\scriptsize] (D) edge [out=200,in=70,looseness=40] node[left] {$\tilde{x}_{1}$} (D);
   \path[->,>=stealth,dashed,font=\scriptsize]  
   (Cb) edge [bend left=15,looseness=0.5] node[left] {$\tilde{z}_{0}$} (Db) ;
   \path[->,>=stealth,dashed,font=\scriptsize]  
   (Db) edge [bend left=15,looseness=0.5] node[right] {$\tilde{z}^{-1}_{0}$} (Cb) ;
   \path[->,>=stealth,dashed,font=\scriptsize]  
   (C2b) edge [bend left=15,looseness=0.5] node[left] {$\tilde{z}_{m\!-\!1}$} (Cb) ;
   \path[->,>=stealth,dashed,font=\scriptsize]  
   (Cb) edge [bend left=15,looseness=0.5] node[right] {$\tilde{z}^{-1}_{m\!-\!1}$} (C2b) ;
\path[->,>=stealth,font=\scriptsize] (Cb) edge [out=245,in=115,looseness=40] node[left] {$z_0$} (Cb);
\path[->,>=stealth,font=\scriptsize] (C2b) edge [out=290,in=160,looseness=40] node[left] {$z_{m-1}$} (C2b);
\path[->,>=stealth,font=\scriptsize] (Db) edge [out=200,in=70,looseness=40] node[left] {$z_{1}$} (Db);
   \end{tikzpicture}
\caption{Representation as quivers with relations of the algebras obtained from $\CC\bar{Q}_m$ and $\CC\bar{Q}_m^{op}$ by localisation.} \label{fig:M2}
\end{figure}
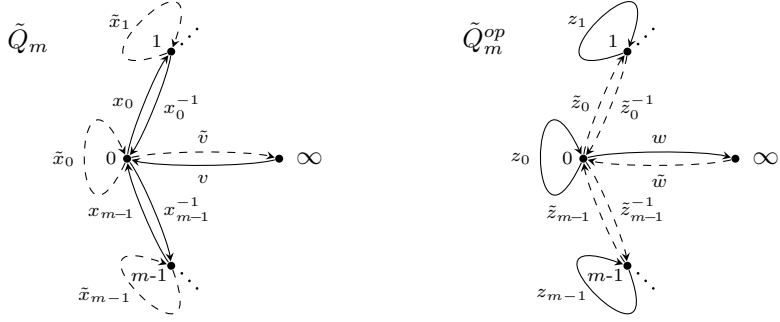

\subsubsection{Space associated with $\tilde{Q}_m$}

Denote the reduced space $\Rep(\tilde{R}_m/(\mu-\tilde\lambda),\alpha)/\!/\Gl_\alpha$ by $\tCn$. It is the subset of $\Cn$ where each $X_s$ is invertible, and where we use the elements $\tilde{X}_s=X_s X_s^\ast$, $\tilde V=V^\ast$. 
To follow the principle  $\pP$  with the matrices $(X_s,V)$, we consider again the subspace $\Cn'\subset \tCn\subset \Cn$ where there exists a representative with  
\begin{equation}
 X_s=\diag(x_1,\ldots,x_n),\,\, s\in I, \quad V=(1,\ldots,1)\,,
\end{equation}
where $x_i\neq 0$ and $x_i^m\neq x_j^m$ for all $i\neq j$, and where we take  
\begin{equation}
 (\tilde{X}_s)_{ij}=\delta_{ij}p_j +\delta_{ij}(\lambda_1+\ldots+\lambda_s)
-\delta_{(i\neq j)}|\lambda| \frac{x_i^{m-s}x_j^s}{x_i^m-x_j^m},\,\, s\in I, \quad \tilde V=-|\lambda|(1,\ldots,1)^T\,,
\end{equation}
for free parameters $(p_i)\in \CC^n$. This choice is unique up to the action by $\ZSn$. The variables $(x,p)$ correspond to $(q,qp)$ in the choice associated with $Q_m$ in \ref{ss:DualRCM}, so that we get the following result from Lemma \ref{L:DarbH1}.     
\begin{lem} \label{L:DarbHt1}
 The reduced Poisson bracket is such that $\br{x_i,x_j}=0=\br{p_i,p_j}$ and $\br{x_i,p_j}=\frac1m x_i\delta_{ij}$. 
\end{lem}
In particular, on a dense subspace of $\Cn'$ we have Darboux coordinates $(q_i,p_i)$ for $x_i=e^{\frac1m q_i}$. 

\subsubsection{Space associated with $\tilde{Q}_m^{op}$}

Denote the reduced space $\Rep(\tilde{R}_m^{op}/(\mu^{op}-\tilde\lambda),\alpha)/\!/\Gl_\alpha$ by $\tCn^{op}$. The space  $\tCn^{op}\subset \Cn^{op}$ can be described as the set of matrices 
\begin{equation}
 Z_s \in \Mat_{n\times n}(\CC),\,\, \tilde{Z}_s\in \Gl_n(\CC), \quad \tilde W \in \Mat_{1\times n}(\CC),\,\, W \in \Mat_{n\times 1}(\CC) \,,
\end{equation}
satisfying 
\begin{equation} \label{Eq:momHtop}
 \tilde{Z}^{-1}_{s-1}Z_{s-1}\tilde{Z}_{s-1} - Z_{s}+\delta_{s0} W \tilde W = \lambda_s \Id_n\,,
\end{equation}
under identification of the elements in each orbit for the action 
\begin{equation}
 g\cdot (Z_s,\tilde{Z}_s,W,\tilde W) = (g_{s}Z_s g_{s}^{-1},g_{s}\tilde{Z}_s g_{s+1}^{-1},g_0W,,\tilde W g_0^{-1})\,, \quad 
g=(g_s)\in \Gl_n(\CC)^m\,.
\end{equation}
To apply the principle  $\pP$  to the matrices $(Z_s,W)$ representing the continuous arrows of $\tilde{Q}_m^{op}$, we need the $Z_s$ to be related diagonal matrices. In view of \eqref{Eq:momHtop} with $s\neq 0$, $Z_{s-1}$ and $Z_s+\lambda_s \Id_n$ share the same spectrum.   
So we consider the subspace $(\tCn^{op})'\subset \tCn^{op}$ where  there exists a representative with  
\begin{equation} \label{Eq:ZW}
 Z_s=\diag(\tilde{q}_1,\ldots,\tilde{q}_n)-(\lambda_1+\ldots+\lambda_s)\Id_n,\,\, s\in I, \quad W=(1,\ldots,1)^T\,,
\end{equation}
and we assume that 
\begin{equation}\label{Eq:zcond}
\tilde{q}_i\neq 0,\,\, \tilde{q}_i\neq \tilde{q}_j,\,\, \tilde{q}_i \neq \tilde{q}_j \pm |\lambda|,
\quad \text{ for all } i\neq j\,.   
\end{equation}
\begin{lem} \label{L:ZopCoord}
 The subspace $(\tCn^{op})'$ can be parametrised by $(\tilde{q}_i)\in \CC^n$ satisfying \eqref{Eq:zcond} and $(\omega_i)\in (\CC^\times)^n$ where the matrices 
$(Z_s,W)$ are given by \eqref{Eq:ZW}, together with 
\begin{equation*}
 \tilde{Z}_s=\Id_n,\,\, s\neq m-1, \quad (\tilde{Z}_{m-1})_{ij}=-|\lambda| \frac{\omega_i}{\tilde{q}_i-\tilde{q}_j-|\lambda|} \prod_{k\neq j} \frac{\tilde{q}_k-\tilde{q}_j-|\lambda|}{\tilde{q}_k-\tilde{q}_j}\,, \quad 
\tilde{W}_j= |\lambda| \prod_{k\neq j} \frac{\tilde{q}_k-\tilde{q}_j-|\lambda|}{\tilde{q}_k-\tilde{q}_j}\,.
\end{equation*}
Moreover, this choice is unique up to $S_n$ action by simultaneous permutation of entries.
\end{lem}
\begin{proof}
Let us start with a representative such that \eqref{Eq:ZW} holds. 
 By assumption on $(\tilde{q}_i)$ and using \eqref{Eq:momHtop} with $s\neq 0$, $\tilde Z_s$ is diagonal, and we can fix the gauge (up to a residual permutation action) so that $\tilde Z_s=\Id_n$. If we set $\omega_i=(\tilde{Z}_{m-1}W)_i$, we get from \eqref{Eq:momHtop} with $s=0$ that 
\begin{equation}
 (\tilde{Z}_{m-1})_{ij}=-\frac{\omega_i}{\tilde{q}_i-\tilde{q}_j-|\lambda|}\tilde{W}_j\,.
\end{equation}
Using again the same equation in the form $\tilde{Z}^{-1}_{m-1}Z_{m-1}\tilde{Z}_{m-1} = Z_{0}+\lambda_0 \Id_n- W \tilde W $, we get that $Z_{m-1}$ has the same spectrum as the right-hand side. Since the eigenvalues of $Z_{m-1}$ are distinct, we must have the equality of characteristic polynomials in $\eta$
\begin{equation} \label{Eq:tSpec1}
 \prod_{k=1}^n (\tilde{q}_k-\lambda_1-\ldots-\lambda_{m-1}-\eta)
=\det(Z_{m-1}-\eta \Id_n)= \det ( Z_{0}+(\lambda_0-\eta) \Id_n- W \tilde W)\,.
\end{equation}
The right-hand side of \eqref{Eq:tSpec1} is a rank one deformation of a generically invertible matrix, so 
\begin{equation*}
\begin{aligned}
  \det ( Z_{0}+(\lambda_0-\eta) \Id_n- W \tilde W)=&  \det ( Z_{0}+(\lambda_0-\eta) \Id_n)\, [1-\tilde W( Z_{0}+(\lambda_0-\eta) \Id_n)^{-1} W]  \\
=& \prod_{k=1}^n (\tilde{q}_k+\lambda_0-\eta) - \sum_{l=1}^n \tilde W_l \prod_{k\neq l} (\tilde{q}_k+\lambda_0-\eta)\,.
\end{aligned}
\end{equation*}
Using this expression and evaluating \eqref{Eq:tSpec1} at $\eta=\tilde{q}_l+\lambda_0$, we get the claimed entries for $\tilde W$, hence for $\tilde{Z}_{m-1}$. It is then easy to check that $\omega_i=(\tilde{Z}_{m-1}W)_i$. By invertibility of $\tilde{Z}_{m-1}$ we get $\omega_i \neq 0$ for all $i$.
\end{proof}

\begin{lem} \label{L:DarbH3}
 The reduced Poisson bracket is such that $\br{\tilde{q}_i,\tilde{q}_j}=0=\br{\omega_i,\omega_j}$ and $\br{\tilde{q}_i,\omega_j}=\delta_{ij}\omega_j$. 
\end{lem}
\begin{proof}
The double bracket on $\tilde{R}_m^{op}$ is obtained by extending the one on $\CC\bar{Q}_m^{op}$ as explained at the beginning of this subsection. We can obtain in that way 
\begin{equation*}
 \dgal{\tilde{z}_r,\tilde{z}_s}=0\,, \quad \dgal{z_0,\tilde{z}_s}=\delta_{s0} e_0 \otimes \tilde{z}_0\,, \quad 
\dgal{z_0,z_0}=e_0 \otimes z_0 - z_0 \otimes e_0\,.
\end{equation*}
 Denote by $\br{-,-}^\lambda$ the $H_0$-Poisson structure on $A^\lambda:=\tilde{R}_m^{op}/(\mu-\tilde \lambda)$ and $\bar a \in H_0(A^\lambda)$ the image of an element $a\in A^\lambda$. We can compute from the above double brackets that for $\tilde{z}:=\tilde{z}_0\ldots \tilde{z}_{m-1}$ and any $k,l\in \N$, 
 \begin{equation*}
  \br{\overline{z_0^k},\overline{z_0^l}}^\lambda=0,\quad 
 \br{\overline{z_0^k},\overline{\tilde{z} z_0^l}}^\lambda=k\, \overline{\tilde{z} z_0^{k+l-1}}\,, \quad
\br{\overline{\tilde{z} z_0^k},\overline{\tilde{z} z_0^l}}^\lambda=\left[\sum_{r=1}^{k-1}-\sum_{r=1}^{l-1} \right] \, \overline{\tilde{z}z_0^{k+l-r-1}\tilde{z} z_0^r}\,. 
 \end{equation*}
 Hence by Theorem \ref{Thm:H0Rep}, if we let $\tilde{Z}_{cyc}=\tilde{Z}_0\ldots \tilde{Z}_{m-1}$ we get on $\tCn^{op}$
 \begin{equation}
 \begin{aligned}
  & \br{\tr Z_0^k, \tr Z_0^l}=0,  \quad \br{\tr Z_0^k, \tr \tilde{Z}_{cyc} Z_0^l} =k \tr \tilde{Z} Z_0^{k+l-1}\,, \\
 &\br{\tr \tilde{Z}_{cyc} Z_0^k, \tr \tilde{Z}_{cyc} Z_0^l} =
\left[\sum_{r=1}^{k-1}-\sum_{r=1}^{l-1} \right]\, \tr (\tilde{Z}_{cyc} Z_0^{k+l-r-1} \tilde{Z}_{cyc} Z_0^r)\,.
 \end{aligned}
 \end{equation}
 It remains to use these identities and the local coordinates to get the reduced Poisson structure.
\end{proof}
As a corollary, on a dense subspace of $\tCn^{op}$ we have Darboux coordinates $(\tilde{q}_i,\tilde{p}_i)$ for $\omega_i=e^{\tilde{p}_i}$. 

\subsubsection{Duality}

So far, we have obtained Darboux coordinates on dense subsets of $\tCn,\tCn^{op}$. Denoting $X_{cyc}=X_0\ldots X_{m-1}$ and $\tilde{Z}_{cyc}=\tilde{Z}_0\ldots \tilde{Z}_{m-1}$, we can use the Poisson diffeomorphism $\tilde{\Psi}:\tCn\to \tCn^{op}$ to obtain that 
\begin{equation*}
 \tr Z_{0}^{k}\circ \tilde{\Psi} =(-1)^{k} \tr \tilde{X}_{0}^{k}\,, \quad 
\tr X_{cyc}^{k}\circ \tilde{\Psi}^{-1} = \tr (\tilde{Z}_{cyc})^{k}\,.
\end{equation*}
Hence, we can conclude that $(\tr \tilde{X}_{0}^{k})_{k=1}^n$ and $(\tr (\tilde{Z}_{cyc})^{k})_{k=1}^n$ form integrable systems on $\tCn$ and $\tCn^{op}$ respectively, which are in action-angle duality. In the coordinates $(q_i,p_i)$ we can write the functions $(\tr \tilde{X}_{0}^{k})_{k=1}^n$ as hyperbolic CM Hamiltonians of type $A_{n-1}$ of order $km$ in the momenta $(p_i)$. In the coordinates $(\tilde{q}_i,\tilde{p}_i)$, the functions $\tr (\tilde{Z}_{cyc})^{k}$ are all defining Hamiltonians of the rational RS system of type $A_{n-1}$. The dependence on $m$ of this second family is only visible in the coupling $\lambda_\infty$. 

The construction of this subsection using two quivers and their representation spaces is similar to the choice of two slices inside one phase space outlined by Gorsky and Rubtsov  \cite[\S4.5]{GR}. 

\begin{rem} \label{Rem:compl1}
The duality of the hyperbolic CM system and rational RS system in the real case goes back to Ruijsenaars \cite{R88}.  
Since we work over $\CC$, we do not distinguish the hyperbolic and trigonometric CM systems which are equivalent up to making the change of coordinates $q_i \mapsto \sqrt{-1} q_i$. 
Another consequence of the fact that we work in the complex setting is that the second integrable system is written in the coordinates $(\tilde{q}_i,\tilde{p}_i)$ as rational RS Hamiltonians in  MacDonald form. In that case,  the Lax matrix and the Hamiltonians do not involve square roots contrary to the original real case \cite{RS86}. We also use the MacDonald form for the hyperbolic RS system and its variants presented in the next subsections.
\end{rem}

\subsection{Modified hyperbolic RS systems and their duals} \label{ss:DualmRS}

We consider the quivers $Q_m$ and $Q_m^{op}$ as in \ref{ss:DualRCM}. We construct the path algebras over $\CC$  of their doubles, depicted in Figure \ref{fig:M1}, then we define the algebras $A_m:=A_{Q_m}$, $A_m^{op}:=A_{Q_m^{op}}$ by universal localisation as in \ref{ssqHamQuiver}. To avoid confusion with the results in \ref{ss:DualRCM}, we will denote the elements of the double with a hat instead of a star, e.g. $x_s^\ast$ will be denoted $\hat{x}_s$. 
We fix the following ordering $<_s$ on the elements $a$ of the doubles such that $t(a)=s$,  
\begin{equation}
 \begin{aligned}
  x_s<_s \hat{x}_{s-1},\,\, s\in I\setminus\{0\},\quad \text{and }\quad  x_0<_0 \hat{x}_{m-1} <_0 \hat{v}\,; \\ 
  \hat{y}_s<_s y_{s-1},\,\, s\in I\setminus\{0\},\quad \text{and }\quad \hat{y}_0 <_0 y_{m-1} <_0 w\,. 
 \end{aligned}
\end{equation}
(For $m=1$, we only have $x_0<_0 \hat{x}_{0} <_0 \hat{v}$ or $\hat{y}_0 <_0 y_{0} <_0 w$.) 
By \ref{ssqHamQuiver}, this defines a quasi-Hamiltonian algebra structure on $A_m$ and $A_m^{op}$. Using Theorem \ref{Thm:qHQuivers} and Example \ref{Exmp:Q1qHam}, the choice of orderings gives the following isomorphism of quasi-Hamiltonian algebras 
\begin{equation} \label{Eq:psiqH1}
\hat{\psi}:A_m\to A_m^{op}\,,\quad \hat{\psi}(x_s)=\hat{y}_s,\,\, \hat{\psi}(\hat{x}_s)=-(e_{s+1}+y_s\hat{y}_s)^{-1}y_s, \quad 
\hat{\psi}(v)=\hat{w},\,\, \hat{\psi}(\hat{v})=-(e_0+w \hat{w})^{-1}w\,.
\end{equation}
Denote the multiplicative moment maps by $\Phi$ and $\Phi^{op}$. 
In view of Proposition \ref{Pr:H0Rep}, the map $\hat{\psi}$ induces a Poisson isomorphism 
\begin{equation}\label{Eq:psiqH2}
\hat{\Psi} : \Rep(A_m/(\Phi-\tilde{\rc}),\alpha)/\!/\Gl_\alpha \to \Rep(A_m^{op}/(\Phi^{op}-\tilde{\rc}),\alpha)/\!/\Gl_\alpha\,,
\end{equation}
 for any $\alpha\in \N^{m+1}$ and $\tilde{\rc}=\sum_{s\in I}\rc_s e_s+\rc_\infty e_\infty$ with $\rc_s,\rc_\infty \in \CC^\times$. We will consider the cases where $\alpha_s=n$ for each $s\in I$ with $n\geq 1$, $\alpha_\infty=1$, while $\rc_\infty=(\prod_s \rc_s)^{-n}$ and the $(\rc_s)$ are subject to the regularity conditions 
\begin{equation}
 \prod_{s\in I} \rc_s\neq 1\,, \quad \prod_{s\in I} \rc_s^k \neq  \prod_{r\leq \rho \leq r'}\rc_\rho\,,
\end{equation}
for all $k\in \Z$ and $1\leq r \leq r'<m-1$. Under these conditions, the spaces appearing in \eqref{Eq:psiqH2} are smooth and irreducible of dimension $2n$ \cite{BEF,CF}.

\subsubsection{Space associated with $Q_m$}

Denote the reduced space $\Rep(A_m/(\Phi-\tilde{\rc}),\alpha)/\!/\Gl_\alpha$ by $\hCn$. By construction it can be described as the set of matrices 
\begin{equation}
 X_s,\hat{X}_s \in \Mat_{n\times n}(\CC)\,, \quad V\in \Mat_{1\times n}(\CC),\,\, \hat{V} \in \Mat_{n\times 1}(\CC) \,,
\end{equation}
satisfying the $m$ relations 
\begin{equation} \label{Eq:momqH1}
 (\Id_n+X_s\hat{X}_s) (\Id_n+ \hat{X}_{s-1} X_{s-1})^{-1}(\Id_n+\delta_{s0} \hat V V)^{-1} = \rc_s \Id_n\,,
\end{equation}
where all the factors appearing in \eqref{Eq:momqH1} are invertible, and we identify the elements in the same orbit of the action 
\begin{equation}
 g\cdot (X_s,\hat X_s,V,\hat V) = (g_sX_s g_{s+1}^{-1},g_{s+1}\hat{X}_s g_s^{-1},V g_0^{-1},g_0 \hat{V})\,, \quad 
g=(g_s)\in \Gl_n(\CC)^m\,.
\end{equation}
In view of the principle  $\pP$, we consider the subspace $\hCn'\subset \hCn$ where there exists a representative with  
\begin{equation} \label{Eq:XW}
 X_s=\diag(x_1,\ldots,x_n),\,\, s\in I, \quad V=(1,\ldots,1)\,,
\end{equation}
where the $(x_i)\in \CC^n$ satisfy the following conditions with $t:=\prod_{s\in I}\rc_s$, 
\begin{equation}\label{Eq:xcond}
x_i\neq 0,\,\, x_i^m\neq x_j^m,\,\, x_i^m \neq t x_j^m, \quad \text{ for all } i\neq j\,.   
\end{equation}

\begin{lem}  \label{L:qXCoord}
 The subspace $\hCn'$ can be parametrised by $(x_i)\in \CC^n$ satisfying \eqref{Eq:xcond} and $(\rho_i)\in (\CC^\times)^n$ where the matrices 
$(X_s,V)$ are given by \eqref{Eq:XW}, together with\footnote{In the local form of $\hat{X}_s$, we consider that $\rc_1\ldots \rc_s$ for $s=0$ is the empty product equal to $+1$. Hereafter, we follow this convention.} 
\begin{equation*}
 \hat{X}_s=-\delta_{ij}\frac{1}{x_i}+\rc_1\ldots \rc_s(1-t^{-1}) \rho_j\frac{x_i^{m-s-1}x_j^s}{x_j^m-t^{-1}x_i^m} 
\prod_{l\neq i} \frac{x_l^m-t^{-1}x_i^m}{x_l^m-x_i^m}\,, \quad 
\hat{V}_i= -(1-t^{-1}) \prod_{l\neq i} \frac{x_l^m-t^{-1}x_i^m}{x_l^m-x_i^m}\,.
\end{equation*}
Moreover, this choice is unique up to $\ZSn$ action by simultaneous permutation of entries for $S_n$, and rescaling $x_i\mapsto \mu_m^{k_i}x_i$ for $(k_i)\in \Z_m^n$ with $\mu_m$ a fixed primitive $m$-th root of unity.
\end{lem}
\begin{proof}
 In  $\hCn'$, we can rewrite the moment map conditions  \eqref{Eq:momqH1} as 
\begin{equation} \label{Eq:momqH1b}
 X_s(\hat{X}_s+X_s^{-1})X_{s-1}^{-1} ( \hat{X}_{s-1} +X_{s-1}^{-1})^{-1}(\Id_n+\delta_{s0} \hat V V)^{-1} = \rc_s \Id_n\,.
\end{equation}
Taking a representative such that \eqref{Eq:XW} holds, we can use these identities to get 
\begin{equation*}
 (\hat{X}_s+X_s^{-1})_{ij}=\rc_1\ldots \rc_s \, x_i^{-s}x_j^s \,(\hat{X}_0+X_0^{-1})_{ij}\,, \quad 
(\hat{X}_0+X_0^{-1})_{ij}=t\frac{x_i^{-1}\hat{V}_i \rho_j}{1-t x_i^{-m}x_j^m}\,,
\end{equation*}
where we have set 
\begin{equation} \label{Eq:rho}
\rho_i=(VD^{-m+1}(\hat{X}_0+X_0^{-1})D^m)_i\,, \quad \text{for } D=\diag(x_1,\ldots,x_n) \,.
\end{equation}
From this,  we obtain that if $\hat{V}$ has the form claimed in the statement, then it will follow that it is true also for the matrices $\hat{X}_s$. In particular, it can then be checked that \eqref{Eq:rho} is satisfied for these particular matrices. 

To determine the entries of $\hat{V}$ for a representative with \eqref{Eq:XW}, we note that the moment map condition implies 
\begin{equation} \label{Eq:momqH1c}
t^{-1} (\hat{X}_0+X_0^{-1})D^{-m} ( \hat{X}_{0} +X_{0}^{-1})^{-1}= D^{-m} + D^{-1}\hat V\, VD^{-m+1}\,.
\end{equation}
The matrices on both sides of this equality share the same spectrum hence, as in the proof of Lemma \ref{L:ZopCoord}, equality of their characteristic polynomials in $\eta$ yields  
\begin{equation}
 \prod_{k=1}^n (t^{-1}x^{-m}_k - \eta) = \prod_{k=1}^n (x^{-m}_k - \eta) + \sum_l \hat{V}_l x_l^{-m} \prod_{k\neq l} (x^{-m}_k - \eta)\,.
\end{equation}
Evaluating this identity at $\eta=x^{-m}_i$, we get the desired $\hat{V}_i$.  

Finally, we note that $\rho_i\neq 0$ because the matrices $\hat{X}_{s} +X_{s}^{-1}$ are invertible. The uniqueness of the representative up to $\ZSn$ action is easily obtained.
\end{proof}

\begin{lem} \label{L:DarbqH1}
 The reduced Poisson bracket is such that $\br{x_i,x_j}=0=\br{\rho_i,\rho_j}$ and $\br{x_i,\rho_j}=\frac1m \delta_{ij}x_i\rho_j$. 
\end{lem}
\begin{proof}
 Due to the choice of ordering, we get from \ref{ssqHamQuiver} that the double bracket on $A_m$ satisfies  
\begin{equation*}
 \begin{aligned}
  \dgal{x_s,x_r}=&\frac12 \delta_{s,r+1}\, x_rx_{r+1}\otimes e_{r+1} - \frac12 \delta_{s,r-1}\, e_r\otimes x_{r-1}x_r\,, \quad 
\dgal{\hat{x}_0,\hat{x}_0}=0\,, \\
\dgal{x_s,\hat{x}_0}=&\delta_{s,0}\,\left[e_1 \otimes e_0 +\frac12 \hat{x}_0x_0\otimes e_0 + \frac12 e_1 \otimes x_0\hat{x}_0 \right]
-\frac12 \delta_{s,1}\, x_1 \otimes \hat{x}_0 + \frac12 \delta_{s,m-1}\, \hat{x}_0 \otimes x_{m-1}\,.
 \end{aligned}
\end{equation*}
 Introducing $x:=x_0\ldots x_{m-1}$ and $\check{x}_0:=e_0+x_0\hat{x}_0$, we can then get  
\begin{equation}
 \begin{aligned} \label{Eq:dbrhat}
  \dgal{x,x}=&\frac12 (x^2\otimes e_0 - e_0 \otimes x^2)\,, \quad 
\dgal{\check{x}_0,\check{x}_0}=\frac12 (\check{x}_0^2 \otimes e_0 - e_0 \otimes \check{x}_0^2)\,, \\
\dgal{x,\check{x}_0}=&\frac12 (x \otimes \check{x}_0 + \check{x}_0 \otimes x + \check{x}_0 x \otimes e_0 - e_0 \otimes x \check{x}_0)\,.
 \end{aligned}
\end{equation}
 Denote by $\br{-,-}^{\rc}$ the $H_0$-Poisson structure on $A^{\rc}:=A_m/(\Phi-\tilde{\rc})$ and $\bar a \in H_0(A^{\rc})$ the image of an element $a\in A^{\rc}$. In a way similar to the proof of \cite[Proposition 4.4]{CF}, we can compute from \eqref{Eq:dbrhat} that for  any $k,l\in \N$, 
 \begin{equation*}
  \br{\overline{x^k},\overline{x^l}}^{\rc}=0,\quad 
 \br{\overline{x^k},\overline{\check{x}_0 x^l}}^{\rc}=k\, \overline{\check{x}_0 x^{k+l}}\,, \quad
\br{\overline{\check{x}_0 x^k},\overline{\check{x}_0 x^l}}^{\rc}=\left[\sum_{r=1}^{k}-\sum_{r=1}^{l} \right] \, \overline{\check{x}_0 x^{k+l-r-} \check{x}_0 x^r}\,.
 \end{equation*}
 Hence by Theorem \ref{Thm:H0Rep}, if we let $X_{cyc}:=X_0\ldots X_{m-1}$ and $\check{X}_0:=\Id_n+X_0\hat{X}_0$, we get on $\hCn$
 \begin{equation} \label{Eq:trbrX}
 \begin{aligned}
  & \br{\tr X_{cyc}^k, \tr X_{cyc}^l}=0,  \quad \br{\tr X_{cyc}^k, \tr \check{X}_0 X_{cyc}^l} =k \tr \check{X}_0 X_{cyc}^{k+l}\,, \\
 &\br{\tr \check{X}_0 X_{cyc}^k, \tr \check{X}_0 X_{cyc}^l} =
\left[\sum_{r=1}^{k}-\sum_{r=1}^{l} \right]\, \tr (\check{X}_0 X_{cyc}^{k+l-r} \check{X}_0 X_{cyc}^r)\,.
 \end{aligned}
 \end{equation}
We can then derive the reduced Poisson structure by writing these identities in local coordinates, which can be done e.g. by adapting \cite[\S3.1]{CF}. 
\end{proof}

As an application of this lemma, we can get Darboux coordinates $(q_i,p_i)$ on a dense subset of $\hCn$ by considering $x_i=e^{\frac1m q_i}$ and $\rho_i=e^{p_i}$.

\begin{rem}
 As part of the local computations needed to prove Lemma \ref{L:DarbqH1}, we can get that  
\begin{equation} \label{Eq:nu}
  \br{\nu_i, \nu_j} 
=\nu_i \nu_j
\frac{x_i^m+x_j^m}{x_i^m-x_j^m}\frac{(t^{-1}-1)^2 x_i^m x_j^m}{(x_i^m - t^{-1}x_j^m)(x_j^m - t^{-1}x_i^m)}\,, \quad 
\nu_i := \rho_i \prod_{l\neq i}\frac{x_l^m-t^{-1}x_i^m}{x_l^m-x_i^m}\,.
\end{equation}
The Poisson bracket of the elements $(x_i,\nu_i)$ was first obtained in this form by Fock and Rosly \cite{FockRosly}, see also \cite[\S~2.4]{Oblomkov}. Note that the Poisson bracket \eqref{Eq:nu} is invariant under replacing $t^{-1}$ by $t$. This explains the difference between our parametrisation in Lemma \ref{L:qXCoord} and the one in \cite[Sections~3,4]{CF}. There are another two ways to obtain Darboux coordinates, namely using the Ruijsenaars form as in \cite[Appendix]{FockRosly}, or a form related to the $q$KP hierarchy \cite[Section 6]{Iliev00}. We refer to \cite[\S~4.2.1]{Fthesis} for a review of these different possibilities in the case $m=1$, which are easily adapted to any $m\geq 1$. 
\end{rem}

\subsubsection{Space associated with $Q_m^{op}$}

Denote the reduced space $\Rep(A_m^{op}/(\Phi^{op}-\tilde{\rc}),\alpha)/\!/\Gl_\alpha$ by $\hCn^{op}$. It can be given as the set of matrices 
\begin{equation}
 Y_s,\hat{Y}_s \in \Mat_{n\times n}(\CC)\,, \quad \hat{W}\in \Mat_{1\times n}(\CC),\,\, W \in \Mat_{n\times 1}(\CC) \,,
\end{equation}
satisfying the $m$ relations 
\begin{equation} \label{Eq:momqH2}
 (\Id_n+\hat{Y}_s Y_s)^{-1} (\Id_n+ Y_{s-1} \hat{Y}_{s-1})(\Id_n+\delta_{s0} W\hat W) = \rc_s \Id_n\,,
\end{equation}
where all the factors appearing in \eqref{Eq:momqH2} are invertible, and we identify the elements in the same orbit for the action 
\begin{equation}
 g\cdot (Y_s,\hat Y_s,W,\hat W) = (g_{s+1} Y_s g_{s}^{-1},g_{s}\hat{Y}_s g_{s+1}^{-1},g_0 W , \hat{W}g_0^{-1})\,, \quad 
g=(g_s)\in \Gl_n(\CC)^m\,.
\end{equation}
In analogy with the space associated with $Q_m$, we consider the subspace $(\hCn^{op})'\subset \hCn^{op}$ where there exists a representative with  
\begin{equation} \label{Eq:YW}
 Y_s=\diag(y_1,\ldots,y_n),\,\, s\in I, \quad W=(1,\ldots,1)^T\,,
\end{equation}
where the $(y_i)\in \CC^n$ satisfy the conditions  \eqref{Eq:xcond}, and we set again $t:=\prod_{s\in I}\rc_s$. 

\begin{lem}  \label{L:qYCoord}
 The subspace $(\hCn^{op})'$ can be parametrised by $(y_i)\in \CC^n$ satisfying \eqref{Eq:xcond} and $(\tau_i)\in (\CC^\times)^n$ where the matrices 
$(Y_s,W)$ are given by \eqref{Eq:YW}, together with 
\begin{equation*}
 \hat{Y}_s=-\delta_{ij}\frac{1}{y_i}+{\rc}_1^{-1}\ldots \rc_s^{-1}(1-t) \tau_i\frac{y_i^{s}y_j^{m-s-1}}{y_i^m-t y_j^m} 
\prod_{k\neq j} \frac{y_k^m-t y_j^m}{y_k^m-y_j^m}\,, \quad 
\hat{W}_j= -(1-t) \prod_{k\neq j} \frac{y_k^m-t y_j^m}{y_k^m-y_j^m}\,.
\end{equation*}
Moreover, this choice is unique up to $\ZSn$ action.
\end{lem}
\begin{proof}
 The result can be derived in the same way as Lemma \ref{L:qXCoord}. 
To obtain the form of $\hat W$, we remark that the moment map \eqref{Eq:momqH2} implies the following identity : 
\begin{equation} \label{Eq:momqH2b}
 (Y_0+\hat{Y}_0)^{-1} D^{-m} (Y_0+\hat{Y}_0) = t^{-1} D^{-m} + t^{-1} D^{1-m}W \hat{W}D^{-1}\,, \quad D=\diag(y_1,\ldots,y_n)\,.
\end{equation}
It then suffices to equal both characteristic polynomials to determine $\hat{W}$. 
\end{proof}

\begin{lem} \label{L:DarbqH2}
 The reduced Poisson bracket is such that $\br{y_i,y_j}=0=\br{\tau_i,\tau_j}$ and $\br{y_i,\tau_j}=\frac1m \delta_{ij}y_i\tau_j$. 
\end{lem}
\begin{proof}
 Due to the choice of ordering, the double bracket on $A_m^{op}$ is such that 
\begin{equation*}
 \begin{aligned}
  \dgal{y_s,y_r}=&\frac12 \delta_{s,r+1}\, e_{r+1}\otimes y_{r+1}y_r - \frac12 \delta_{s,r-1}\, y_r y_{r-1}\otimes e_r\,, \quad 
\dgal{\hat{y}_0,\hat{y}_0}=0\,, \\
\dgal{y_s,\hat{y}_0}=&\delta_{s,0}\,\left[e_0 \otimes e_1 +\frac12 \hat{y}_0y_0\otimes e_1 + \frac12 e_0 \otimes y_0\hat{y}_0 \right]
-\frac12 \delta_{s,1}\, \hat{y}_0 \otimes y_1 + \frac12 \delta_{s,m-1}\,y_{m-1}\otimes \hat{y}_0\,.
 \end{aligned}
\end{equation*}
 Introducing $y:=y_{m-1}\ldots y_0$ and $\check{y}_0:=e_0+\hat{y}_0y_0$, we can then get  
\begin{equation*}
 \begin{aligned}
  \dgal{y,y}=&\frac12 (e_0 \otimes y^2 - y^2\otimes e_0)\,, \quad 
\dgal{\check{y}_0,\check{y}_0}=\frac12 (e_0 \otimes \check{y}_0^2 - \check{y}_0^2 \otimes e_0 )\,, \\
\dgal{y,\check{y}_0}=&\frac12 (y \otimes \check{y}_0 + \check{y}_0 \otimes y + e_0 \otimes y \check{y}_0- \check{y}_0 y \otimes e_0 )\,.
 \end{aligned}
\end{equation*}
If we let $Y_{cyc}:=Y_{m-1}\ldots Y_{0}$ and $\check{Y}_0:=\Id_n+\hat{Y}_0Y_0$, we get on $\hCn^{op}$ the following identities 
 \begin{equation}
 \begin{aligned} \label{Eq:trbrY}
  & \br{\tr Y_{cyc}^k, \tr Y_{cyc}^l}=0,  \quad \br{\tr Y_{cyc}^k, \tr \check{Y}_0 Y_{cyc}^l} =k \tr \check{Y}_0 Y_{cyc}^{k+l}\,, \\
 &\br{\tr \check{Y}_0 Y_{cyc}^k, \tr \check{Y}_0 Y_{cyc}^l} =
\left[\sum_{r=1}^{k}-\sum_{r=1}^{l} \right]\, \tr (\check{Y}_0 Y_{cyc}^{k+l-r} \check{Y}_0 Y_{cyc}^r)\,,
 \end{aligned}
 \end{equation}
which can be derived as in Lemma \ref{L:DarbqH1}.  Notice that \eqref{Eq:trbrX} and \eqref{Eq:trbrY} are the same equations if we replace  
$X_{cyc},\check{X}_0$ by $Y_{cyc},\check{Y}_0$, or vice-versa. If we also remark that the functions 
$\tr X_{cyc}^k,\tr \check{X}_0 X_{cyc}^l$ written in coordinates in $\hCn'$, and the functions 
$\tr Y_{cyc}^k,\tr \check{Y}_0 Y_{cyc}^l$ written in coordinates in $(\hCn^{op})'$ are exactly the same when we replace 
$(x_j,\rho_j,t)$ by $(y_j,\tau_j,t^{-1})$, the statement follows from  Lemma \ref{L:DarbqH1}. 
\end{proof}
As a corollary, we can get Darboux coordinates $(\hat{q}_i,\hat{p}_i)$ on a dense subset of $\hCn^{op}$ by considering $y_i=e^{\frac1m \hat{q}_i}$ and $\tau_i=e^{\hat{p}_i}$. 

\subsubsection{Duality}

We have obtained Darboux coordinates on dense subsets of $\hCn,\hCn^{op}$, which we now use to get integrable systems in action-angle duality. Introducing $X_{cyc}=X_0\ldots X_{m-1}$, 
$Y_{cyc}=Y_{m-1}\ldots Y_{0}$, $\hat{Y}_{cyc}=\hat{Y}_0\ldots \hat{Y}_{m-1}$ and $L_{cyc}=L_{m-1}\ldots L_{0}$ for $L_s=(\Id_n+\hat{X}_sX_s)^{-1}\hat{X}_s$, we can use the Poisson diffeomorphism $\hat{\Psi}:\hCn\to \hCn^{op}$ to obtain that 
\begin{equation*}
 \tr \hat{Y}_{cyc}^{k}\circ \hat{\Psi} = \tr X^{k}\,, \quad 
\tr L_{cyc}^{k}\circ \hat{\Psi}^{-1} =(-1)^{km}\, \tr Y_{cyc}^{k}\,.
\end{equation*}
Hence, we can conclude that $(\tr L_{cyc}^{k})_{k=1}^n$ and $( \tr \hat{Y}_{cyc}^{k})_{k=1}^n$ form integrable systems on $\hCn$ and $\hCn^{op}$ respectively, which are in action-angle duality. 
In the coordinates $(\hat{q}_i,\hat{p}_i)$ we can write the functions $(\tr \hat{Y}_{cyc}^{k})_{k=1}^n$ as deformations of the hyperbolic (or trigonometric, see Remark \ref{Rem:compl1}) RS  Hamiltonians of type $A_{n-1}$. They are written explicitly as the family $(\tr Y^j)$ for $m=1$, and the family $(H_{m,j})$ for $m\geq 2$ in \cite{CF}, and their relation to other integrable systems is discussed. 

In order to write the functions $(\tr L_{cyc}^{k})_{k=1}^n$ in the coordinates $(q_i,p_i)$, we note that on $\hCn'$ 
\begin{equation}
 \begin{aligned}
&  \Id_n+\hat{X}_s X_s=\rc_1\ldots \rc_s\, (1-t)\, T_s^L C T_s^R\,, \quad \text{ where } \\
&(T_s^L)_{ij}=\delta_{ij}x_i^{m-s-1}\prod_{k\neq i}\frac{x_k^m-t^{-1}x_i^m}{x_k^m-x_i^m}\,,\quad 
(T_s^R)_{ij}=\delta_{ij} \rho_j x_j^{s+1}\,, \quad 
C_{ij}=\frac{1}{x_i^m-t x_j^m}\,.
 \end{aligned}
\end{equation}
As $C$ is a Cauchy matrix, its inverse can be computed to be 
\begin{equation}
 C^{-1}_{ij}=(1-t)(1-t^{-1})\frac{x_i^m x_j^m}{x_i^m-t^{-1}x_j^m}\,
\prod_{k\neq j} \frac{x_k^m-t^{-1} x_j^m}{x_k^m-x_j^m} \,\prod_{l\neq i}\frac{x_l^m-t x_i^m}{x_l^m-x_i^m}\,.
\end{equation}
Therefore, the entries of $L_s=(\Id_n+\hat{X}_s X_s)^{-1}\hat{X}_s$ are given by 
\begin{equation}
 \begin{aligned}
(L_s)_{ij}=&-\frac{1-t^{-1}}{\rc_1\ldots \rc_s}\rho_i^{-1}\, \frac{x_i^{m-s-1}x_j^s}{x_i^m-t^{-1}x_j^m}\, \prod_{l\neq i}\frac{x_l^m-t x_i^m}{x_l^m-x_i^m}\\
&+(1-t)(1-t^{-1})\sum_k \frac{x_i^{m-s-1}x_k^m x_j^s}{(x_i^m-t^{-1}x_k^m)(x_k^m-t x_j^m)}\rho_i^{-1}\rho_j \prod_{l\neq i}\frac{x_l^m-t x_i^m}{x_l^m-x_i^m}\, \prod_{a\neq k} \frac{x_a^m-t^{-1} x_k^m}{x_a^m-x_k^m}\,.
 \end{aligned}
\end{equation}
In the simplest case $m=1$, the first element in the family $(\tr L_{cyc}^{k})_{k=1}^n$ can be written as 
\begin{equation}
 \tr L_0= -\sum_i \rho_i^{-1}\prod_{l\neq i}\frac{x_l-t x_i}{x_l-x_i}
+(1-t^{-1})^2\sum_{i,k}\frac{x_k}{(x_i-t^{-1}x_k)^2}\, \prod_{l\neq i}\frac{x_l-t x_i}{x_l-x_i}\, \prod_{a\neq k} \frac{x_a-t^{-1} x_k}{x_a-x_k}\,.
\end{equation}
In terms of the Darboux coordinates $(q_i,p_i)$ and $t=e^{2\gamma}$, $\gamma\in \CC^\times$, this can be transformed into 
\begin{equation}
\begin{aligned}
  \tr L_0=& -\sum_ie^{-p_i}\prod_{l\neq i}\frac{\sinh\left(\frac{q_l-q_i}{2}-\gamma \right)}{\sinh\left(\frac{q_l-q_i}{2}\right)} \\
&+\sum_{i,k}\frac{\sinh^2\left(\gamma \right)\, e^{-q_i}}{\sinh^2\left(\frac{q_i-q_k}{2}+\gamma\right)}\,
 \prod_{l\neq i} \frac{\sinh\left(\frac{q_l-q_i}{2}-\gamma \right)}{\sinh\left(\frac{q_l-q_i}{2}\right)}\, 
\prod_{a\neq k} \frac{\sinh\left(\frac{q_a-q_k}{2}+\gamma \right)}{\sinh\left(\frac{q_a-q_k}{2}\right)}\,.
\end{aligned}
\end{equation}
The duality between the families $(\tr L_{cyc}^{k})_{k=1}^n$ and $( \tr \hat{Y}_{cyc}^{k})_{k=1}^n$ seem to be new. 

\subsection{Self-duality of hyperbolic RS systems} \label{ss:DualRS}

We modify \ref{ss:DualmRS} as follows. Consider the algebra $\check{A}_m$ obtained by adding to $A_m$ local inverses $x_{s}^{-1}=e_{s+1}x_s^{-1}e_s$, i.e. these elements satisfy \eqref{Eq:Qrel}. 
We let $\check{x}_s=e_s+x_s\hat{x}_s$ and $\check{v}=\hat{v}$. Note that $x_s,v,\check{x}_s,\check{v}$, the idempotents and the inverses $x_s^{-1},\check{x}_s^{-1},(e_\infty+v \check{v})^{-1}, (e_0+\check{v}v)^{-1}$ generate $\check{A}_m$. Moreover, $\check{A}_m$ is a quasi-Hamiltonian algebra if we extend the double bracket from $A_{Q_m}$ described in \ref{ss:DualmRS} by localisation. 

In the same way, we introduce the algebra $\check{A}_m^{op}$ obtained by adding to $A_m^{op}$ defined in \ref{ss:DualmRS} local inverses $y_{s}^{-1}=e_{s}y_s^{-1}e_{s+1}$. We use $z_s=e_s+\hat{y}_sy_s$, $\check{z}_s=\hat{y}_s$, $\check{w}=\hat{w}$ and $w$ together with the idempotents and the inverses belonging to  $\check{A}_m^{op}$ as generators of this algebra. It is also a quasi-Hamiltonian algebra by localisation. By construction, the map $\hat{\psi}$ \eqref{Eq:psiqH1} extends to an isomorphism of quasi-Hamiltonian algebras which can be written as 
\begin{equation} \label{Eq:psiqHt1}
\check{\psi}:\check{A}_m\to \check{A}_m^{op}\,,\quad \check{\psi}(x_s)=\check{z}_s,\,\, \check{\psi}(\check{x}_s)=z_s^{-1}, \quad 
\check{\psi}(v)=\check{w},\,\, \check{\psi}(\check{v})=-(e_0+w \check{w})^{-1}w\,.
\end{equation}
The map $\check{\psi}$ induces a Poisson isomorphism $\check{\Psi}$ on associated Poisson varieties as in \ref{ss:DualmRS}. We consider the same dimension vector $\alpha$ and regular parameter $\tilde{\rc}$ as in \ref{ss:DualmRS} for the rest of this subsection.

To use the principle  $\pP$ in order to find dual integrable systems, we will see $\check{A}_m$ and $\check{A}_m^{op}$ as algebras attached to the quivers with relations $\tilde{Q}_m,\tilde{Q}_m^{op}$ depicted in Figure \ref{fig:M2}. (We use $\check{x}_s$ instead of $\tilde{x}_s$ and do the same for $\check{v},\check{z}_s,\check{w}$ to avoid confusion with the cases considered in \ref{ss:DualTCM}.) 

\subsubsection{Space associated with $\tilde{Q}_m$}

Denote the reduced space $\Rep(\check{A}_m/(\Phi-\tilde{\rc}),\alpha)/\!/\Gl_\alpha$ by $\cCn$. It is the subset of $\hCn$ described in \ref{ss:DualmRS} where each $X_s$ is invertible, and where we use the elements $\check{X}_s=\Id_n+X_s \hat{X}_s$, $\check V=\hat{V}$. 
To follow the principle  $\pP$  with the matrices $(X_s,V)$, we consider again the subspace $\hCn'\subset \cCn\subset \hCn$ where there exists a representative with $X_s,V$ satisfying \eqref{Eq:XW}. The diagonal entries $(x_i)$ of the $X_s$  satisfy \eqref{Eq:xcond}, and we have 
\begin{equation}
 \check{X}_s=\rc_1\ldots \rc_s(1-t^{-1}) \rho_j\frac{x_i^{m-s}x_j^s}{x_j^m-t^{-1}x_i^m} 
\prod_{l\neq i} \frac{x_l^m-t^{-1}x_i^m}{x_l^m-x_i^m}\,, \quad 
\hat{V}_i= -(1-t^{-1}) \prod_{l\neq i} \frac{x_l^m-t^{-1}x_i^m}{x_l^m-x_i^m}\,.
\end{equation}
for $t:=\prod_s \rc_s$. The $(\rho_i)\in (\CC^\times)^n$ are free, and this choice is unique up to the action by $\ZSn$. The Poisson bracket between the variables $(x,\rho)$ is given in Lemma \ref{L:DarbqH1}. 
In particular, on a dense subspace of $\cCn$ we have Darboux coordinates $(q_i,p_i)$ for $x_i=e^{\frac1m q_i}$ and $\rho_i=e^{p_i}$.

\subsubsection{Space associated with $\tilde{Q}_m^{op}$}

Denote the reduced space $\Rep(\check{A}_m^{op}/(\Phi-\tilde{\rc}),\alpha)/\!/\Gl_\alpha$ by $\cCn^{op}$. 
It can be given as the set of matrices 
\begin{equation}
 Z_s,\check{Z}_s \in \Gl_{n}(\CC)\,, \quad \check{W}\in \Mat_{1\times n}(\CC),\,\, W \in \Mat_{n\times 1}(\CC) \,,
\end{equation}
satisfying the $m$ relations 
\begin{equation} \label{Eq:momqHc1}
Z_s^{-1} \check{Z}_{s-1}^{-1} Z_{s-1} \check{Z}_{s-1} (\Id_n+\delta_{s0} W\check W) = \rc_s \Id_n\,,
\end{equation}
where all the factors appearing in \eqref{Eq:momqHc1} are invertible, and we identify the elements in the same orbit of the action 
\begin{equation}
 g\cdot (Z_s,\check Z_s,W,\check W) = (g_{s} Z_s g_{s}^{-1},g_{s}\check{Z}_s g_{s+1}^{-1},g_0 W , \hat{W}g_0^{-1})\,, \quad 
g=(g_s)\in \Gl_n(\CC)^m\,.
\end{equation}
To apply the principle  $\pP$  to the matrices $(Z_s,W)$ representing the continuous arrows of $\tilde{Q}_m^{op}$, we need the $Z_s$ to be related diagonal matrices. In view of \eqref{Eq:momqHc1} with $s\neq 0$, $Z_{s-1}$ and $\rc_sZ_s$ share the same spectrum.   
So we consider the subspace $(\cCn^{op})'\subset \cCn^{op}$ where  there exists a representative with  
\begin{equation} \label{Eq:cZW}
 Z_s={\rc}_1^{-1}\ldots \rc_s^{-1}\,\diag(z_1,\ldots,z_n),\,\, s\in I, \quad W=(1,\ldots,1)^T\,,
\end{equation}
and we assume that for $t:=\prod_{s\in I}\rc_s$, 
\begin{equation}\label{Eq:czcond}
z_i\neq 0,\,\, z_i\neq z_j,\,\, z_i \neq t z_j, \quad \text{ for all } i\neq j\,.   
\end{equation}

\begin{lem} \label{L:qZopCoord}
 The subspace $(\cCn^{op})'$ can be parametrised by $(z_i)\in \CC^n$ satisfying \eqref{Eq:czcond} and $(\sigma_i)\in (\CC^\times)^n$ where the matrices 
$(Z_s,W)$ are given by \eqref{Eq:cZW}, together with 
\begin{equation*}
 \check{Z}_s=\Id_n,\,\, s\neq m-1, \quad 
(\check{Z}_{m-1})_{ij}=(1-t)\sigma_i \frac{z_i}{z_i-tz_j} \prod_{k\neq j} \frac{z_k-tz_j}{z_k-z_j}\,, \quad 
\check{W}_j= -(1-t)  \prod_{k\neq j} \frac{z_k-tz_j}{z_k-z_j}\,.
\end{equation*}
Moreover, this choice is unique up to $S_n$ action by simultaneous permutation of entries. 
\end{lem}
\begin{proof}
 Let us start with a representative such that \eqref{Eq:cZW} holds. 
 By assumption on $(z_i)$ and using \eqref{Eq:momqHc1} with $s\neq 0$, $\tilde Z_s$ is diagonal, and we can fix the gauge (up to a finite action) so that $\tilde Z_s=\Id_n$. Using the case $s=0$ in \eqref{Eq:momqHc1}, we get that 
\begin{equation}
 (\check{Z}_{m-1})_{ij}=-\frac{\sigma_i \check{W}_j}{1-t z_i^{-1}z_j}\,,
\end{equation}
where we have set $\sigma_i=(\check{Z}_{m-1}W)_i$. The equation that we have just used can also be written as 
\begin{equation} \label{Eq:momqHc2}
 t \check{Z}_{m-1}^{-1}D^{-1}\check{Z}_{m-1}= D^{-1}+W\, (\check{W}D^{-1})\,, \quad D=\diag(z_1,\ldots,z_{n})\,.
\end{equation}
Using that the two sides of \eqref{Eq:momqHc2} share the same spectrum, we get as in the proof of Lemma \ref{L:ZopCoord} the following equality of characteristic polynomials in $\eta$ 
\begin{equation}
 \prod_{k=1}^n (t z_k^{-1} - \eta) = \prod_{k=1}^n (z_k^{-1} - \eta) + \sum_l \check{W}_l z_l^{-1} \prod_{k\neq l} (z^{-1}_k - \eta)\,.
\end{equation}
Evaluating this identity at $\eta=z^{-1}_j$, we get the desired $\check{W}_j$.  
 
Finally, we can check that $\sigma_i=(\check{Z}_{m-1}W)_i$ holds, and we note that $\sigma_i\neq 0$ because the matrix $\check{Z}_{m-1}$ is invertible. The residual $S_n$ action is clear.
\end{proof}

\begin{lem} \label{L:DarbqH3}
 The reduced Poisson bracket is such that $\br{z_i,z_j}=0=\br{\sigma_i,\sigma_j}$ and $\br{z_i,\sigma_j}=\delta_{ij}z_i\sigma_j$. 
\end{lem}
\begin{proof}
We first work in $A_m^{op}$ as in Lemma \ref{L:DarbqH2}. 
 Due to the choice of ordering, we have 
\begin{equation*}
 \begin{aligned}
  \dgal{\hat{y}_s,\hat{y}_r}=&\frac12 \delta_{s,r+1}\, \hat{y}_r \hat{y}_{r+1}\otimes e_{r+1} - \frac12 \delta_{s,r-1}\, e_r\otimes \hat{y}_{r-1}\hat{y}_r\,, \quad 
\dgal{y_0,y_0}=0\,, \\
\dgal{y_0,\hat{y}_s}=&\delta_{s,0}\,\left[e_0 \otimes e_1 +\frac12 \hat{y}_0y_0\otimes e_1 + \frac12 e_0 \otimes y_0\hat{y}_0 \right]
-\frac12 \delta_{s,m-1}\, \hat{y}_{m-1} \otimes y_0 + \frac12 \delta_{s,1}\,y_{0}\otimes \hat{y}_1\,.
 \end{aligned}
\end{equation*}
By localisation, these relations hold in $\check{A}_m^{op}$, where we can determine the double brackets between the generators  $z_0=e_0+\hat{y}_0y_0$ and $\check{z}_s=\hat{y}_s$. 
 Introducing $\check{z}:=\check{z}_0\ldots \check{z}_{m-1}$, we can then get  
\begin{equation*}
 \begin{aligned}
  \dgal{\check{z},\check{z}}=&\frac12 (\check{z}^2 \otimes e_0 - e_0\otimes \check{z}^2)\,, \quad 
\dgal{z_0,z_0}=\frac12 (e_0 \otimes z_0^2 - z_0^2 \otimes e_0 )\,, \\
\dgal{\check{z},z_0}=&-\frac12 (e_0\otimes \check{z}z_0 + z_0\check{z} \otimes e_0 + \check{z} \otimes z_0 - z_0 \otimes \check{z} )\,.
 \end{aligned}
\end{equation*}
If we let $\check{Z}_{cyc}:=\check{Z}_0\ldots \check{Z}_{m-1}$, we get from these double brackets the following identities on $\hCn^{op}$ 
 \begin{equation}
 \begin{aligned} \label{Eq:trbrZ}
  & \br{\tr Z_0^k, \tr Z_0^l}=0,  \quad \br{\tr Z_0^k, \tr \check{Z}_{cyc} Z_0^l} = k \tr(\check{Z}_{cyc}Z_0^{k+l})\,, \\
 &\br{\tr \check{Z}_{cyc} Z_0^k, \tr\check{Z}_{cyc} Z_0^l} = 
\left[\sum_{r=1}^{k}-\sum_{r=1}^{l} \right]\, \tr (\check{Z}_{cyc} Z_0^{k+l-r} \check{Z}_{cyc} Z_0^r)\,,
 \end{aligned}
 \end{equation}
which can be derived as in Lemma \ref{L:DarbqH1}.  
 Notice that \eqref{Eq:trbrX} and \eqref{Eq:trbrZ} are the same equations if we replace  
 $X_{cyc},\check{X}_0$ by $Z_0,\check{Z}_{cyc}$ (in that order), or vice-versa. 
If we also remark that the functions 
 $\tr X_{cyc}^k,\tr \check{X}_0 X_{cyc}^l$ written in coordinates in $\hCn'$, and the functions 
 $\tr Z_0^k,\tr \check{Z}_{cyc} Z_0^l$ written in coordinates in $(\cCn^{op})'$ are exactly the same upon replacing  
$(x_j^m,\rho_j,t)$ by $(z_j,\sigma_j,t^{-1})$, then the statement follows from  Lemma \ref{L:DarbqH1}. 
\end{proof}
As a corollary, we can get Darboux coordinates $(\check{q}_i,\check{p}_i)$ on a dense subset of $\cCn^{op}$ by considering $z_i=e^{\check{q}_i}$ and $\sigma_i=e^{\check{p}_i}$.

\subsubsection{Duality}

So far, we have obtained Darboux coordinates on dense subsets of $\cCn,\cCn^{op}$. Denoting $X_{cyc}=X_0\ldots X_{m-1}$ and $\check{Z}_{cyc}=\check{Z}_0\ldots \check{Z}_{m-1}$, we can use the Poisson diffeomorphism $\check{\Psi}:\cCn\to \cCn^{op}$ to obtain that 
\begin{equation*}
 \tr Z_{0}^{-k}\circ \check{\Psi} = \tr \check{X}_{0}^{k}\,, \quad 
\tr X_{cyc}^{k}\circ \check{\Psi}^{-1} = \tr (\check{Z}_{cyc})^{k}\,.
\end{equation*}
Hence, we conclude that we have action-angle duality between the integrable systems $(\tr \check{X}_{0}^{k})_{k=1}^n$ and $(\tr (\check{Z}_{cyc})^{k})_{k=1}^n$ defined on $\cCn$ and $\cCn^{op}$ respectively. In the coordinates $(q_i,p_i)$ we can write the functions $(\tr \check{X}_{0}^{k})_{k=1}^n$ as hyperbolic (or trigonometric, see Remark \ref{Rem:compl1}) RS Hamiltonians of type $A_{n-1}$, which are exponential of order $km$ in the momenta $(p_i)$. In the coordinates $(\check{q}_i,\check{p}_i)$, the functions $\tr (\check{Z}_{cyc})^{k}$ are also Hamiltonians of the hyperbolic  RS system of type $A_{n-1}$, but they are simply exponential of order $k$ in the momenta. The dependence on $m$ of this second family is hidden in the coupling $t=\prod_{s\in I} \rc_s$.

\begin{rem}
 Fix $m=1$. In that case, the duality map hence obtained is close to the one considered in \cite[Proposition 3.8]{CF}, though the duality map used in that paper relates the spaces with inverse parameters $t$ and $t^{-1}$.  
Note also that, upon identifying the generators $(x,y)$ from Example \ref{Exmp:C1} with $(x_0,x_0^{-1}\check{x}_0)$, the quasi-Hamiltonian isomorphisms \eqref{Eq:AutC1} induce Poisson automorphisms on $\cCn$ which act as the identity on $V$ and $\check{V}$ while $X_0,\check{X}_0$ transform as
\begin{equation}
 (X_0,\check{X}_0)\mapsto (\check{X}_0,\check{X}_0 X_0^{-1}\check{X}_0)\,, \quad 
(X_0,\check{X}_0) \mapsto (\check{X}_0^{-1}X_0,X_0)\,, \quad 
(X_0,\check{X}_0) \mapsto (\alpha X_0,\alpha\beta \check{X}_0),
\end{equation}
with $\alpha,\beta\in \CC^\times$. 
  As noted in Example \ref{Ex:C1bis}, we get in particular that $\Sl_2(\Z)$ acts on $\cCn$ by Poisson automorphism. This $\Sl_2(\Z)$ action was related to the self-duality of the hyperbolic RS system in \cite[\S2.2]{GR}. The analogue of this result in the real setting can be found in \cite{FK12,FK13}. The original discovery of self-duality of the hyperbolic RS system in the real case is due to Ruijsenaars \cite{R88}. 
\end{rem}

\subsection{Additional remarks and further directions}

\subsubsection{} The flows of the Hamiltonian vector fields associated with the two integrable systems presented in \ref{ss:DualRCM} can be computed explicitly by adapting \cite{CS}. In particular, they are complete in $\Cn$ and $\Cn^{op}$. Similarly, the flows associated with the different integrable systems found in the other subsections can also be computed explicitly (see \cite{CF} to get the flows of $(\tr \hat{Y}_{cyc}^k)$ from \ref{ss:DualmRS} and the systems in \ref{ss:DualRS}), and they are complete in the phase spaces that support them.

\subsubsection{} The different integrable systems defined in the previous subsections admit spin extensions \cite{GH,KrZ}, which can be described as adding internal degrees of freedom to the different particles. The phase space of such spin extensions can be obtained by adding multiple framing arrows from the vertex $\infty$ to the different vertices in the cyclic quiver, see e.g. \cite{CF2,CS,Fai,Si}. Motivated by the work of Reshetikhin on the duality of spin systems in the CM-RS family \cite{Res,Res2}, it would be interesting to understand if the duality map can be realised at the level of quivers for systems with spins.


\appendix

\section{The fusion double bracket}  \label{App:Dbr}

Let $A$ be an algebra over $B=\oplus_{s\in I}\kk e_s$. Fix $i,j \in I$, $i \neq j$, and consider the fusion algebra $A^f=A^f_{e_j \to e_i}$ obtained by fusing $e_j$ onto $e_i$ as in \ref{ssFus}. The fusion algebra $A^f$ always admits a particular double bracket, denoted $\dgal{-,-}_{fus}$, which we define now. 
For $\epsilon=1-e_j$, recall from Lemma \ref{AfGenerat} that the generators of $A^f$ are of the form $\epsilon t \epsilon$ for $t\in \epsilon A \epsilon$, 
$e_{ij} u  \epsilon$ for $u\in e_j A \epsilon$, $\epsilon v e_{ji}$ for $v \in \epsilon A e_j$, and $e_{ij} w e_{ji}$ for $e_j A e_j$.  The double bracket $\dgal{-,-}_{fus}$ is 
given on generators as follows :
\begin{subequations}
  \begin{align}
&\dgal{\epsilon t \epsilon , \epsilon \tilde{t} \epsilon}_{fus}=0\,, \label{tt}\\
&\dgal{\epsilon t \epsilon , e_{ij} u \epsilon}_{fus}=\frac12 \left( e_i \otimes t e_{ij}u - e_i t  \otimes e_{ij}u\right)\,, \label{tu}\\
&\dgal{\epsilon t \epsilon ,\epsilon v e_{ji} }_{fus}=\frac12 \left( v e_{ji} t   \otimes e_i - v e_{ji} \otimes t e_i\right)\,, \label{tv}\\
&\dgal{\epsilon t \epsilon , e_{ij} w e_{ji}}_{fus}=\frac12 \left( e_{ij}we_{ji}t \otimes e_i + e_i \otimes t e_{ij} w e_{ji} - e_{ij} w e_{ji} \otimes t e_i - e_i t \otimes e_{ij} w e_{ji}\right)\,, \label{tw}
  \end{align}
\end{subequations}
when the first component $\epsilon t \epsilon$ is a generator of the first type \eqref{type1}; 
\begin{subequations}
  \begin{align}
&\dgal{e_{ij} u  \epsilon , \epsilon t \epsilon}_{fus}=\frac12 (e_{ij}u \otimes e_i t- t e_{ij} u  \otimes e_i) \,, \label{ut}\\
&\dgal{e_{ij} u \epsilon , e_{ij} \tilde{u} \epsilon}_{fus}= \frac12 (e_i \otimes e_{ij}u e_{ij}\tilde{u}-e_{ij}\tilde{u} e_{ij}u \otimes e_i) \,, \label{uu}\\
&\dgal{e_{ij} u \epsilon ,\epsilon v e_{ji} }_{fus}=\frac12(e_{ij}u \otimes e_i ve_{ji}-v e_{ji} \otimes e_{ij}u e_i)\,, \label{uv}\\
&\dgal{e_{ij} u\epsilon , e_{ij} w e_{ji}}_{fus}=\frac12(e_i \otimes e_{ij}u e_{ij}we_{ji} - e_{ij}w  e_{ji} \otimes e_{ij}u e_i) \,, \label{uw}
  \end{align}
\end{subequations}
when the first component $e_{ij} u  \epsilon$ is a generator of the second type \eqref{type2}; 
\begin{subequations}
  \begin{align}
 &\dgal{\epsilon v e_{ji} , \epsilon t \epsilon}_{fus}=\frac12 (t e_i \otimes v e_{ji}-e_i \otimes v e_{ji}t) \,, \label{vt}\\
&\dgal{\epsilon v e_{ji}, e_{ij} u \epsilon}_{fus}= \frac12 ( e_{ij}u e_i \otimes v e_{ji} - e_i v e_{ji}\otimes e_{ij}u) \,, \label{vu}\\
&\dgal{\epsilon v e_{ji},\epsilon \tilde{v} e_{ji} }_{fus}=\frac12( \tilde{v} e_{ji}v e_{ji}\otimes e_i - e_i \otimes v e_{ji}\tilde{v} e_{ji})\,, \label{vv}\\
&\dgal{\epsilon v e_{ji} , e_{ij} w e_{ji}}_{fus}=\frac12(e_{ij}we_{ji}ve_{ji}\otimes e_i - e_i v e_{ji}\otimes e_{ij}we_{ji} ) \,, \label{vw}
  \end{align}
\end{subequations}
when the first component $\epsilon v e_{ji}$ is a generator of the third type \eqref{type3}; 
\begin{subequations}
  \begin{align}
 &\dgal{e_{ij} w e_{ji} , \epsilon t \epsilon}_{fus}=\frac12 (t e_i \otimes e_{ij}we_{ji}+e_{ij}we_{ji}\otimes e_i t - t e_{ij}we_{ji}\otimes e_i - e_i \otimes e_{ij}we_{ji}t) \,, \label{wt}\\
&\dgal{e_{ij} w e_{ji}, e_{ij} u \epsilon}_{fus}= \frac12 ( e_{ij}u e_i \otimes e_{ij}w  e_{ji}  -  e_{ij}u e_{ij}we_{ji} \otimes e_i) \,, \label{wu}\\
&\dgal{e_{ij} w e_{ji},\epsilon v e_{ji} }_{fus}=\frac12( e_{ij}we_{ji} \otimes e_i v e_{ji} - e_i \otimes e_{ij}we_{ji}ve_{ji})\,, \label{wv}\\
&\dgal{e_{ij} w e_{ji} , e_{ij} \tilde{w} e_{ji}}_{fus}=0 \,, \label{ww}
  \end{align}
\end{subequations}
when the first component $e_{ij} w e_{ji}$ is a generator of the fourth type \eqref{type4}. 
The case corresponding to $i=1,j=2$ gives the double bracket considered in Proposition \ref{Pr:IsoFusqHam} on $A^f_{e_2\to e_1}$.
The double bracket $\dgal{-,-}_{fus}$ was first introduced by Van den Bergh in \cite[Theorem 5.3.1]{VdB1}, and the explicit form given above was computed in \cite[Lemma 2.19]{F19}.

\section{Identities for the proof of Lemma \ref{L:IsoFusqHam0}} \label{App:LIsoFusqHam0}

In this section, we collect  the values of all the elements in $A_2\otimes A_2$ given by 
\begin{equation*}
 \dgalindU{c,d}\,,\,\,\, \dgalfusU{c,d}\,, \quad \dgalindD{\psi(c),\psi(d)}\,,\,\,\, \dgalfusD{\psi(c),\psi(d)}\,,
\end{equation*}
for all the needed specialisations of the generators $c,d \in A_1$. It is then easy to check that \eqref{Eq:FusqHam} is satisfied in all such cases, i.e. 
\begin{equation*}
 \psi^{\otimes2}\dgalindU{c,d}+\psi^{\otimes2}\dgalfusU{c,d}= \dgalindD{\psi(c),\psi(d)}+\dgalfusD{\psi(c),\psi(d)}\,.
\end{equation*}
As a result, the proof of Lemma \ref{L:IsoFusqHam0} is completed.

\begin{rem}
  In each case, we write $\dgalindU{c,d}$ as an element of $e A f \otimes e' A f'$ 
for some $e,e'\in\{\he,e_1,e_{12}\}$ and $f,f'\in\{\he,e_1,e_{21}\}$. 
In particular, this completely characterises how $\psi^{\otimes 2}$ acts on $\dgalindU{c,d}$. 
\end{rem}

\subsection{} We consider the first specialisation of a generator of first type, that is $c=a$ for $a\in \he A \he$.

\textit{Case 1.1 : $d=b$, $b \in \he A \he$.} 
\begin{equation}
 \begin{aligned}
  &\dgalindU{c,d}=\he \dgal{a,b}'\he \otimes \he \dgal{a,b}''\he, \,\,\, \dgalfusU{c,d}=0, \\
  &\dgalindD{\psi(c),\psi(d)}=\dgal{a,b}, \,\,\, \dgalfusD{\psi(c),\psi(d)} =0\,.
 \end{aligned}
\end{equation}

\textit{Case 1.2 : $d=b$, $b \in e_1 A \he$.} 
\begin{equation}
 \begin{aligned}
  &\dgalindU{c,d}=e_1 \dgal{a,b}'\he \otimes \he \dgal{a,b}''\he, \,\,\, \dgalfusU{c,d}=0, \\ 
  &\dgalindD{\psi(c),\psi(d)}=\Phi_2 e_{21}\dgal{a,b}, \,\,\, \dgalfusD{\psi(c),\psi(d)} =0\,.
 \end{aligned}
\end{equation}

\textit{Case 1.3 : $d=b$, $b \in \he A e_1$.} 
\begin{equation}
 \begin{aligned}
  &\dgalindU{c,d}=\he \dgal{a,b}'\he \otimes \he \dgal{a,b}''e_1, \,\,\, \dgalfusU{c,d}=0, \\
  &\dgalindD{\psi(c),\psi(d)}=\dgal{a,b}e_{12}\Phi_2^{-1}, \,\,\, \dgalfusD{\psi(c),\psi(d)} =0\,.
 \end{aligned}
\end{equation}

\textit{Case 1.4 : $d=b$, $b \in e_1 A e_1$.} 
\begin{equation}
 \begin{aligned}
  &\dgalindU{c,d}=e_1 \dgal{a,b}'\he \otimes \he \dgal{a,b}''e_1, \,\,\, \dgalfusU{c,d}=0, \\ 
  &\dgalindD{\psi(c),\psi(d)}=\Phi_2 e_{21}\dgal{a,b}e_{12}\Phi_2^{-1}, \,\,\, \dgalfusD{\psi(c),\psi(d)} =0\,.
 \end{aligned}
\end{equation}

\textit{Case 1.5 : $d=e_{12}b$, $b \in e_2 A \he$.} 
\begin{equation}
 \begin{aligned}
  &\dgalindU{c,d}=e_{12} \dgal{a,b}'\he \otimes \he \dgal{a,b}''\he, \,\,\, \dgalfusU{c,d}=0, \\
  &\dgalindD{\psi(c),\psi(d)}=\dgal{a,b}, \,\,\, \dgalfusD{\psi(c),\psi(d)} =0\,.
 \end{aligned}
\end{equation}

\textit{Case 1.6 : $d=e_{12}b$, $b \in e_2 A e_1$.} 
\begin{equation}
 \begin{aligned}
  &\dgalindU{c,d}=e_{12} \dgal{a,b}'\he \otimes \he \dgal{a,b}''e_1, \,\,\, \dgalfusU{c,d}=0, \\
  &\dgalindD{\psi(c),\psi(d)}=\dgal{a,b}e_{12}\Phi_2^{-1}, \,\,\, \dgalfusD{\psi(c),\psi(d)} =0\,.
 \end{aligned}
\end{equation}

\textit{Case 1.7 : $d=be_{21}$, $b \in \he A e_2$.} 
\begin{equation}
 \begin{aligned}
  &\dgalindU{c,d}=\he \dgal{a,b}'\he \otimes \he \dgal{a,b}''e_{21}, \,\,\, \dgalfusU{c,d}=0, \\
  &\dgalindD{\psi(c),\psi(d)}=\dgal{a,b}, \,\,\, \dgalfusD{\psi(c),\psi(d)} =0\,.
 \end{aligned}
\end{equation}

\textit{Case 1.8 : $d=be_{21}$, $b \in e_1 A e_2$.} 
\begin{equation}
 \begin{aligned}
  &\dgalindU{c,d}=e_1 \dgal{a,b}'\he \otimes \he \dgal{a,b}''e_{21}, \,\,\, \dgalfusU{c,d}=0, \\ 
  &\dgalindD{\psi(c),\psi(d)}=\Phi_2 e_{21}\dgal{a,b}, \,\,\, \dgalfusD{\psi(c),\psi(d)} =0\,.
 \end{aligned}
\end{equation}

\textit{Case 1.9 : $d=e_{12}be_{21}$, $b \in e_2 A e_2$.} 
\begin{equation}
 \begin{aligned}
  &\dgalindU{c,d}=e_{12}\dgal{a,b}'\he \otimes \he \dgal{a,b}''e_{21}, \,\,\dgalfusU{c,d}=0, \\
  &\dgalindD{\psi(c),\psi(d)}=\dgal{a,b}, \,\, \dgalfusD{\psi(c),\psi(d)} =0.
 \end{aligned}
\end{equation}

\subsection{} We consider the second specialisation of a generator of first type, that is $c=a$ for $a\in e_1 A \he$. 

\textit{Case 2.1 : $d=b$, $b\in e_1 A \he$.} This was done in the proof of Lemma \ref{L:IsoFusqHam0}. 

\textit{Case 2.2 : $d=b$, $b \in \he A e_1$.} 
\begin{equation}
 \begin{aligned}
  \dgalindU{c,d}=&\he \dgal{a,b}'\he \otimes e_1\dgal{a,b}''e_1, \,\,\, \dgalfusU{c,d}=0, \\ 
  \dgalindD{\psi(c),\psi(d)}=&\Phi_2 e_{21}\ast\dgal{a,b}e_{12}\Phi_2^{-1}-\frac12 (b e_{12}\Phi_2 e_{21}a \otimes \Phi_2^{-1} - b e_{12}\Phi_2^{-1}e_{21}a \otimes \Phi_2), \\
  \dgalfusD{\psi(c),\psi(d)} =&\frac12(b e_{12}\Phi_2 e_{21}a \otimes \Phi_2^{-1} - b e_{12}\Phi_2^{-1}e_{21}a \otimes \Phi_2)\,.
 \end{aligned}
\end{equation}

\textit{Case 2.3 : $d=b$, $b \in e_1 A e_1$.} 
\begin{equation}
 \begin{aligned}
  \dgalindU{c,d}=&e_1 \dgal{a,b}'\he \otimes e_1\dgal{a,b}''e_1, \quad \dgalfusU{c,d}=0, \\ 
  \dgalindD{\psi(c),\psi(d)}=&\Phi_2 e_{21}\ast\Phi_2 e_{21}\dgal{a,b}e_{12}\Phi_2^{-1}\\
  &-\frac12(\Phi_2 e_{21}be_{12}\Phi_2 e_{21} a \otimes \Phi_2^{-1}- \Phi_2^2 e_{21}a \otimes e_{21}be_{12}\Phi_2^{-1})\\
  &-\frac12(e_{21}a \otimes \Phi_2^{2} e_{21}be_{12}\Phi_2^{-1} - \Phi_2e_{21}be_{12}\Phi_2^{-1}e_{21}a \otimes \Phi_2), \\
  \dgalfusD{\psi(c),\psi(d)} =&
  +\frac12(\Phi_2 e_{21}be_{12}\Phi_2 e_{21} a \otimes \Phi_2^{-1}- \Phi_2^2 e_{21}a \otimes e_{21}be_{12}\Phi_2^{-1}) \\
  &+\frac12(e_{21}a \otimes \Phi_2^{2} e_{21}be_{12}\Phi_2^{-1} - \Phi_2e_{21}be_{12}\Phi_2^{-1}e_{21}a \otimes \Phi_2)\,.
 \end{aligned}
\end{equation}

\textit{Case 2.4 : $d=e_{12}b$, $b \in e_2 A \he$.} 
\begin{equation}
 \begin{aligned}
  \dgalindU{c,d}=&e_{12}\dgal{a,b}'\he \otimes e_1\dgal{a,b}''\he, \quad \dgalfusU{c,d}=-\frac12 a \otimes e_{12}b, \\ 
  \dgalindD{\psi(c),\psi(d)}=&\Phi_2 e_{21}\ast\dgal{a,b}-\frac12(e_{21}a \otimes \Phi_2b + \Phi_{2}e_{21}a \otimes b), \\
  \dgalfusD{\psi(c),\psi(d)} =&\frac12 e_{21}a \otimes \Phi_2 b\,.
 \end{aligned}
\end{equation}

\textit{Case 2.5 : $d=e_{12}b$, $b \in e_2 A e_1$.} 
\begin{equation}
 \begin{aligned}
  \dgalindU{c,d}=&e_{12}\dgal{a,b}'\he \otimes e_1\dgal{a,b}''e_1, \quad \dgalfusU{c,d}=-\frac12 a \otimes e_{12}b, \\ 
  \dgalindD{\psi(c),\psi(d)}=&\Phi_2 e_{21}\ast\dgal{a,b}e_{12}\Phi_2^{-1}-\frac12 \Phi_{2}e_{21}a \otimes b e_{12}\Phi_2^{-1}\\
  &-\frac12(b e_{12} \Phi_2 e_{21}a \otimes \Phi_2^{-1}-be_{12}\Phi_2^{-1} e_{21}a \otimes \Phi_2  + e_{21}a \otimes \Phi_2b e_{12}\Phi_2^{-1} )\,,\\
  \dgalfusD{\psi(c),\psi(d)} =&+\frac12 (be_{12}\Phi_2e_{21}a \otimes \Phi_2^{-1}-be_{12}\Phi_2^{-1}e_{21}a\otimes \Phi_2 + e_{21}a \otimes \Phi_2be_{12}\Phi_2^{-1})\,.
 \end{aligned}
\end{equation}

\textit{Case 2.6 : $d=be_{21}$, $b \in \he A e_2$.}\footnote{This is the first case where one needs to be careful when computing $\psi^{\otimes 2}\dgalfusU{c,d}$ in order to verify \eqref{Eq:FusqHam}. It is given in this case by $\frac12 \psi(b e_{21}a) \otimes \psi(e_1)$, and we have by definition of $\psi$ that  $\psi(e_1)=\Phi_2 e_{21}e_{12}\Phi_2^{-1}=e_2$. For the first factor, 
we need to remark that $be_{21}a$ is \emph{not} a specialisation of generators of $A_1$ as defined earlier, but it is a product of two. Hence, $\psi(b e_{21}a)=\psi(be_{21})\psi(a)=b \Phi_2 e_{21} a$.} 
\begin{equation}
 \begin{aligned}
  \dgalindU{c,d}=&\he\dgal{a,b}'\he \otimes e_1\dgal{a,b}''e_{21}, \quad \dgalfusU{c,d}=\frac12 b e_{21}a \otimes e_1, \\ 
  \dgalindD{\psi(c),\psi(d)}=&\Phi_2 e_{21}\ast\dgal{a,b}+ \frac12 (b e_{21}a \otimes \Phi_2 + b\Phi_2 e_{21}a \otimes e_2)\,,\\
  \dgalfusD{\psi(c),\psi(d)} =&-\frac12 b e_{21}a \otimes \Phi_2\,.
 \end{aligned}
\end{equation}

\textit{Case 2.7 : $d=be_{21}$, $b \in e_1 A e_2$.}
\begin{equation}
 \begin{aligned}
  \dgalindU{c,d}=&e_1\dgal{a,b}'\he \otimes e_1\dgal{a,b}''e_{21}, \quad \dgalfusU{c,d}=\frac12 b e_{21}a \otimes e_1, \\ 
  \dgalindD{\psi(c),\psi(d)}=&\Phi_2 e_{21}\ast \Phi_2 e_{21}\dgal{a,b}+ \frac12 \Phi_2 e_{21}b \Phi_2 e_{21}a \otimes e_2 \\
&-\frac12 (e_{21}a \otimes \Phi_2^2 e_{21}b - \Phi_2^2 e_{21}a \otimes e_{21} b - \Phi_2 e_{21}b e_{21}a \otimes \Phi_2) \,,\\  
  \dgalfusD{\psi(c),\psi(d)} =&+\frac12 (e_{21}a \otimes \Phi_2^2 e_{21}b - \Phi_2^2 e_{21}a \otimes e_{21} b - \Phi_2 e_{21}b e_{21}a \otimes \Phi_2)\,.
 \end{aligned}
\end{equation}

\textit{Case 2.8 : $d=e_{12}be_{21}$, $b \in e_2 A e_2$.}
\begin{equation}
 \begin{aligned}
  \dgalindU{c,d}=&e_{12}\dgal{a,b}'\he \otimes e_1\dgal{a,b}''e_{21}, \quad \dgalfusU{c,d}=\frac12 (e_{12}be_{21}a \otimes e_1-a \otimes e_{12}be_{21}), \\ 
  \dgalindD{\psi(c),\psi(d)}=&\Phi_2 e_{21}\ast \dgal{a,b}+ \frac12(b \Phi_2e_{21}a \otimes e_2 - \Phi_2 e_{21}a \otimes b)\\  
  &-\frac12(e_{21}a \otimes \Phi_2 b - b e_{21}a \otimes \Phi_2)\,, \\
  \dgalfusD{\psi(c),\psi(d)} =&+\frac12(e_{21}a \otimes \Phi_2 b - b e_{21}a \otimes \Phi_2)\,.
 \end{aligned}
\end{equation}

\subsection{} We consider the third specialisation of a generator of first type, that is $c=a$ for $a\in \he A e_1$.

\textit{Case 3.1 : $d=b$, $b\in \he A e_1$.} 
\begin{equation}
 \begin{aligned}
  \dgalindU{c,d}=&\he\dgal{a,b}'e_1 \otimes \he\dgal{a,b}''e_1, \quad \dgalfusU{c,d}=0, \\ 
  \dgalindD{\psi(c),\psi(d)}=&\dgal{a,b}e_{12}\Phi_2^{-1}\ast e_{12}\Phi_2^{-1} - \frac12( b e_{12} \Phi_2^{-2} \otimes a e_{12} - b e_{12} \otimes a e_{12}\Phi_2^{-2}) \,, \\
  \dgalfusD{\psi(c),\psi(d)} =&+\frac12( b e_{12} \Phi_2^{-2} \otimes a e_{12} - b e_{12} \otimes a e_{12}\Phi_2^{-2}) \,.
 \end{aligned}
\end{equation}

\textit{Case 3.2 : $d=b$, $b\in e_1 A e_1$.} 

\begin{equation}
 \begin{aligned}
  \dgalindU{c,d}=&e_1\dgal{a,b}'e_1 \otimes \he\dgal{a,b}''e_1, \quad \dgalfusU{c,d}=0, \\ 
  \dgalindD{\psi(c),\psi(d)}=&\Phi_2 e_{21}\dgal{a,b}e_{12}\Phi_2^{-1}\ast e_{12}\Phi_2^{-1} \\
  &-\frac12(\Phi_2 \otimes a e_{12} \Phi_2^{-1}e_{21}be_{12}\Phi_2^{-1}-\Phi_2^{-1} \otimes a e_{12} \Phi_2e_{21}be_{12}\Phi_2^{-1}) \\
  &-\frac12(\Phi_2 e_{21}be_{12}\Phi_2^{-2} \otimes a e_{12} -\Phi_2 e_{21}be_{12} \otimes a e_{12} \Phi_2^{-2})\,, \\
  \dgalfusD{\psi(c),\psi(d)} =&+\frac12(\Phi_2 \otimes a e_{12} \Phi_2^{-1}e_{21}be_{12}\Phi_2^{-1}-\Phi_2^{-1} \otimes a e_{12} \Phi_2e_{21}be_{12}\Phi_2^{-1}) \\
  &+\frac12(\Phi_2 e_{21}be_{12}\Phi_2^{-2} \otimes a e_{12} -\Phi_2 e_{21}be_{12} \otimes a e_{12} \Phi_2^{-2}) \,.
 \end{aligned}
\end{equation}

\textit{Case 3.3 : $d=e_{12}b$, $b\in e_2 A \he$.} 

\begin{equation}
 \begin{aligned}
  \dgalindU{c,d}=&e_{12}\dgal{a,b}'e_1 \otimes \he\dgal{a,b}''\he, \quad \dgalfusU{c,d}=\frac12 e_1 \otimes a e_{12}b, \\ 
  \dgalindD{\psi(c),\psi(d)}=&\dgal{a,b}\ast e_{12}\Phi_2^{-1} + \frac12 (e_2 \otimes a e_{12}\Phi_2^{-1}b + \Phi_2^{-1}\otimes a e_{12}b)\,, \\
  \dgalfusD{\psi(c),\psi(d)} =& -\frac12 \Phi_2^{-1}\otimes a e_{12}b\,.
 \end{aligned}
\end{equation}

\textit{Case 3.4 : $d=e_{12}b$, $b\in e_2 A e_1$.} 

\begin{equation}
 \begin{aligned}
  \dgalindU{c,d}=&e_{12}\dgal{a,b}'e_1 \otimes \he\dgal{a,b}''e_1, \quad \dgalfusU{c,d}=\frac12 e_1 \otimes a e_{12}b, \\ 
  \dgalindD{\psi(c),\psi(d)}=&\dgal{a,b}e_{12}\Phi_2^{-1}\ast e_{12}\Phi_2^{-1} +\frac12 e_2 \otimes a e_{12} \Phi_2^{-1} b e_{12} \Phi_2^{-1}\\
  &-\frac12(b e_{12}\Phi_2^{-2}\otimes a e_{12} - b e_{12} \otimes a e_{12} \Phi_2^{-2} - \Phi_2^{-1} \otimes a e_{12}b e_{12}\Phi_2^{-1})\,, \\
  \dgalfusD{\psi(c),\psi(d)} =& +\frac12(b e_{12}\Phi_2^{-2}\otimes a e_{12} - b e_{12} \otimes a e_{12} \Phi_2^{-2} - \Phi_2^{-1} \otimes a e_{12}b e_{12}\Phi_2^{-1})\,.
 \end{aligned}
\end{equation}

\textit{Case 3.5 : $d=be_{21}$, $b\in \he A e_2$.} 

\begin{equation}
 \begin{aligned}
  \dgalindU{c,d}=&\he\dgal{a,b}'e_1 \otimes \he\dgal{a,b}''e_{21}, \quad \dgalfusU{c,d}=-\frac12 b e_{21}\otimes a, \\ 
  \dgalindD{\psi(c),\psi(d)}=&\dgal{a,b}\ast e_{12}\Phi_2^{-1} -\frac12(b \Phi_2^{-1} \otimes a e_{12} + b \otimes a e_{12}\Phi_2^{-1})   \,, \\
  \dgalfusD{\psi(c),\psi(d)} =& +\frac12 b \Phi_2^{-1} \otimes a e_{12}\,.
 \end{aligned}
\end{equation}

\textit{Case 3.6 : $d=be_{21}$, $b\in e_1 A e_2$.} 

\begin{equation}
 \begin{aligned}
  \dgalindU{c,d}=&e_1\dgal{a,b}'e_1 \otimes \he\dgal{a,b}''e_{21}, \quad \dgalfusU{c,d}=-\frac12 b e_{21}\otimes a, \\ 
  \dgalindD{\psi(c),\psi(d)}=&\Phi_2 e_{21}\dgal{a,b}\ast e_{12}\Phi_2^{-1}   -\frac12 \Phi_2 e_{21}b \otimes a e_{12}\Phi_2^{-1} \\
  &-\frac12(\Phi_2 e_{21}b\Phi_2^{-1} \otimes a e_{12} + \Phi_2 \otimes a e_{12}\Phi_2^{-1}e_{21}b - \Phi_2^{-1}\otimes a e_{12}\Phi_2 e_{21}b)\,, \\
  \dgalfusD{\psi(c),\psi(d)} =& +\frac12(\Phi_2 e_{21}b\Phi_2^{-1} \otimes a e_{12} + \Phi_2 \otimes a e_{12}\Phi_2^{-1}e_{21}b - \Phi_2^{-1}\otimes a e_{12}\Phi_2 e_{21}b)\,.
 \end{aligned}
\end{equation}

\textit{Case 3.7 : $d=e_{12}be_{21}$, $b\in e_2 A e_2$.} 

\begin{equation}
 \begin{aligned}
  \dgalindU{c,d}=&e_{12}\dgal{a,b}'e_1 \otimes \he\dgal{a,b}''e_{21}, \quad \dgalfusU{c,d}=\frac12(e_1 \otimes a e_{12}be_{21} - e_{12}b e_{21}\otimes a), \\ 
  \dgalindD{\psi(c),\psi(d)}=&\dgal{a,b}\ast e_{12}\Phi_2^{-1} +\frac12(e_2 \otimes a e_{12}\Phi_2^{-1}b - b \otimes a e_{12} \Phi_2^{-1})\\
  &-\frac12(b \Phi_2^{-1} \otimes a e_{12} -\Phi_2^{-1} \otimes a e_{12}b)\,, \\
  \dgalfusD{\psi(c),\psi(d)} =& +\frac12(b \Phi_2^{-1} \otimes a e_{12} -\Phi_2^{-1} \otimes a e_{12}b)\,.
 \end{aligned}
\end{equation}

\subsection{} We consider the fourth specialisation of a generator of first type, that is $c=a$ for $a\in e_1 A e_1$.

\textit{Case 4.1 : $d=b$, $b\in e_1 A e_1$.} 
\begin{equation}
 \begin{aligned}
  \dgalindU{c,d}=&e_1\dgal{a,b}'e_1 \otimes e_1\dgal{a,b}''e_1, \quad \dgalfusU{c,d}=0, \\ 
  \dgalindD{\psi(c),\psi(d)}=&\Phi_2 e_{21} \ast \Phi_2 e_{21}\dgal{a,b}e_{12}\Phi_2^{-1}\ast e_{12}\Phi_2^{-1} \\
  &-\frac12( \Phi_2^2 e_{21}a e_{12}\Phi_2^{-1} \otimes e_{21}be_{12}\Phi_2^{-1}-e_{21}a e_{12}\Phi_2^{-1} \otimes \Phi_2^2 e_{21}be_{12}\Phi_2^{-1}) \\
  &-\frac12(\Phi_2 e_{21}b e_{12}\otimes \Phi_2 e_{21}ae_{12}\Phi_2^{-2}-\Phi_2 e_{21}b e_{12}\Phi_2^{-2} \otimes \Phi_2 e_{21}ae_{12}) \\
  &- \frac12(\Phi_2 e_{21}be_{12}\Phi_2^{-1}e_{21}ae_{12}\Phi_2^{-1}\otimes \Phi_2 - \Phi_2 e_{21}be_{12}\Phi_2e_{21}ae_{12}\Phi_2^{-1}\otimes \Phi_2^{-1})   \\
  &-\frac12 (\Phi_2^{-1}\otimes \Phi_2 e_{21}ae_{12}\Phi_2e_{21}be_{12}\Phi_2^{-1} - \Phi_2 \otimes \Phi_2 e_{21}ae_{12}\Phi_2^{-1}e_{21}be_{12}\Phi_2^{-1})\,, \\
  \dgalfusD{\psi(c),\psi(d)} =&+\frac12( \Phi_2^2 e_{21}a e_{12}\Phi_2^{-1} \otimes e_{21}be_{12}\Phi_2^{-1}-e_{21}a e_{12}\Phi_2^{-1} \otimes \Phi_2^2 e_{21}be_{12}\Phi_2^{-1}) \\
  &+\frac12(\Phi_2 e_{21}b e_{12}\otimes \Phi_2 e_{21}ae_{12}\Phi_2^{-2}-\Phi_2 e_{21}b e_{12}\Phi_2^{-2} \otimes \Phi_2 e_{21}ae_{12}) \\
  &+ \frac12(\Phi_2 e_{21}be_{12}\Phi_2^{-1}e_{21}ae_{12}\Phi_2^{-1}\otimes \Phi_2 - \Phi_2 e_{21}be_{12}\Phi_2e_{21}ae_{12}\Phi_2^{-1}\otimes \Phi_2^{-1})   \\
  &+\frac12 (\Phi_2^{-1}\otimes \Phi_2 e_{21}ae_{12}\Phi_2e_{21}be_{12}\Phi_2^{-1} - \Phi_2 \otimes \Phi_2 e_{21}ae_{12}\Phi_2^{-1}e_{21}be_{12}\Phi_2^{-1})   \,.
 \end{aligned}
\end{equation}

\textit{Case 4.2 : $d=e_{12}b$, $b\in e_2 A \he$.} 
\begin{equation}
 \begin{aligned}
  \dgalindU{c,d}=&e_{12}\dgal{a,b}'e_1 \otimes e_1\dgal{a,b}''\he, \quad \dgalfusU{c,d}=\frac12(e_1 \otimes a e_{12}b - a \otimes e_{12}b), \\ 
  \dgalindD{\psi(c),\psi(d)}=&\Phi_2 e_{21} \ast \dgal{a,b}\ast e_{12}\Phi_2^{-1} 
  +\frac12(e_2 \otimes \Phi_2 e_{21}ae_{12}\Phi_2^{-1}b - \Phi_2 e_{21}ae_{12}\Phi_2^{-1} \otimes b)\\
  &-\frac12(e_{21}ae_{12}\Phi_2^{-1}\otimes \Phi_2 b - \Phi_2^{-1}\otimes \Phi_2 e_{21}ae_{12}b)\,, \\
  \dgalfusD{\psi(c),\psi(d)} =&+\frac12(e_{21}ae_{12}\Phi_2^{-1}\otimes \Phi_2 b - \Phi_2^{-1}\otimes \Phi_2 e_{21}ae_{12}b)   \,.
 \end{aligned}
\end{equation}

\textit{Case 4.3 : $d=e_{12}b$, $b\in e_2 A e_1$.} 
\begin{equation}
 \begin{aligned}
  \dgalindU{c,d}=&e_{12}\dgal{a,b}'e_1 \otimes e_1\dgal{a,b}''e_1, \quad \dgalfusU{c,d}=\frac12(e_1 \otimes a e_{12}b - a \otimes e_{12}b), \\ 
  \dgalindD{\psi(c),\psi(d)}=&\Phi_2 e_{21} \ast \dgal{a,b}e_{12}\Phi_2^{-1} \ast e_{12}\Phi_2^{-1} \\
  &+\frac12(e_2 \otimes \Phi_2 e_{21}ae_{12}\Phi_2^{-1}be_{12}\Phi_2^{-1} - \Phi_2 e_{21}ae_{12}\Phi_2^{-1} \otimes be_{12}\Phi_2^{-1}) \\
  &-\frac12(e_{21}ae_{12}\Phi_2^{-1}\otimes \Phi_2 b e_{12}\Phi_2^{-1} - \Phi_2^{-1} \otimes \Phi_2 e_{21}a e_{12}be_{12}\Phi_2^{-1}) \\
  &-\frac12 (b e_{12}\Phi_2 e_{21}ae_{12}\Phi_2^{-1}\otimes \Phi_2^{-1}-b e_{12}\Phi_2^{-1}e_{21}a e_{12}\Phi_2^{-1}\otimes \Phi_2) \\
  &-\frac12(b e_{12}\Phi_2^{-2} \otimes \Phi_2 e_{21}ae_{12}- be_{12}\otimes \Phi_2 e_{21}a e_{12}\Phi_2^{-2}  )\,, \\
  \dgalfusD{\psi(c),\psi(d)} =&+\frac12(e_{21}ae_{12}\Phi_2^{-1}\otimes \Phi_2 b e_{12}\Phi_2^{-1} - \Phi_2^{-1} \otimes \Phi_2 e_{21}a e_{12}be_{12}\Phi_2^{-1}) \\
  &+\frac12 (b e_{12}\Phi_2 e_{21}ae_{12}\Phi_2^{-1}\otimes \Phi_2^{-1}-b e_{12}\Phi_2^{-1}e_{21}a e_{12}\Phi_2^{-1}\otimes \Phi_2) \\
  &+\frac12(b e_{12}\Phi_2^{-2} \otimes \Phi_2 e_{21}ae_{12}- be_{12}\otimes \Phi_2 e_{21}a e_{12}\Phi_2^{-2}  )   \,.
 \end{aligned}
\end{equation}

\textit{Case 4.4 : $d=be_{21}$, $b\in \he A e_2$.} 
\begin{equation}
 \begin{aligned}
  \dgalindU{c,d}=&\he\dgal{a,b}'e_1 \otimes e_1\dgal{a,b}''e_{21}, \quad \dgalfusU{c,d}=\frac12(b e_{21}a \otimes e_1 - b e_{21}\otimes a e_1), \\ 
  \dgalindD{\psi(c),\psi(d)}=&\Phi_2 e_{21} \ast \dgal{a,b}\ast e_{12}\Phi_2^{-1} 
  +\frac12(b \Phi_2 e_{21}ae_{12} \Phi_2^{-1}\otimes e_2-b \otimes \Phi_2 e_{21}ae_{12} \Phi_2^{-1})\\
  &-\frac12(b \Phi_2^{-1}\otimes \Phi_2 e_{21}ae_{12} - b e_{21}ae_{12}\Phi_2^{-1}\otimes \Phi_2)\,, \\
  \dgalfusD{\psi(c),\psi(d)} =&+\frac12(b \Phi_2^{-1}\otimes \Phi_2 e_{21}ae_{12} - b e_{21}ae_{12}\Phi_2^{-1}\otimes \Phi_2)   \,.
 \end{aligned}
\end{equation}

\textit{Case 4.5 : $d=be_{21}$, $b\in e_1 A e_2$.} 
\begin{equation}
 \begin{aligned}
  \dgalindU{c,d}=&e_1\dgal{a,b}'e_1 \otimes e_1\dgal{a,b}''e_{21}, \quad \dgalfusU{c,d}=\frac12(b e_{21}a \otimes e_1 - b e_{21}\otimes a e_1), \\ 
  \dgalindD{\psi(c),\psi(d)}=&\Phi_2 e_{21} \ast \Phi_2 e_{21}\dgal{a,b}\ast e_{12}\Phi_2^{-1} \\
  &+\frac12(\Phi_2 e_{21} b \Phi_2 e_{21}ae_{12} \Phi_2^{-1}\otimes e_2-\Phi_2 e_{21} b \otimes \Phi_2 e_{21}ae_{12} \Phi_2^{-1})\\
  &- \frac12 (e_{21}ae_{12}\Phi_2^{-1} \otimes \Phi_2^2 e_{21}b - \Phi_2^2 e_{21}ae_{12}\Phi_2^{-1} \otimes e_{21}b) \\
  & - \frac12(\Phi_2 \otimes \Phi_2 e_{21}ae_{12}\Phi_2^{-1}e_{21}b - \Phi_2^{-1} \otimes \Phi_2 e_{21}ae_{12}\Phi_2 e_{21}b ) \\
  &-\frac12(\Phi_2 e_{21}b \Phi_2^{-1} \otimes \Phi_2 e_{21}a e_{12} - \Phi_2 e_{21}b e_{21}ae_{12}\Phi_2^{-1}\otimes \Phi_2)\,, \\
  \dgalfusD{\psi(c),\psi(d)} =&+\frac12 (e_{21}ae_{12}\Phi_2^{-1} \otimes \Phi_2^2 e_{21}b - \Phi_2^2 e_{21}ae_{12}\Phi_2^{-1} \otimes e_{21}b) \\
  & + \frac12(\Phi_2 \otimes \Phi_2 e_{21}ae_{12}\Phi_2^{-1}e_{21}b - \Phi_2^{-1} \otimes \Phi_2 e_{21}ae_{12}\Phi_2 e_{21}b ) \\
  &+\frac12(\Phi_2 e_{21}b \Phi_2^{-1} \otimes \Phi_2 e_{21}a e_{12} - \Phi_2 e_{21}b e_{21}ae_{12}\Phi_2^{-1}\otimes \Phi_2)   \,.
 \end{aligned}
\end{equation}

\textit{Case 4.6 : $d=e_{12}be_{21}$, $b\in e_2 A e_2$.} 
\begin{equation}
 \begin{aligned}
  \dgalindU{c,d}=&e_{12}\dgal{a,b}'e_1 \otimes e_1\dgal{a,b}''e_{21}\,, \\
  \dgalfusU{c,d}=&\frac12(e_{12}be_{21} a \otimes e_1 + e_1 \otimes a e_{12}be_{21} - e_{12}b e_{21}\otimes a - a \otimes e_{12}be_{21}), \\ 
  \dgalindD{\psi(c),\psi(d)}=&\Phi_2 e_{21} \ast \dgal{a,b}\ast e_{12}\Phi_2^{-1} \\
  &+\frac12(b \Phi_2 e_{21}ae_{12} \Phi_2^{-1} \otimes e_2 + e_2 \otimes \Phi_2 e_{21}ae_{12}\Phi_2^{-1} b) \\
  &-\frac12(b \otimes \Phi_2 e_{21}ae_{12}\Phi_2^{-1}  + \Phi_2 e_{21}ae_{12} \Phi_2^{-1}\otimes b) \\
  &-\frac12 (b \Phi_2^{-1} \otimes \Phi_2 e_{21}ae_{12} + e_{21}ae_{12}\Phi_2^{-1} \otimes \Phi_2b)\\ 
  &+\frac12( b e_{21}ae_{12}\Phi_2^{-1} \otimes \Phi_2 +\Phi_2^{-1} \otimes \Phi_2 e_{21}ae_{12} b) \,,\\
  \dgalfusD{\psi(c),\psi(d)} =&+\frac12 (b \Phi_2^{-1} \otimes \Phi_2 e_{21}ae_{12} + e_{21}ae_{12}\Phi_2^{-1} \otimes \Phi_2b)\\ 
  &-\frac12( b e_{21}ae_{12}\Phi_2^{-1} \otimes \Phi_2 +\Phi_2^{-1} \otimes \Phi_2 e_{21}ae_{12} b)   \,.
 \end{aligned}
\end{equation}

\subsection{} We consider the first specialisation of a generator of second type, that is $c=e_{12}a$ for $a\in e_2 A \he$. 

\textit{Case 5.1 : $d=e_{12}b$, $b\in e_2 A \he$.} 
\begin{equation}
 \begin{aligned}
  &\dgalindU{c,d}=e_{12}\dgal{a,b}'\he \otimes e_{12}\dgal{a,b}''\he, \quad \dgalfusU{c,d}=0, \\ 
  &\dgalindD{\psi(c),\psi(d)}=\dgal{a,b} \,, \quad  \dgalfusD{\psi(c),\psi(d)} =0 \,.
 \end{aligned}
\end{equation}

\textit{Case 5.2 : $d=e_{12}b$, $b\in e_2 A e_1$.} 
\begin{equation}
 \begin{aligned}
  \dgalindU{c,d}=&e_{12}\dgal{a,b}'\he \otimes e_{12}\dgal{a,b}''e_1, \quad \dgalfusU{c,d}=-\frac12 e_{12}be_{12}a \otimes e_1, \\ 
  \dgalindD{\psi(c),\psi(d)}=&\dgal{a,b}e_{12}\Phi_2^{-1}-\frac12(b e_{12}a \otimes \Phi_2^{-1}+b e_{12}\Phi_2^{-1}a\otimes e_2)  \,,\\ 
  \dgalfusD{\psi(c),\psi(d)} =&+\frac12 b e_{12}a\otimes \Phi_2^{-1} \,.
 \end{aligned}
\end{equation}

\textit{Case 5.3 : $d=be_{21}$, $b\in \he A e_2$.} 
\begin{equation}
 \begin{aligned}
  &\dgalindU{c,d}=\he \dgal{a,b}'\he \otimes e_{12}\dgal{a,b}''e_{21}, \quad \dgalfusU{c,d}=0, \\ 
  &\dgalindD{\psi(c),\psi(d)}=\dgal{a,b} \,, \quad  \dgalfusD{\psi(c),\psi(d)} =0 \,.
 \end{aligned}
\end{equation}

\textit{Case 5.4 : $d=be_{21}$, $b\in e_1 A e_2$.} 
\begin{equation}
 \begin{aligned}
  \dgalindU{c,d}=&e_1\dgal{a,b}'\he \otimes e_{12}\dgal{a,b}''e_{21}, \quad \dgalfusU{c,d}=\frac12 e_{12}a \otimes b e_{21}, \\ 
  \dgalindD{\psi(c),\psi(d)}=&\Phi_2 e_{21}\dgal{a,b}+\frac12(\Phi_2 a \otimes e_{21}b + a \otimes \Phi_2 e_{21}b)  \,,\\ 
  \dgalfusD{\psi(c),\psi(d)} =&-\frac12 \Phi_2 a \otimes e_{21}b  \,.
 \end{aligned}
\end{equation}

\textit{Case 5.5 : $d=e_{12}be_{21}$, $b\in e_2 A e_2$.} 
\begin{equation}
 \begin{aligned}
  &\dgalindU{c,d}=e_{12} \dgal{a,b}'\he \otimes e_{12}\dgal{a,b}''e_{21}, \quad \dgalfusU{c,d}=0, \\ 
  &\dgalindD{\psi(c),\psi(d)}=\dgal{a,b} \,, \quad  \dgalfusD{\psi(c),\psi(d)} =0 \,.
 \end{aligned}
\end{equation}

\subsection{} We consider the second specialisation of a generator of second type, that is $c=e_{12}a$ for $a\in e_2 A e_1$. 

\textit{Case 6.1 : $d=e_{12}b$, $b\in e_2 A e_1$.} 
\begin{equation}
 \begin{aligned}
  \dgalindU{c,d}=&e_{12}\dgal{a,b}'e_1 \otimes e_{12}\dgal{a,b}''e_1,\\ 
  \dgalfusU{c,d}=&\frac12(e_1 \otimes e_{12}a e_{12}b - e_{12}b e_{12}a \otimes e_1), \\ 
  \dgalindD{\psi(c),\psi(d)}=& \dgal{a,b}e_{12}\Phi_2^{-1}\ast e_{12}\Phi_2^{-1} \\
  &+\frac12(e_2 \otimes a e_{12}\Phi_2^{-1}b e_{12}\Phi_2^{-1} - b e_{12}\Phi_2^{-1}a e_{12}\Phi_2^{-1} \otimes e_2) \\
  &-\frac12(b e_{12}ae_{12}\Phi_2^{-1}\otimes \Phi_2^{-1} - \Phi_2^{-1} \otimes a e_{12}b e_{12}\Phi_2^{-1}) \\
  &-\frac12 (b e_{12} \Phi_2^{-2} \otimes a e_{12} - b e_{12} \otimes a e_{12} \Phi_2^{-2})\,,\\ 
  \dgalfusD{\psi(c),\psi(d)} =&+\frac12(b e_{12}ae_{12}\Phi_2^{-1}\otimes \Phi_2^{-1} - \Phi_2^{-1} \otimes a e_{12}b e_{12}\Phi_2^{-1}) \\
  &+\frac12 (b e_{12} \Phi_2^{-2} \otimes a e_{12} - b e_{12} \otimes a e_{12} \Phi_2^{-2})\,.
 \end{aligned}
\end{equation}

\textit{Case 6.2 : $d=be_{21}$, $b\in \he A e_2$.} 
\begin{equation}
 \begin{aligned}
  \dgalindU{c,d}=&\he\dgal{a,b}'e_1 \otimes e_{12}\dgal{a,b}''e_{21},\quad  \dgalfusU{c,d}=-\frac12 b e_{21} \otimes e_{12}a, \\ 
  \dgalindD{\psi(c),\psi(d)}=& \dgal{a,b}\ast e_{12}\Phi_2^{-1} -\frac12(b \otimes a e_{12}\Phi_2^{-1} + b \Phi_2^{-1} \otimes a e_{12}   )\,,\\ 
  \dgalfusD{\psi(c),\psi(d)} =&+\frac12 b \Phi_2^{-1} \otimes a e_{12}\,.
 \end{aligned}
\end{equation}

\textit{Case 6.3 : $d=be_{21}$, $b\in e_1 A e_2$.} 
\begin{equation}
 \begin{aligned}
  \dgalindU{c,d}=&e_1\dgal{a,b}'e_1 \otimes e_{12}\dgal{a,b}''e_{21},\quad  \dgalfusU{c,d}=\frac12 (e_{12}a \otimes b e_{21}-b e_{21} \otimes e_{12}a), \\ 
  \dgalindD{\psi(c),\psi(d)}=& \Phi_2 e_{21}\dgal{a,b}\ast e_{12}\Phi_2^{-1} +\frac12(a e_{12}\Phi_2^{-1}\otimes \Phi_2 e_{21}b - \Phi_2 e_{21}b \otimes a e_{12}\Phi_2^{-1}) \\
  &-\frac12 (\Phi_2 \otimes a e_{12}\Phi_2^{-1}e_{21}b - \Phi_2^{-1} \otimes a e_{12}\Phi_2 e_{21}b) \\
  &-\frac12 ( \Phi_2 e_{21}b \Phi_2^{-1} \otimes a e_{12} - \Phi_2 a e_{12} \Phi_2^{-1} \otimes e_{21}b )\,,\\ 
  \dgalfusD{\psi(c),\psi(d)} =&+\frac12 (\Phi_2 \otimes a e_{12}\Phi_2^{-1}e_{21}b - \Phi_2^{-1} \otimes a e_{12}\Phi_2 e_{21}b) \\
  &+\frac12 ( \Phi_2 e_{21}b \Phi_2^{-1} \otimes a e_{12} - \Phi_2 a e_{12} \Phi_2^{-1} \otimes e_{21}b )\,.
 \end{aligned}
\end{equation}

\textit{Case 6.4 : $d=e_{12}be_{21}$, $b\in e_2 A e_2$.} 
\begin{equation}
 \begin{aligned}
  \dgalindU{c,d}=&e_{12}\dgal{a,b}'e_1 \otimes e_{12}\dgal{a,b}''e_{21},\\ 
  \dgalfusU{c,d}=&\frac12 (e_1 \otimes e_{12}a e_{12} b e_{21} - e_{12}be_{21}\otimes e_{12}a) , \\ 
  \dgalindD{\psi(c),\psi(d)}=& \dgal{a,b}\ast e_{12}\Phi_2^{-1} +\frac12(e_2 \otimes a e_{12}\Phi_2^{-1}b - b \otimes a e_{12}\Phi_2^{-1})  \\
  &-\frac12 (b \Phi_2^{-1} \otimes a e_{12} - \Phi_2^{-1} \otimes a e_{12}b)\,,\\ 
  \dgalfusD{\psi(c),\psi(d)} =&+\frac12 (b \Phi_2^{-1} \otimes a e_{12} - \Phi_2^{-1} \otimes a e_{12}b)\,.
 \end{aligned}
\end{equation}

\subsection{} We consider the first specialisation of a generator of third type, that is $c=ae_{21}$ for $a\in \he A e_2$.

\textit{Case 7.1 : $d=be_{21}$, $b\in \he A e_2$.} 
\begin{equation}
 \begin{aligned}
  &\dgalindU{c,d}=\he\dgal{a,b}'e_{21} \otimes \he\dgal{a,b}''e_{21},\quad   \dgalfusU{c,d}=0 , \\ 
  &\dgalindD{\psi(c),\psi(d)}= \dgal{a,b}\,, \quad   \dgalfusD{\psi(c),\psi(d)} =0\,.
 \end{aligned}
\end{equation}

\textit{Case 7.2 : $d=be_{21}$, $b\in e_1 A e_2$.} 
\begin{equation}
 \begin{aligned}
  \dgalindU{c,d}=&e_1\dgal{a,b}'e_{21} \otimes \he\dgal{a,b}''e_{21},\quad   \dgalfusU{c,d}= -\frac12 e_1 \otimes a e_{21}be_{21}, \\ 
  \dgalindD{\psi(c),\psi(d)}=& \Phi_2e_{21}\dgal{a,b} -\frac12(e_2 \otimes a \Phi_2 e_{21}b + \Phi_2 \otimes a e_{21}b)\,,\\ 
  \dgalfusD{\psi(c),\psi(d)} =&+\frac12 \Phi_2 \otimes a e_{21}b\,.
 \end{aligned}
\end{equation}

\textit{Case 7.3 : $d=e_{12}be_{21}$, $b\in e_2 A e_2$.} 

\begin{equation}
 \begin{aligned}
  &\dgalindU{c,d}=e_{12}\dgal{a,b}'e_{21} \otimes \he\dgal{a,b}''e_{21},\quad   \dgalfusU{c,d}= 0, \\ 
  &\dgalindD{\psi(c),\psi(d)}= \dgal{a,b} , \quad   \dgalfusD{\psi(c),\psi(d)} =0\,.
 \end{aligned}
\end{equation}

\subsection{} We consider the second specialisation of a generator of third type, that is $c=ae_{21}$ for $a\in e_1 A e_2$. 

\textit{Case 8.1 : $d=be_{21}$, $b\in e_1 A e_2$.} 
\begin{equation}
 \begin{aligned}
  \dgalindU{c,d}=&e_1\dgal{a,b}'e_{21} \otimes e_1\dgal{a,b}''e_{21},\\
  \dgalfusU{c,d}=&\frac12 (b e_{21}ae_{21}\otimes e_1 -e_1 \otimes a e_{21}be_{21}), \\ 
  \dgalindD{\psi(c),\psi(d)}=& \Phi_2e_{21} \ast \Phi_2e_{21}\dgal{a,b} \\
  &+\frac12(\Phi_2 e_{21}b \Phi_2 e_{21}a \otimes e_2 - e_2 \otimes \Phi_2 e_{21}a \Phi_2 e_{21}b)\\
  &-\frac12 (\Phi_2 \otimes \Phi_2 e_{21}ae_{21}b - \Phi_2 e_{21}b e_{21}a \otimes \Phi_2) \\
  &-\frac12(e_{21}a \otimes \Phi_2^{2}e_{21}b - \Phi_2^2 e_{21}a \otimes e_{21}b)\,,\\ 
  \dgalfusD{\psi(c),\psi(d)} =&+\frac12 (\Phi_2 \otimes \Phi_2 e_{21}ae_{21}b - \Phi_2 e_{21}b e_{21}a \otimes \Phi_2) \\
  &+\frac12(e_{21}a \otimes \Phi_2^{2}e_{21}b - \Phi_2^2 e_{21}a \otimes e_{21}b)\,.
 \end{aligned}
\end{equation}

\textit{Case 8.2 : $d=e_{12}be_{21}$, $b\in e_2 A e_2$.} 

\begin{equation}
 \begin{aligned}
  \dgalindU{c,d}=& e_{12}\dgal{a,b}'e_{21} \otimes e_1\dgal{a,b}''e_{21},\\
  \dgalfusU{c,d}=& \frac12(e_{12}be_{21}ae_{21} \otimes e_1 - ae_{21} \otimes e_{12}be_{21}), \\ 
  \dgalindD{\psi(c),\psi(d)}=&\Phi_2e_{21} \ast \dgal{a,b} +\frac12(b \Phi_2 e_{21}a \otimes e_2 - \Phi_2 e_{21}a \otimes b )\\
  &-\frac12 ( e_{21}a \otimes \Phi_2 b - b e_{21}a \otimes \Phi_2), \\
  \dgalfusD{\psi(c),\psi(d)}=& +\frac12 ( e_{21}a \otimes \Phi_2 b - b e_{21}a \otimes \Phi_2)\,.
 \end{aligned}
\end{equation}

\subsection{} We consider  a generator of fourth type, that is $c=e_{12}ae_{21}$ for $a\in e_2 A e_2$. 

\textit{Case 9.1 : $d=e_{12}be_{21}$, $b\in e_2 A e_2$.} 
\begin{equation}
 \begin{aligned}
  &\dgalindU{c,d}= e_{12}\dgal{a,b}'e_{21} \otimes e_{12}\dgal{a,b}''e_{21},\quad   \dgalfusU{c,d}=0, \\ 
  &\dgalindD{\psi(c),\psi(d)}=\dgal{a,b}, \quad   \dgalfusD{\psi(c),\psi(d)}=0\,.
 \end{aligned}
\end{equation}


\section{Identities for the proof of Lemma \ref{L:IsoFusqHam1}} \label{App:LIsoFusqHam1}

In this appendix, we show that \eqref{Eq:IsoqHam3} holds. It suffices to verify this identity on a choice of generators of $A_1$. We will consider nine cases, which are given below with their images in $A_2$ under $\psi$. For $\he=1-e_2-e_3$, we have 
 \begin{equation*}
 \begin{aligned}
  &\psi(t)=t ,\,\,\text{if } t \in \he A \he; 
 \end{aligned}
 \end{equation*}
on generators of first type  \eqref{type2} of $A_1$, 
 \begin{equation*}
  \psi(e_{12}u)=e_{12}u,\,\,\text{if } u \in e_2 A \he; \quad 
  \psi(e_{12}e_{23}u)=e_{13}u ,\,\,\text{if } u \in e_3 A \he; 
 \end{equation*} 
on generators of second type  \eqref{type2} of $A_1$,  
 \begin{equation*}
  \psi(v e_{21})=v e_{21},\,\,\text{if } v \in \he A e_2; \quad 
  \psi(ve_{32}e_{21})=ve_{31} ,\,\,\text{if } v \in \he A e_3; 
 \end{equation*} 
on generators of third type  \eqref{type3} of $A_1$,   
  \begin{equation*}
  \begin{aligned}
   &\psi(e_{12} w e_{21})= e_{12} w e_{21},\,\,\text{if } w \in e_2 A e_2; \quad 
  \psi(e_{12} w e_{32} e_{21})=e_{12}we_{31},\,\,\text{if } w \in e_2 A e_3; \\
  & \psi(e_{12}e_{23} w  e_{21})=e_{13}we_{21},\,\,\text{if } w \in e_3 A e_2; \quad 
   \psi(e_{12}e_{23} w e_{32} e_{21})=e_{13}we_{31},\,\,\text{if } w \in e_3 A e_3; 
  \end{aligned}
 \end{equation*}
 on generators of fourth type \eqref{type4} of $A_1$. 

Thanks to the cyclic antisymmetry, we have $45$ cases where we have to check \eqref{Eq:IsoqHam3}. If both $c,d \in A_1$ are of the form $t$ for $t \in \he A \he$, $e_{12}u$ for $u \in e_2 A \he$, $v e_{21}$ for $v \in \he A e_2$, or 
$e_{12} w e_{21}$ for $w \in e_2 A e_2$, it is easy to derive that 
\begin{equation} 
\dgal{\psi(c),\psi(d)}_{2,fus}^{2\to 1}=(\psi\otimes \psi)\dgal{c,d}_{1,fus}^{2\to 1}\,, \quad 
\dgal{\psi(c),\psi(d)}_{2,fus}^{3\to 1}=0,\,\, (\psi\otimes \psi)\dgal{c,d}_{1,fus}^{3\to 2}=0,
  \end{equation}
so that \eqref{Eq:IsoqHam3} is satisfied. So we are left with $35$ cases to check. We will compute one case explicitly, and in the remaining cases we will only collect the four double brackets involved in \eqref{Eq:IsoqHam3} for the reader's convenience. 

\subsection{} We consider that $c$ is a generator of first type, that is $c=a$ for $a\in \he A \he$. 

\textit{Case 1.1 : $d=e_{12}be_{32}e_{21}$, $b \in e_2 A e_3$.}

By definition, the double bracket $\dgal{c,d}_{1,fus}^{3\to 2}$ is obtained from the fusion bracket in $A^f_{e_3\to e_2}$ by inducing it in $A_1$, so that we can write 
\begin{equation*}
 \dgal{c,d}_{1,fus}^{3\to 2}=e_{12} \dgal{a,be_{32}}_{fus}^{3\to 2} e_{21}\,,
\end{equation*}
with the double bracket on the right-hand side given by the fusion bracket from Appendix \ref{App:Dbr} with $j=3$, $i=2$. In $A^f_{e_3\to e_2}$, $a$ is a generator of first type \eqref{type1} while  $be_{32}$ is a generator of third type \eqref{type3}, so that the double bracket between them is given by \eqref{tv}, and we get 
\begin{equation*}
 \dgal{c,d}_{1,fus}^{3\to 2}=\frac12 e_{12} [be_{32}a \otimes e_2 - b e_{32} \otimes a e_2] e_{21}=0\,,
\end{equation*}
  as $ae_2=0= e_2 a$. Next, we have 
  \begin{equation*}
 \dgal{c,d}_{1,fus}^{3\to 2}=\dgal{a,e_{12} (be_{32}) e_{21}}_{fus}^{2\to 1} \,,
\end{equation*}
  with the double bracket on the right-hand side given by the fusion bracket from Appendix \ref{App:Dbr} with $j=2$, $i=1$. In $A_1$, $a$ is a generator of first type \eqref{type1} while $e_{12} (be_{32}) e_{21}$ is a generator of fourth type \eqref{type4} so that we get from \eqref{tw} 
\begin{equation*}
    \begin{aligned}
      \dgal{c,d}_{1,fus}^{3\to 2}=&
 \frac12 \left( e_{12} (be_{32}) e_{21}a \otimes e_1 + e_1 \otimes ae_{12} (be_{32}) e_{21} - e_{12} (be_{32}) e_{21} \otimes a e_1 - e_1 a \otimes e_{12} (be_{32}) e_{21}\right) \,, \\
 =& \frac12 \left( dc \otimes e_1 + e_1 \otimes cd - d \otimes c e_1 - e_1 c \otimes d\right)\,.
    \end{aligned}
\end{equation*}
We now get the right-hand side of \eqref{Eq:IsoqHam3} by applying $\psi \otimes \psi$ to these two double brackets. 

Meanwhile, we have in $A_2$ that $\psi(c)=a$ and $\psi(d)=e_{12}b e_{31}$. We note that by inducing the double bracket from $A^f_{e_2\to e_1}$ to $A_2$, 
\begin{equation*} 
\dgal{\psi(c),\psi(d)}_{2,fus}^{2\to 1}=\dgal{a,e_{12}b}_{2,fus}^{2\to 1} e_{31}\,,
  \end{equation*}
  where the double on the right-hand side given by the fusion bracket from Appendix \ref{App:Dbr} with $j=2$, $i=1$. We have that $a$ is a generator of first type \eqref{type1} while $e_{12}b$ is a generator of second type \eqref{type2}, so that using \eqref{tu} 
  \begin{equation*} 
\dgal{\psi(c),\psi(d)}_{2,fus}^{2\to 1}=\frac12 (e_1 \otimes a e_{12}b - e_1 a \otimes e_{12}b) e_{31}
=\frac12 (e_1 \otimes \psi(c)\psi(d) - e_1\psi(c) \otimes \psi(d))\,.
  \end{equation*}
Finally, we have 
\begin{equation*} 
\dgal{\psi(c),\psi(d)}_{2,fus}^{3\to 1}=\dgal{a,(e_{12}b)e_{31}}_{2,fus}^{2\to 1} \,,
  \end{equation*}
where the double on the right-hand side given by the fusion bracket from Appendix \ref{App:Dbr} with $j=3$, $i=1$. We have that $a$ is a generator of first type \eqref{type1} while $(e_{12}b)e_{31}$ is a generator of third type \eqref{type3}, so that using \eqref{tv} 
  \begin{equation*} 
\dgal{\psi(c),\psi(d)}_{2,fus}^{3\to 1}=\frac12 (  e_{12}b e_{31} a \otimes e_1 -  e_{12}b e_{31} \otimes a e_1 )
=\frac12 ( \psi(d)\psi(c) \otimes e_1 - \psi(d) \otimes \psi(c)e_1)\,.
  \end{equation*}
Gathering the four expressions obtained for the double brackets, we can see that \eqref{Eq:IsoqHam3} is satisfied. 
  
\textit{Case 1.2 : $d=e_{12}e_{23}b$, $b \in e_3 A \he$.}
\begin{equation} 
\begin{aligned}
 \dgal{c,d}_{1,fus}^{3\to 2}=0\,, \quad 
 &\dgal{c,d}_{1,fus}^{2\to 1}=\frac12 (e_1 \otimes cd - e_1 c\otimes d)\,, \\
 \dgal{\psi(c),\psi(d)}_{2,fus}^{2\to 1}=0\,, \quad 
 &\dgal{\psi(c),\psi(d)}_{2,fus}^{3\to 1}=\frac12 (e_1 \otimes \psi(c)\psi(d) - e_1 \psi(c)\otimes \psi(d))\,.
\end{aligned}
  \end{equation}

\textit{Case 1.3 : $d=be_{32}e_{21}$, $b \in \he A e_3$.}
\begin{equation} 
\begin{aligned}
 \dgal{c,d}_{1,fus}^{3\to 2}=0\,, \quad 
 &\dgal{c,d}_{1,fus}^{2\to 1}=\frac12 (dc \otimes e_1 - d\otimes c e_1)\,, \\
 \dgal{\psi(c),\psi(d)}_{2,fus}^{2\to 1}=0\,, \quad 
 &\dgal{\psi(c),\psi(d)}_{2,fus}^{3\to 1}=\frac12 (\psi(d)\psi(c) \otimes e_1 - \psi(d)\otimes \psi(c)e_1)\,.
\end{aligned}
  \end{equation}

\textit{Case 1.4 : $d=e_{12}e_{23}b e_{21}$, $b \in e_3 A e_2$.}
\begin{equation} 
\begin{aligned}
 &\dgal{c,d}_{1,fus}^{3\to 2}=0\,, \quad 
 \dgal{c,d}_{1,fus}^{2\to 1}=\frac12 (dc \otimes e_1 + e_1 \otimes cd - d\otimes c e_1 - e_1 c \otimes d)\,, \\
 &\dgal{\psi(c),\psi(d)}_{2,fus}^{2\to 1}=\frac12 (\psi(d)\psi(c) \otimes e_1 - \psi(d)\otimes \psi(c)e_1)\,, \\
 &\dgal{\psi(c),\psi(d)}_{2,fus}^{3\to 1}=\frac12 (e_1 \otimes \psi(c)\psi(d) - e_1 \psi(c)\otimes \psi(d))\,.
\end{aligned}
  \end{equation}

\textit{Case 1.5 : $d=e_{12}e_{23}be_{32}e_{21}$, $b \in e_3 A e_3$.}
\begin{equation} 
\begin{aligned}
 &\dgal{c,d}_{1,fus}^{3\to 2}=0\,, \quad 
 \dgal{c,d}_{1,fus}^{2\to 1}=\frac12 (dc \otimes e_1 + e_1 \otimes cd - d\otimes c e_1 - e_1 c \otimes d)\,, \quad 
 \dgal{\psi(c),\psi(d)}_{2,fus}^{2\to 1}=0\,, \\
 &\dgal{\psi(c),\psi(d)}_{2,fus}^{3\to 1}=\frac12 \big(\psi(d)\psi(c) \otimes e_1  + e_1 \otimes \psi(c)\psi(d) - \psi(d)\otimes \psi(c)e_1 - e_1 \psi(c)\otimes \psi(d)\big)\,.
\end{aligned}
  \end{equation}
  
\subsection{} We consider that $c$ is a generator of second type of the form $c=e_{12}a$ for $a\in e_2 A \he$. 

\textit{Case 2.1 : $d=e_{12}e_{23}b$, $b \in e_3 A \he$.}  
  \begin{equation} 
\begin{aligned}
 &\dgal{c,d}_{1,fus}^{3\to 2}=-\frac12 c\otimes d\,, \quad 
 \dgal{c,d}_{1,fus}^{2\to 1}=\frac12 ( e_1 \otimes cd - dc \otimes e_1 ) \,, \\
 &\dgal{\psi(c),\psi(d)}_{2,fus}^{2\to 1}=-\frac12 \psi(d)\psi(c) \otimes e_1\,, \quad
 \dgal{\psi(c),\psi(d)}_{2,fus}^{3\to 1}=\frac12 (e_1 \otimes \psi(c)\psi(d) -  \psi(c)\otimes \psi(d))\,.
\end{aligned}
  \end{equation}
  
\textit{Case 2.2 : $d=b e_{32}e_{21}$, $b \in \he A e_3$.}  
  \begin{equation} 
\begin{aligned}
 &\dgal{c,d}_{1,fus}^{3\to 2}=\frac12 dc\otimes e_1\,, \quad 
 \dgal{c,d}_{1,fus}^{2\to 1}=\frac12 ( c \otimes e_1 d - d \otimes c e_1 ) \,, \\
 &\dgal{\psi(c),\psi(d)}_{2,fus}^{2\to 1}=\frac12 \psi(c) \otimes e_1 \psi(d)\,, \quad
 \dgal{\psi(c),\psi(d)}_{2,fus}^{3\to 1}=\frac12 ( \psi(d)\psi(c)\otimes e_1 -  \psi(d)\otimes \psi(c) e_1)\,.
\end{aligned}
  \end{equation}  
  
\textit{Case 2.3 : $d=e_{12}b e_{32}e_{21}$, $b \in e_2 A e_3$.}  
  \begin{equation} 
\begin{aligned}
 &\dgal{c,d}_{1,fus}^{3\to 2}=\frac12 dc\otimes e_1\,, \quad 
 \dgal{c,d}_{1,fus}^{2\to 1}=\frac12 ( e_1 \otimes cd - d \otimes c e_1 ) \,, \\
 &\dgal{\psi(c),\psi(d)}_{2,fus}^{2\to 1}=\frac12 e_1 \otimes \psi(c)  \psi(d)\,, \quad
 \dgal{\psi(c),\psi(d)}_{2,fus}^{3\to 1}=\frac12 ( \psi(d)\psi(c)\otimes e_1 -  \psi(d)\otimes \psi(c) e_1)\,.
\end{aligned}
  \end{equation}  
  
\textit{Case 2.4 : $d=e_{12}e_{23}b e_{21}$, $b \in e_3 A e_2$.}  
  \begin{equation} 
\begin{aligned}
 &\dgal{c,d}_{1,fus}^{3\to 2}=-\frac12 c\otimes d\,, \quad 
 \dgal{c,d}_{1,fus}^{2\to 1}=\frac12 ( e_1 \otimes cd - d \otimes c e_1 ) \,, \\
 &\dgal{\psi(c),\psi(d)}_{2,fus}^{2\to 1}=-\frac12  \psi(d) \otimes \psi(c)e_1 \,, \quad
 \dgal{\psi(c),\psi(d)}_{2,fus}^{3\to 1}=\frac12 ( e_1 \otimes \psi(c)\psi(d) -  \psi(c)\otimes \psi(d) )\,.
\end{aligned}
  \end{equation}  
  
\textit{Case 2.5 : $d=e_{12}e_{23}b e_{32}e_{21}$, $b \in e_3 A e_3$.}  
  \begin{equation} 
\begin{aligned}
 &\dgal{c,d}_{1,fus}^{3\to 2}=\frac12 (dc \otimes e_1 - c\otimes d)\,, \quad 
 \dgal{c,d}_{1,fus}^{2\to 1}=\frac12 ( e_1 \otimes cd - d \otimes c e_1 ) \,, \quad 
 \dgal{\psi(c),\psi(d)}_{2,fus}^{2\to 1}=0\,, \\
 &\dgal{\psi(c),\psi(d)}_{2,fus}^{3\to 1}=\frac12 \big(\psi(d)\psi(c) \otimes e_1  + e_1 \otimes \psi(c)\psi(d) - \psi(d)\otimes \psi(c)e_1 -  \psi(c)\otimes \psi(d)\big)\,.
\end{aligned}
  \end{equation}

\subsection{} We consider that $c$ is a generator of third type of the form $c=a e_{21}$ for $a\in \he A e_2$. 

\textit{Case 3.1 : $d=e_{12}e_{23}b$, $b \in e_3 A \he$.}  
  \begin{equation} 
\begin{aligned}
 &\dgal{c,d}_{1,fus}^{3\to 2}=\frac12 e_1\otimes cd\,, \quad 
 \dgal{c,d}_{1,fus}^{2\to 1}=\frac12 ( de_1 \otimes c - e_1 c \otimes d ) \,, \\
 &\dgal{\psi(c),\psi(d)}_{2,fus}^{2\to 1}=\frac12 \psi(d)e_1 \otimes \psi(c),\,\,
 \dgal{\psi(c),\psi(d)}_{2,fus}^{3\to 1}=\frac12 (e_1 \otimes \psi(c)\psi(d) -  e_1\psi(c)\otimes \psi(d)).
\end{aligned}
  \end{equation}
  
\textit{Case 3.2 : $d=be_{32}e_{21}$, $b \in \he A e_3$.}  
  \begin{equation} 
\begin{aligned}
 &\dgal{c,d}_{1,fus}^{3\to 2}=-\frac12 d\otimes c\,, \quad 
 \dgal{c,d}_{1,fus}^{2\to 1}=\frac12 ( dc \otimes e_1 - e_1 \otimes cd ) \,, \\
 &\dgal{\psi(c),\psi(d)}_{2,fus}^{2\to 1}=-\frac12 e_1 \otimes \psi(c)\psi(d)\,, \,\,
 \dgal{\psi(c),\psi(d)}_{2,fus}^{3\to 1}=\frac12 (\psi(d)\psi(c) \otimes e_1 -  \psi(d)\otimes \psi(c)).
\end{aligned}
  \end{equation}
  
\textit{Case 3.3 : $d=e_{12}be_{32}e_{21}$, $b \in e_2 A e_3$.}  
  \begin{equation} 
\begin{aligned}
 &\dgal{c,d}_{1,fus}^{3\to 2}=-\frac12 d\otimes c\,, \quad 
 \dgal{c,d}_{1,fus}^{2\to 1}=\frac12 ( dc \otimes e_1 - e_1c \otimes d ) \,, \\
 &\dgal{\psi(c),\psi(d)}_{2,fus}^{2\to 1}=-\frac12 e_1\psi(c) \otimes \psi(d)\,, \,\,
 \dgal{\psi(c),\psi(d)}_{2,fus}^{3\to 1}=\frac12 (\psi(d)\psi(c) \otimes e_1 -  \psi(d)\otimes \psi(c)).
\end{aligned}
  \end{equation}
  
  \textit{Case 3.4 : $d=e_{12}e_{23}b e_{21}$, $b \in e_3 A e_2$.}  
  \begin{equation} 
\begin{aligned}
 &\dgal{c,d}_{1,fus}^{3\to 2}=\frac12 e_1 \otimes cd\,, \quad 
 \dgal{c,d}_{1,fus}^{2\to 1}=\frac12 ( dc \otimes e_1 - e_1c \otimes d ) \,, \\
 &\dgal{\psi(c),\psi(d)}_{2,fus}^{2\to 1}=\frac12 \psi(d)\psi(c) \otimes e_1\,, \,\,
 \dgal{\psi(c),\psi(d)}_{2,fus}^{3\to 1}=\frac12 (e_1 \otimes \psi(c)\psi(d) - e_1\psi(c)\otimes \psi(d)).
\end{aligned}
  \end{equation}
  
  \textit{Case 3.5 : $d=e_{12}e_{23}b e_{32}e_{21}$, $b \in e_3 A e_3$.}  
  \begin{equation} 
\begin{aligned}
 &\dgal{c,d}_{1,fus}^{3\to 2}=\frac12 (e_1 \otimes cd-d \otimes c)\,, \quad 
 \dgal{c,d}_{1,fus}^{2\to 1}=\frac12 ( dc \otimes e_1 - e_1c \otimes d ),\,\, \dgal{\psi(c),\psi(d)}_{2,fus}^{2\to 1}=0\,, \\
 &\dgal{\psi(c),\psi(d)}_{2,fus}^{3\to 1}=\frac12 \big(\psi(d)\psi(c) \otimes e_1  + e_1 \otimes \psi(c)\psi(d) - \psi(d)\otimes \psi(c) -  e_1 \psi(c)\otimes \psi(d)\big)\,.
\end{aligned}
  \end{equation}  
  
  \subsection{} We consider that $c$ is a generator of fourth type of the form $c=e_{12} a e_{21}$ for $a\in e_2 A e_2$. 

\textit{Case 4.1 : $d=e_{12}e_{23}b$, $b \in e_3 A \he$.}  
  \begin{equation} 
\begin{aligned}
 &\dgal{c,d}_{1,fus}^{3\to 2}=\frac12 (e_1\otimes cd- c\otimes d)\,, \quad 
 \dgal{c,d}_{1,fus}^{2\to 1}=\frac12 ( de_1 \otimes c - d c \otimes e_1 ) \,, \\
 &\dgal{\psi(c),\psi(d)}_{2,fus}^{2\to 1}=\frac12 \big( \psi(d)e_1 \otimes \psi(c) - \psi(d)\psi(c)\otimes e_1 \big),\\
 &\dgal{\psi(c),\psi(d)}_{2,fus}^{3\to 1}=\frac12 \big(e_1 \otimes \psi(c)\psi(d) -  \psi(c)\otimes \psi(d)\big).
\end{aligned}
  \end{equation}
  
\textit{Case 4.2 : $d=be_{32}e_{21}$, $b \in \he A e_3$.}  
  \begin{equation} 
\begin{aligned}
 &\dgal{c,d}_{1,fus}^{3\to 2}=\frac12 (dc\otimes e_1 -d\otimes c)\,, \quad 
 \dgal{c,d}_{1,fus}^{2\to 1}=\frac12 ( c \otimes e_1d - e_1 \otimes cd ) \,, \\
 &\dgal{\psi(c),\psi(d)}_{2,fus}^{2\to 1}=\frac12 \big(\psi(c)\otimes e_1\psi(d) -e_1 \otimes \psi(c)\psi(d)\big)\,, \\
& \dgal{\psi(c),\psi(d)}_{2,fus}^{3\to 1}=\frac12 \big(\psi(d)\psi(c) \otimes e_1 -  \psi(d)\otimes \psi(c)\big)\,.
\end{aligned}
  \end{equation}
  
\textit{Case 4.3 : $d=e_{12}be_{32}e_{21}$, $b \in e_2 A e_3$.}  
  \begin{equation} 
\begin{aligned}
 &\dgal{c,d}_{1,fus}^{3\to 2}=\frac12(dc\otimes e_1 - d\otimes c)\,, \quad 
 \dgal{c,d}_{1,fus}^{2\to 1}=0 \,, \\
 &\dgal{\psi(c),\psi(d)}_{2,fus}^{2\to 1}=0\,, \,\,
 \dgal{\psi(c),\psi(d)}_{2,fus}^{3\to 1}=\frac12 (\psi(d)\psi(c) \otimes e_1 -  \psi(d)\otimes \psi(c)).
\end{aligned}
  \end{equation}  
  
  \textit{Case 4.4 : $d=e_{12}e_{23}b e_{21}$, $b \in e_3 A e_2$.}  
  \begin{equation} 
\begin{aligned}
 &\dgal{c,d}_{1,fus}^{3\to 2}=\frac12 (e_1 \otimes cd - c\otimes d)\,, \quad 
 \dgal{c,d}_{1,fus}^{2\to 1}=0 \,, \\
 &\dgal{\psi(c),\psi(d)}_{2,fus}^{2\to 1}=0\,, \,\,
 \dgal{\psi(c),\psi(d)}_{2,fus}^{3\to 1}=\frac12 \big(e_1 \otimes \psi(c)\psi(d) - \psi(c)\otimes \psi(d)\big).
\end{aligned}
  \end{equation}
  
  \textit{Case 4.5 : $d=e_{12}e_{23}b e_{32}e_{21}$, $b \in e_3 A e_3$.}  
  \begin{equation} 
\begin{aligned}
 &\dgal{c,d}_{1,fus}^{3\to 2}=\frac12 (dc\otimes e_1 + e_1 \otimes cd-d \otimes c - c\otimes d)\,, \quad 
 \dgal{c,d}_{1,fus}^{2\to 1}=0,\,\, \dgal{\psi(c),\psi(d)}_{2,fus}^{2\to 1}=0\,, \\
 &\dgal{\psi(c),\psi(d)}_{2,fus}^{3\to 1}=\frac12 \big(\psi(d)\psi(c) \otimes e_1  + e_1 \otimes \psi(c)\psi(d) - \psi(d)\otimes \psi(c) -  \psi(c)\otimes \psi(d)\big)\,.
\end{aligned}
  \end{equation}

\subsection{} We consider that $c$ is a generator of second type of the form $c=e_{12}e_{23} a$ for $a\in e_3 A \he$. 

\textit{Case 5.1 : $d=e_{12}e_{23}b$, $b \in e_3 A \he$.}  
  \begin{equation} 
\begin{aligned}
 &\dgal{c,d}_{1,fus}^{3\to 2}=0\,, \quad 
 \dgal{c,d}_{1,fus}^{2\to 1}=\frac12 ( e_1 \otimes cd - d c \otimes e_1 ) \,, \\
 &\dgal{\psi(c),\psi(d)}_{2,fus}^{2\to 1}=0\,, \,\, 
 \dgal{\psi(c),\psi(d)}_{2,fus}^{3\to 1}=\frac12 \big(e_1 \otimes \psi(c)\psi(d) -  \psi(d)\psi(c)\otimes e_1 \big).
\end{aligned}
  \end{equation}

\textit{Case 5.2 : $d=be_{32}e_{21}$, $b \in \he A e_3$.}  
  \begin{equation} 
\begin{aligned}
 &\dgal{c,d}_{1,fus}^{3\to 2}=0\,, \quad 
 \dgal{c,d}_{1,fus}^{2\to 1}=\frac12 ( c \otimes e_1d - d \otimes c e_1 ) \,, \\
 &\dgal{\psi(c),\psi(d)}_{2,fus}^{2\to 1}=0\,, \,\, 
 \dgal{\psi(c),\psi(d)}_{2,fus}^{3\to 1}=\frac12 \big(\psi(c)\otimes e_1\psi(d) -  \psi(d)\otimes \psi(c)e_1 \big)\,.
\end{aligned}
  \end{equation}

\textit{Case 5.3 : $d=e_{12}be_{32}e_{21}$, $b \in e_2 A e_3$.}  
  \begin{equation} 
\begin{aligned}
 &\dgal{c,d}_{1,fus}^{3\to 2}=\frac12 c\otimes d\,, \quad 
 \dgal{c,d}_{1,fus}^{2\to 1}=\frac12 (e_1 \otimes cd - d \otimes c e_1) \,, \\
 &\dgal{\psi(c),\psi(d)}_{2,fus}^{2\to 1}=\frac12 e_1 \otimes \psi(c)\psi(d)\,, \,\,
 \dgal{\psi(c),\psi(d)}_{2,fus}^{3\to 1}=\frac12 \big(\psi(c)\otimes \psi(d) -  \psi(d)\otimes \psi(c)e_1\big).
\end{aligned}
  \end{equation}  

  \textit{Case 5.4 : $d=e_{12}e_{23}b e_{21}$, $b \in e_3 A e_2$.}  
  \begin{equation} 
\begin{aligned}
 &\dgal{c,d}_{1,fus}^{3\to 2}=-\frac12  dc\otimes e_1\,, \quad 
 \dgal{c,d}_{1,fus}^{2\to 1}=\frac12 (e_1 \otimes cd - d \otimes c e_1) \,, \\
 &\dgal{\psi(c),\psi(d)}_{2,fus}^{2\to 1}=-\frac12 \psi(d) \otimes \psi(c)e_1\,, \,\,
 \dgal{\psi(c),\psi(d)}_{2,fus}^{3\to 1}=\frac12 \big(e_1 \otimes \psi(c)\psi(d) - \psi(d)\psi(c)\otimes e_1\big).
\end{aligned}
  \end{equation}

  \textit{Case 5.5 : $d=e_{12}e_{23}b e_{32}e_{21}$, $b \in e_3 A e_3$.}  
  \begin{equation} 
\begin{aligned}
 &\dgal{c,d}_{1,fus}^{3\to 2}=0\,, \quad 
 \dgal{c,d}_{1,fus}^{2\to 1}=\frac12 (e_1 \otimes cd-d \otimes ce_1),\\
& \dgal{\psi(c),\psi(d)}_{2,fus}^{2\to 1}=0\,, \,\,
 \dgal{\psi(c),\psi(d)}_{2,fus}^{3\to 1}=\frac12 \big(e_1 \otimes \psi(c)\psi(d) - \psi(d)\otimes \psi(c) e_1 \big)\,.
\end{aligned}
  \end{equation}

\subsection{} We consider that $c$ is a generator of third type of the form $c=ae_{32}e_{21}$ for $a\in \he A e_3$. 

\textit{Case 6.1 : $d=be_{32}e_{21}$, $b \in \he A e_3$.}  
  \begin{equation} 
\begin{aligned}
 &\dgal{c,d}_{1,fus}^{3\to 2}=0\,, \quad 
 \dgal{c,d}_{1,fus}^{2\to 1}=\frac12 ( d c \otimes e_1 - e_1 \otimes cd) \,, \\
 &\dgal{\psi(c),\psi(d)}_{2,fus}^{2\to 1}=0\,, \,\, 
 \dgal{\psi(c),\psi(d)}_{2,fus}^{3\to 1}=\frac12 \big( \psi(d)\psi(c)\otimes e_1 - e_1 \otimes \psi(c)\psi(d) \big).
\end{aligned}
  \end{equation}

\textit{Case 6.2 : $d=e_{12}be_{32}e_{21}$, $b \in e_2 A e_3$.}  
  \begin{equation} 
\begin{aligned}
 &\dgal{c,d}_{1,fus}^{3\to 2}=-\frac12 e_1\otimes cd\,, \quad 
 \dgal{c,d}_{1,fus}^{2\to 1}=\frac12 (dc \otimes e_1 - e_1c \otimes d) \,, \\
 &\dgal{\psi(c),\psi(d)}_{2,fus}^{2\to 1}=-\frac12 e_1 \psi(c) \otimes \psi(d)\,, \,\,
 \dgal{\psi(c),\psi(d)}_{2,fus}^{3\to 1}=\frac12 \big(\psi(d) \psi(c)\otimes e_1 -  e_1\otimes \psi(c)\psi(d)  \big).
\end{aligned}
  \end{equation}  

  \textit{Case 6.3 : $d=e_{12}e_{23}b e_{21}$, $b \in e_3 A e_2$.}  
  \begin{equation} 
\begin{aligned}
 &\dgal{c,d}_{1,fus}^{3\to 2}=\frac12  d\otimes c\,, \quad 
 \dgal{c,d}_{1,fus}^{2\to 1}=\frac12 (dc \otimes e_1 - e_1c \otimes d) \,, \\
 &\dgal{\psi(c),\psi(d)}_{2,fus}^{2\to 1}=\frac12 \psi(d)\psi(c) \otimes e_1\,, \,\,
 \dgal{\psi(c),\psi(d)}_{2,fus}^{3\to 1}=\frac12 \big(\psi(d) \otimes \psi(c) - e_1\psi(c)\otimes \psi(d)\big).
\end{aligned}
  \end{equation}

  \textit{Case 6.4 : $d=e_{12}e_{23}b e_{32}e_{21}$, $b \in e_3 A e_3$.}  
  \begin{equation} 
\begin{aligned}
 &\dgal{c,d}_{1,fus}^{3\to 2}=0\,, \quad 
 \dgal{c,d}_{1,fus}^{2\to 1}=\frac12 (dc \otimes e_1 - e_1c \otimes d),\\
& \dgal{\psi(c),\psi(d)}_{2,fus}^{2\to 1}=0\,, \,\,
 \dgal{\psi(c),\psi(d)}_{2,fus}^{3\to 1}=\frac12 \big(\psi(d) \psi(c) \otimes e_1 - e_1\psi(c)\otimes \psi(d)\big)\,.
\end{aligned}
  \end{equation}  

\subsection{} We consider the remaining cases. The elements $c,d$ are generators of fourth type in each of them. 

\textit{Case 7.1 : $c=e_{12}ae_{32}e_{21}$, $a \in e_2 A e_3$, and $d=e_{12}be_{32}e_{21}$, $b \in e_2 A e_3$.}  
  \begin{equation} 
\begin{aligned}
 &\dgal{c,d}_{1,fus}^{3\to 2}=\frac12 (dc \otimes e_1 - e_1 \otimes cd)\,, \quad 
 \dgal{c,d}_{1,fus}^{2\to 1}=0 \,, \\
 &\dgal{\psi(c),\psi(d)}_{2,fus}^{2\to 1}=0\,, \,\,
 \dgal{\psi(c),\psi(d)}_{2,fus}^{3\to 1}=\frac12 \big(\psi(d) \psi(c)\otimes e_1 -  e_1\otimes \psi(c)\psi(d)  \big).
\end{aligned}
  \end{equation} 

\textit{Case 7.2 : $c=e_{12}ae_{32}e_{21}$, $a \in e_2 A e_3$, and $d=e_{12}e_{23}b e_{21}$, $b \in e_3 A e_2$.}  
  \begin{equation} 
\begin{aligned}
 &\dgal{c,d}_{1,fus}^{3\to 2}=\frac12 (d \otimes c - c \otimes d)\,, \quad 
 \dgal{c,d}_{1,fus}^{2\to 1}=0 \,, \\
 &\dgal{\psi(c),\psi(d)}_{2,fus}^{2\to 1}=0\,, \,\,
 \dgal{\psi(c),\psi(d)}_{2,fus}^{3\to 1}=\frac12 \big(\psi(d) \otimes \psi(c) -  \psi(c)\otimes \psi(d)  \big).
\end{aligned}
  \end{equation} 

\textit{Case 7.3 : $c=e_{12}ae_{32}e_{21}$, $a \in e_2 A e_3$, and $d=e_{12}e_{23}b e_{32}e_{21}$, $b \in e_3 A e_3$.}  
  \begin{equation} 
\begin{aligned}
 &\dgal{c,d}_{1,fus}^{3\to 2}=\frac12 (dc \otimes e_1 - c \otimes d)\,, \quad 
 \dgal{c,d}_{1,fus}^{2\to 1}=0 \,, \\
 &\dgal{\psi(c),\psi(d)}_{2,fus}^{2\to 1}=0\,, \,\,
 \dgal{\psi(c),\psi(d)}_{2,fus}^{3\to 1}=\frac12 \big(\psi(d)\psi(c) \otimes e_1 -  \psi(c)\otimes \psi(d)  \big).
\end{aligned}
  \end{equation} 

\textit{Case 7.4 : $c=e_{12}e_{23}a e_{21}$, $a \in e_3 A e_2$, and $d=e_{12}e_{23}b e_{21}$, $b \in e_3 A e_2$.}  
  \begin{equation} 
\begin{aligned}
 &\dgal{c,d}_{1,fus}^{3\to 2}=\frac12 (e_1 \otimes cd - dc \otimes e_1)\,, \quad 
 \dgal{c,d}_{1,fus}^{2\to 1}=0 \,, \\
 &\dgal{\psi(c),\psi(d)}_{2,fus}^{2\to 1}=0\,, \,\,
 \dgal{\psi(c),\psi(d)}_{2,fus}^{3\to 1}=\frac12 \big(e_1\otimes \psi(c)\psi(d)  -  \psi(d) \psi(c)\otimes e_1  \big).
\end{aligned}
  \end{equation} 

\textit{Case 7.5 : $c=e_{12}e_{23}a e_{21}$, $a \in e_3 A e_2$, and $d=e_{12}e_{23}b e_{32}e_{21}$, $b \in e_3 A e_3$.}  
  \begin{equation} 
\begin{aligned}
 &\dgal{c,d}_{1,fus}^{3\to 2}=\frac12 (e_1 \otimes cd - d \otimes c)\,, \quad 
 \dgal{c,d}_{1,fus}^{2\to 1}=0 \,, \\
 &\dgal{\psi(c),\psi(d)}_{2,fus}^{2\to 1}=0\,, \,\,
 \dgal{\psi(c),\psi(d)}_{2,fus}^{3\to 1}=\frac12 \big(e_1\otimes \psi(c)\psi(d)  -  \psi(d) \otimes \psi(c)  \big).
\end{aligned}
  \end{equation} 

\textit{Case 7.6 : $c=e_{12}e_{23}a e_{32}e_{21}$, $a \in e_3 A e_3$, and $d=e_{12}e_{23}b e_{32}e_{21}$, $b \in e_3 A e_3$.}  
  \begin{equation} 
\begin{aligned}
 &\dgal{c,d}_{1,fus}^{3\to 2}=0\,, \quad 
 \dgal{c,d}_{1,fus}^{2\to 1}=0 \,, \quad
 \dgal{\psi(c),\psi(d)}_{2,fus}^{2\to 1}=0\,, \quad
 \dgal{\psi(c),\psi(d)}_{2,fus}^{3\to 1}=0.
\end{aligned}
  \end{equation}


\end{document}